\documentclass[11pt]{amsart}   	

\usepackage[margin = 1in]{geometry}
									
\usepackage{amsmath,amssymb, amsfonts,amsthm}
\usepackage[hide links]{hyperref}
\usepackage{mathtools}
\usepackage{tikz}
\usepackage{tikz-cd}
\usetikzlibrary{decorations.markings}
\usepackage{enumitem}
\usepackage{makecell}
\setcellgapes{4pt}

\newtheorem{thm}{Theorem}[section]
\newtheorem{lem}[thm]{Lemma}
\newtheorem{cor}[thm]{Corollary}
\newtheorem{prop}[thm]{Proposition}

\newtheorem{innercustomthm}{Theorem}
\newenvironment{customthm}[1]
{\renewcommand\theinnercustomthm{#1}\innercustomthm}{\endinnercustomthm}

\theoremstyle{definition}
\newtheorem{define}[thm]{Definition}
\newtheorem{defthm}[thm]{Definition-Theorem}
\newtheorem{ex}[thm]{Example}
\newtheorem{cex}[thm]{Counterexample}
\newtheorem{rem}[thm]{Remark}

\newcommand{\inv}{^{-1}}

\newcommand{\str}{\mathsf{s}\tau\textnormal{-}\mathsf{rigid}}
\newcommand{\stt}{\mathsf{s}\tau\textnormal{-}\mathsf{tilt}}
\newcommand{\smc}{2\textnormal{-}\mathsf{smc}}

\newcommand{\lperp}[1]{\prescript{\perp}{}{#1}}

\usepackage{soul}
\title{$\tau$-Cluster Morphism Categories and Picture Groups}
\author{Eric J. Hanson, Kiyoshi Igusa}
\address{Department of Mathematics, Brandeis University\\415 South Street, Waltham MA 02453, USA}
\email{ehanson4@brandeis.edu (E.J. Hanson, Corresponding Author)}

\date{7 June, 2021}



\thanks{\copyright Taylor \& Francis Group, LLC. Reproduced under CC-BY-NC 2.0.}

\begin{document}

\noindent Comm. Alg. 49 (2021), no. 10, 4377-4415. \href{https://doi.org/10.1080/00927872.2021.1921184}{DOI:10.1080/00927872.2021.1921184}.

\bigskip

\maketitle

\begin{abstract}
	$\tau$-cluster morphism categories, introduced by Buan and Marsh, are a generalization of cluster morphism categories (defined by Igusa and Todorov). We show the classifying space of such a category is a cube complex, generalizing results of Igusa and Todorov and Igusa. Furthermore, the fundamental group of this space is the picture group of the algebra, first defined by Igusa, Todorov, and Weyman. Finally, we show that for Nakayama algebras, this space is a $K(\pi,1)$. The key step is a combinatorial proof that, for Nakayama algebras, 2-simple minded collections are characterized by pairwise compatibility conditions, a fact not true in general.
\end{abstract}

\tableofcontents

\section{Introduction}

In \cite{igusa_picture}, the second author, Todorov, and Weyman associate to every hereditary algebra of finite type a finitely presented group, called the \emph{picture group}. This group is intimately related with the so-called ``semi-invariant picture'' of the underlying algebra. These (semi-invariant) pictures have in turn appeared in the literature under many names, including scattering diagrams, stability scattering diagrams, cluster mutation fans, and wall-and-chamber structures. Recently, they have featured prominently in work on cluster algebras (e.g. \cite{gross_canonical,cheung_greedy,reading_combinatorial}), representation theory (e.g. \cite{brustle_wall,treffinger_sign}), and invariant theory (e.g. \cite{bridgeland_scattering}).

In this paper, we consider ``semi-invariant pictures'' or ``scattering diagrams'' which are finite. This means that Euclidean space $\mathbb{R}^n$ is divided into finitely many regions separated by ``walls''. The walls are labeled with certain indecomposable modules (called ``bricks''), which may not be distinct. The picture group is generated by these brick labels modulo sufficiently many relations so that for any oriented cycle in $\mathbb{R}^n$ which is transverse to the walls, the product of the labels of walls traversed by this cycle is trivial. For a representation finite hereditary algebra, bricks correspond precisely to positive roots of the underlying Dynkin diagram, and the picture group then has the following presentation (see \cite[Def.~1.1.5, Def.~1.1.8, Thm.~2.2.1]{igusa_picture} or Definition-Theorem~\ref{defthm:hereditarypicture} in the present paper):
\begin{itemize}
	\item There is one generator for each positive root of the underlying Dynkin diagram.
	\item If $\alpha$ and $\beta$ are positive roots and there are no nonzero morphisms between the corresponding modules, then there is a relation
	$$\alpha\cdot\beta = \prod(r\alpha+s\beta),$$
		where the product is over all positive roots of the form $(r\alpha + s\beta)$ in increasing order of the ratio $r/s$ (from $0/1$ to $1/0$).
\end{itemize}

In types $A,D$, and $E$, the relations described above are all of the form $\alpha\cdot \beta = \beta\cdot\alpha$ or $\alpha\cdot\beta = \beta\cdot(\alpha+\beta)\cdot\alpha$. These are called ``square'' and ``pentagon'' relations, respectively. By an observation of Keller, this is related to ``quantum dilogarithms'', which are power series $E(\beta)$ associated to each positive root $\beta$ of a Dynkin quiver. These quantum dilogarithms satisfy square and pentagon relations as well, and it has been conjectured that all other relations among quantum dilogarithms follow from these. This is, by definition, equivalent to saying that quantum dilogarithms give a faithful representation of the picture group.

Another application of picture groups is in the study of ``maximal green sequences''. It is proved in \cite{igusa_pictureMGS} that there is a bijection between the set of maximal green sequences for a Dynkin quiver $Q$ and positive expressions for the ``Coxeter element'' in the picture group of $Q$. For more about picture groups see \cite{igusa_picture,igusa_signed}.

One topic of interest is the computation of the cohomology ring of a picture group. This is difficult to compute directly, so instead, the second author, Todorov, and Weyman associate a finite CW complex, called the \emph{picture space}, to every hereditary algebra in \cite{igusa_picture}. In \cite{igusa_signed}, the second author and Todorov define the \emph{cluster morphism category} of a hereditary algebra and show that in finite type, the classifying space of this category is homeomorphic to the picture space. They then show that the picture space has cohomology isomorphic to that of the picture group (this is also shown in a special case in \cite{igusa_category}). This fact is used in \cite{igusa_picture} to compute the cohomology of any picture group of type $A_n$ by computing the (cellular) cohomology of the corresponding picture space.

In this paper, we extend the construction of the picture group and the associated CW complex to the case of \emph{$\tau$-tilting finite} algebras using $\tau$-\emph{tilting theory} (see \cite{adachi_tilting,demonet_tilting}) and $\tau$-\emph{cluster morphism categories} (see \cite{buan_category,buan_exceptional}). $\tau$-tilting theory was introduced by Adachi, Iyama, and Reiten in \cite{adachi_tilting} as an extension of classical tilting theory which recovers the combinatorics from the hereditary case in general. It has been an area of very active current research since its introduction since $\tau$-tilting theory applies to all finite dimensional algebras. For example,  Br\"ustle, Smith, and Treffinger use $\tau$-tilting theory to describe the wall and chamber structure of finite dimensional algebras in \cite{brustle_wall}, Plamondon shows that for gentle algebras, $\tau$-tilting finiteness and representation finiteness are equivalent in \cite{plamondon_tilting}, and Schroll and Treffinger use $\tau$-tilting theory to expand on the first Brauer-Thrall conjecture in \cite{schroll_tau}. $\tau$-cluster morphism categories were introduced by Buan and Marsh in \cite{buan_category} as a generalization of the cluster morphism categories from \cite{igusa_signed}.

Our final main result is to recover the isomorphism of cohomologies (of the picture group with the picture space) in the case that our algebra is Nakayama. Nakayama algebras form a well-studied class of representation-finite, but not necessarily hereditary, algebras. As such, they remain an important class to study. For example, in \cite{madsen_bounds}, Madsen and Marczinzik give bounds on the global and finitistic dimensions of Nakayama algebras and in \cite{sen_phi}, \c{S}en shows that the $\varphi$-dimension (as defined in \cite{igusa_finitistic}) of a Nakayama algebra of infinite global dimension is always even.

Our proof of the isomorphism of cohomologies in the Nakayama case relies on showing that the associated CW complex has the structure of a \emph{locally $\mathrm{CAT}(0)$ cube complex}\footnote{The term non-positively curved, or NPC, is also used in place of locally $\mathrm{CAT}(0)$ in the literature.}. As a followup to this paper, the authors show in \cite{hanson_pairwise} that this is not true for arbitrary $\tau$-tilting finite algebras. It remains an open question whether there is an isomorphism of cohomologies between the picture group and the picture space when the latter is not a locally $\mathrm{CAT}(0)$ cube complex.

As an application of this work, forthcoming work to appear by the second author, Orr, Todorov, and Weyman use the existence of the isomorphism established in this paper to compute the cohomology rings of the picture groups of cyclic cluster-tilted algebras of type $D_n$, since these are examples of Nakayama algebras.
A surprising consequence of this calculation is that hereditary algebras of type $D_n$ do not have the same picture group cohomology as their cluster-tilted counterparts.

\subsection{Notation and Terminology}

Let $\Lambda$ be a finite dimensional, basic algebra over an arbitrary field $K$. We denote by $\mathsf{mod}\Lambda$ the category of finitely generated (right) $\Lambda$-modules and by $\mathsf{proj}\Lambda$ the subcategory of projective modules. Throughout this paper, all subcategories will be full and closed under isomorphisms. For $M \in \mathsf{mod}\Lambda$, we denote by $\mathsf{add} M$ (resp. $\mathsf{Fac} M, \mathsf{Sub} M$) the subcategory of direct summands (resp. factors, subobjects) of finite direct sums of $M$. Moreover, $\mathsf{Filt} M$ refers to the subcategory of modules admitting a filtration by the direct summands of $M$. Given a subcategory $\mathcal{M} \subseteq \mathsf{mod}\Lambda$, we define $\mathsf{add}\mathcal{M}, \mathsf{Fac}\mathcal{M},\mathsf{Sub}\mathcal{M}$, and $\mathsf{Filt}\mathcal{M}$ analogously.

We denote by $\mathcal{D}^b(\mathsf{mod}\Lambda)$ the bounded derived category of $\mathsf{mod}\Lambda$. The symbol $(-)[1]$ will denote the shift functor in all triangulated categories. We identify $\mathsf{mod}\Lambda$ with the subcategory of $\mathcal{D}^b(\mathsf{mod}\Lambda)$ consisting of stalk complexes centered at zero. For $U \in \mathsf{mod}\Lambda$ or $U \in \mathcal{D}^b(\mathsf{mod}\Lambda)$, we denote by $\mathsf{rk}(U)$ the number of (isoclasses of) indecomposable direct summands of $U$.

For $\mathcal{M}$ an arbitrary module category, $\tau_{\mathcal{M}}$ (or simply $\tau$ if there is no confusion) will be the Auslander-Reiten translate in $\mathcal{M}$. For an object $X$ in a category $\mathcal{C}$, we define the \emph{left-perpendicular category} of $X$ as $\lperp{X}:= \{Y\in\mathcal{C}|\mathrm{Hom}(Y,X)\} = 0$. We define the \emph{right-perpendicular category}, $X^\perp$, dually. If $X \in Y^{\perp}\cap\lperp{Y}$ we say that the objects $X$ and $Y$ are {\it Hom orthogonal}. If in addition $Y[1]\in X^\perp$ and $X[1] \in Y^\perp$, we say $X$ and $Y$ are {\it Hom-Ext orthogonal}. We denote by $\mathsf{ind}(\mathcal{C})$ the category of indecomposable objects of $\mathcal{C}$.

\subsection{Organization and Main Results}

The contents of this paper are as follows. In Section \ref{sec:background}, we overview the results we will use from $\tau$-tilting theory as well as the construction of the $\tau$-cluster morphism category given in \cite{buan_category}.

In Section \ref{sec:cubical}, we recall the definition of a \emph{cubical category} from \cite{igusa_category}. We then prove our first main theorem:
\begin{customthm}{A}[Theorem \ref{thm:cubical}]
Let $\Lambda$ be $\tau$-tilting finite. Then the $\tau$-cluster morphism category of $\Lambda$ is cubical.
\end{customthm}
\noindent This generalizes a known result from \cite{igusa_signed} in the case that $\Lambda$ is hereditary.

In Section \ref{sec:nakayama}, we construct a combinatorial model for the 2-simple minded collections for Nakayama algebras. We use this model to prove that these 2-simple minded collections are given by pairwise compatibility conditions.

Section \ref{sec:pi1} is devoted to studying the fundamental groups of the classifying spaces of $\tau$-cluster morphism categories. We begin by expanding the definition of the \emph{picture group} of a representation finite hereditary algebra given in \cite{igusa_picture} to $\tau$-tilting finite algebras. We then prove our second main theorem.
\begin{customthm}{B}[Theorem \ref{thm:pi1}]
Let $\Lambda$ be $\tau$-tilting finite. Then the fundamental group of the classifying space of the $\tau$-cluster morphism category of $\Lambda$ is isomorphic to the picture group of $\Lambda$.
\end{customthm}
This again generalizes a known result from \cite{igusa_signed} in the case that $\Lambda$ is hereditary.
We end this section by using picture groups to construct faithful group functors for Nakayama algebras. This, together with the results of Section \ref{sec:nakayama}, allows us to conclude our third main theorem.
\begin{customthm}{C}[Theorem \ref{thm:cat0}]
Let $\Lambda$ be a Nakayama algebra. Then the classifying space of the $\tau$-cluster morphism category is a locally $\mathrm{CAT}(0)$ cube complex and hence has cohomology isomorphic to that of the picture group of $\Lambda$. In particular, the cohomological dimension of the picture group is bounded by $\mathsf{rk}(\Lambda)$, the rank of~$\Lambda$.
\end{customthm}
This generalizes a known result from \cite{igusa_category} in the case that $\Lambda$ is a path algebra of type $A_n$. It is also shown in \cite{igusa_signed} that if $\Lambda$ is hereditary, then the classifying space of the ($\tau$-)cluster morphism category has cohomology isomorphic to that of the picture group of $\Lambda$. It remains an open question whether the classifying space of the $\tau$-cluster morphism category has cohomology isomorphic to that of the picture group of $\Lambda$ for arbitrary $\tau$-tilting finite algebras.

\section{Background}
\label{sec:background}

\subsection{$\tau$-Tilting Theory}
\label{sec:tilting}

$\tau$-tilting theory was introduced by Adachi, Iyama, and Reiten \cite{adachi_tilting} as an extension of classical tilting theory which recovers the combinatorics from the hereditary case in general. Following the notation of \cite{buan_category}, we recall that a basic object $M\sqcup P[1] \in \mathcal{D}^b(\mathsf{mod}\Lambda)$ is called a \emph{support $\tau$-rigid pair} for $\Lambda$ if $M\in\mathsf{mod}\Lambda$, $P\in\mathsf{proj}\Lambda$, and we have
	$$\mathrm{Hom}(M,\tau M) = 0 = \mathrm{Hom}(P,M).$$
If in addition
	$$\mathsf{rk}(M) + \mathsf{rk}(P) = \mathsf{rk}(\Lambda)$$
then $M\sqcup P[1]$ is called a \emph{support $\tau$-tilting pair}. We denote by $\str\Lambda$ and $\stt\Lambda$ the sets of (isoclasses of) support $\tau$-rigid and support $\tau$-tilting pairs for $\Lambda$. We now recall several facts about completing $\tau$-rigid pairs.

\begin{thm}\label{thm:bongartz} Let $M\sqcup P[1]$ be a support $\tau$-rigid pair for $\Lambda$.
\begin{enumerate} [label=\upshape(\alph*)]
	\item \cite[Thm. 2.10]{adachi_tilting} Let $\mathsf{proj}(\lperp{\tau M}\cap P^\perp)$ denote the subcategory of modules which are Ext-projective in $\lperp{\tau M}\cap P^\perp$. Then there exists a unique $B \in \mathsf{mod}\Lambda$ such that $M\sqcup B\sqcup P[1]$ is a support $\tau$-tilting pair which satisfies $\mathsf{add}(B\sqcup M) = \mathsf{proj}(\lperp{\tau M}\cap P^\perp)$ and $\lperp{\tau M}\cap P^\perp = \lperp{\tau(M\sqcup B)}\cap P^\perp = \mathsf{Fac} (M\sqcup B)$. The module $B$ is called the \emph{Bongartz complement} of $M$ in $P^\perp$.
	\item \cite[Thm. 3.8]{adachi_tilting} If $M\sqcup P[1]$ is almost complete, that is $\mathsf{rk}(M) + \mathsf{rk}(P) = \mathsf{rk}(\Lambda) - 1$, then there exist exactly two support $\tau$-tilting pairs containing $M\sqcup P[1]$ as a direct summand. Moreover, these two pairs have the form $M\sqcup B \sqcup P[1]$ and $M\sqcup C\sqcup P[1]$ where $C \in \mathsf{Fac} M\cup\mathsf{proj}\Lambda[1]$ and $B \in \mathsf{mod}\Lambda$ is as described in (a).
\end{enumerate}
\end{thm}

We say that $M\sqcup C\sqcup P[1]$ is the \emph{left mutation} of $M\sqcup B\sqcup P[1]$ at $B$. Likewise, $M\sqcup B \sqcup P[1]$ is the \emph{right mutation} of $M\sqcup C\sqcup P[1]$ at $C$. This notion of mutation gives the set $\stt\Lambda$ the structure of a poset by taking the transitive closure of the relation $U < V$ if $U$ is a left mutation of $V$.

We now recall that a subcategory $\mathcal{T}\subseteq \mathsf{mod}\Lambda$ is called a \emph{torsion class} if it is closed under extensions and factors. We denote by $\mathsf{tors}\Lambda$ the poset of torsion classes of $\Lambda$ under inclusion. We then have the following.

\begin{thm}\
\smallskip
\begin{enumerate}[label=\upshape(\alph*)]
	\item \cite[Prop. 1.3]{iyama_lattice} The poset $\mathsf{tors}\Lambda$ is a lattice. That is: 
	\begin{itemize}
		\item For every pair of torsion classes $\mathcal{T}, \mathcal{T}' \in \mathsf{tors}\Lambda$, there exists a unique torsion class $\mathcal{T}\vee\mathcal{T}'$ such that $\mathcal{T}\vee\mathcal{T}' \subseteq \mathcal{T}''$ whenever $\mathcal{T},\mathcal{T}'\subseteq\mathcal{T}''$. The torsion class $\mathcal{T}\vee\mathcal{T}'$ is called the \emph{join} of $\mathcal{T}$ and $\mathcal{T}'$.
		\item For every pair or torsion classes $\mathcal{T}, \mathcal{T}' \in \mathsf{tors}\Lambda$, there exists a unique torsion class $\mathcal{T}\wedge\mathcal{T}'$ such that $\mathcal{T}''\subseteq \mathcal{T}\wedge\mathcal{T}'$ whenever $\mathcal{T}''\subseteq\mathcal{T},\mathcal{T}'$. The torsion class $\mathcal{T}\wedge\mathcal{T}'$ is called the \emph{meet} of $\mathcal{T}$ and $\mathcal{T}'$.
	\end{itemize}
	\item \cite[Thm. 2.7, Cor. 2.34]{adachi_tilting} There is a monomorphism of posets $\stt\Lambda\rightarrow\mathsf{tors}\Lambda$ given by $M\sqcup P[1]\mapsto \mathsf{Fac} M$.
	\item \cite[Thm. 3.8]{demonet_tilting} The monomorphism in (b) is an isomorphism of lattices if and only if $\stt\Lambda$ is a finite set. In this case, the algebra $\Lambda$ is called \emph{$\tau$-tilting finite}.
	\end{enumerate}
\end{thm}

From now on, we assume our algebra $\Lambda$ is $\tau$-tilting finite.

\subsection{Semibricks and 2-Simple Minded Collections}
\label{sec:semibricks}

We recall that an (indecomposable) object $S \in \mathsf{mod}\Lambda$ (or more generally $\mathcal{D}^b(\mathsf{mod}\Lambda))$ is called a \emph{brick} if $\mathrm{End}(S)$ is a division algebra. A set $\mathcal{S}$ of objects in $\mathsf{mod}\Lambda$ is called a \emph{semibrick} if it consists of pairwise Hom-orthogonal bricks. We also remark that it is customary to use the term semibrick to refer to both the set $\mathcal{S}$ and the object $\displaystyle \bigsqcup_{S\in \mathcal{S}}S \in \mathcal{D}^b(\mathsf{mod}\Lambda)$. We denote by $\mathsf{brick}\Lambda$ (resp. $\mathsf{sbrick}\Lambda$) the set of isoclasses of bricks (resp. semibricks) in $\mathsf{mod}\Lambda$. Semibricks are closely related to 2-simple minded collections, defined as follows.

\begin{define}
\label{def:smc}
A finite collection $\mathcal{X} \subset \mathcal{D}^b(\mathsf{mod}\Lambda)$ is called \emph{simple minded} if
\begin{enumerate}[label=\upshape(\alph*)]
	\item Each $X\in\mathcal{X}$ is a brick.
	\item $\mathrm{Hom}(X_i,X_j) = 0$ for all $X_i\neq X_j \in \mathcal{X}$.
	\item $\mathrm{Hom}(X_i,X_j[m]) = 0$ for all $X_i,X_j \in \mathcal{X}$ and $m < 0$.
	\item $\mathsf{thick}(\mathcal{X}) = \mathcal{D}^b(\mathsf{mod}\Lambda)$, where $\mathsf{thick}(\mathcal{X})$ is the smallest triangulated subcategory of $\mathcal{D}^b(\mathsf{mod}\Lambda)$ containing $\mathcal{X}$ which is closed under direct summands.
	\end{enumerate}
If in addition $H^i(X) = 0$ for all $X \in \mathcal{X}$ and $i\neq -1,0$, then $\mathcal{X}$ is called a \emph{2-simple minded collection}. We denote by $\smc\Lambda$ the set of 2-simple minded collections in $\mathcal{D}^b(\mathsf{mod}\Lambda)$.
\end{define}

\begin{ex}
Let $\Lambda = K(2\leftarrow1)$. The category $\mathsf{mod}\Lambda$ contains three indecomposable objects, all of which are bricks: $S_1 = (0\leftarrow K)$, $P_1 = (K\xleftarrow{1}K)$, and $P_2 = (K\leftarrow 0)$. There are then five 2-simple minded collections for $\Lambda$:
$$\smc\Lambda = \{S_1\sqcup P_2, S_1\sqcup P_1[1], P_1\sqcup P_2[1],P_2\sqcup S_1[1], S_1[1]\sqcup P_2[1]\}.$$
\end{ex}

We recall the following facts about 2-simple minded collections.

\begin{prop}
\label{prop:sbricks} 
Let $\mathcal{X}$ be a 2-simple minded collection in $\mathcal{D}^b(\mathsf{mod}\Lambda)$. Then
\begin{enumerate}[label=\upshape(\alph*)]
	\item \cite[Cor. 5.5]{koenig_silting} As a set, $|\mathcal{X}| = \mathsf{rk}(\Lambda)$. Equivalently, as an element of the bounded derived category, $\mathsf{rk}(\mathcal{X}) = \mathsf{rk}(\Lambda)$.
	\item \cite[Rmk. 4.11]{brustle_ordered} Up to isomorphism, $\mathcal{X} = \mathcal{S}_p\sqcup \mathcal{S}_n[1]$ with $\mathcal{S}_p, \mathcal{S}_n \in \mathsf{sbrick}\Lambda$.
\end{enumerate}
\end{prop}

We now recall the notion of mutation for 2-simple minded collections.

\begin{defthm}\cite[Def. 7.5, Prop. 7.6, Lem. 7.8]{koenig_silting}
Let $\mathcal{X} = \mathcal{S}_p\sqcup \mathcal{S}_n[1]$ be in $\smc\Lambda$. Then for each brick $S \in \mathcal{S}_p$, there is a 2-simple minded collection $\mu_S(\mathcal{X})$, called the \emph{left mutation} of $\mathcal{X}$ at $S$, given as follows.
\begin{itemize}
	\item The module $S$ is replaced with $S[1]$.
	\item Each module $S' \in \mathcal{S}_p\setminus \{S\}$ is replaced with the cone of
	$$S'[-1] \xrightarrow{g_{S'}} E$$
	where the map $g_{S'}$ is a minimal left $\mathsf{Filt} S$-approximation.
	\item Each shifted module $S''[1] \in \mathcal{S}_n[1]$ is replaced with the cone of
	$$S'' \xrightarrow{g_{S''}} E$$
	where the map $g_{S''}$ is a minimal left $\mathsf{Filt} S$-approximation.
\end{itemize}
In particular, every $S'[-1] \in (\mathcal{S}_p\setminus \{S\})[-1]$ and every $S'' \in \mathcal{S}_n$ admits a minimal left $\mathsf{Filt} S$-approximation. Moreover, for $S'' \in \mathcal{S}_n$, the map $g_{S''}$ must be either a monomorphism or an epimorphism so that the cone of $g_{S''}$ is isomorphic to either a module or a shifted module.
\end{defthm}

Under this definition of mutation, $\smc\Lambda$ forms a poset. Now recall that a subcategory $W\subseteq \mathsf{mod}\Lambda$ is called \emph{wide} if it is closed under extensions, kernels, and cokernels. That is, it is an exactly embedded abelian subcategory. We denote by $\mathsf{wide}\Lambda$ the set of wide subcategories of $\mathsf{mod}\Lambda$. Under inclusion, $\mathsf{wide}\Lambda$ forms a poset. Moreover, since $\Lambda$ is $\tau$-tilting finite, we have the following.

\begin{prop}\
\smallskip
\begin{enumerate}[label=\upshape{(\alph*)}]
	\item \cite[Thm. 3.3]{asai_semibricks} The maps $\mathcal{X} = \mathcal{S}_p\sqcup \mathcal{S}_n[1]\mapsto \mathcal{S}_p$ and $\mathcal{X}=\mathcal{S}_p\sqcup \mathcal{S}_n[1]\mapsto \mathcal{S}_n$ are bijections $\smc\Lambda\rightarrow\mathsf{sbrick}\Lambda$.
	\item \cite[Cor. 4.3, Thm. 4.9]{brustle_ordered} There is an isomorphism of lattices $\smc\Lambda\rightarrow \mathsf{tors}\Lambda$ given by
	$$\mathcal{X} = \mathcal{S}_p\sqcup \mathcal{S}_n[1] \mapsto \mathsf{Filt}\mathsf{Fac} \mathcal{S}_p.$$
	\item \cite[Prop. 2.26]{asai_semibricks} There is a bijection $\smc\Lambda\rightarrow \mathsf{wide}\Lambda$ given by
	$$\mathcal{X} = \mathcal{S}_p\sqcup \mathcal{S}_n[1] \mapsto \mathsf{Filt} \mathcal{S}_p.$$
\end{enumerate}
\end{prop}

One of our goals is to define 2-simple minded collections using pairwise compatibility conditions. This leads to the following definition.

\begin{define}
\label{def:sbp}
Let $\mathcal{S}_p,\mathcal{S}_n$ be semibricks in $\mathsf{mod}\Lambda$. We say that $\mathcal{S}_p\sqcup \mathcal{S}_n[1]$ is a \emph{semibrick pair} if
$$\mathrm{Hom}(\mathcal{S}_p,\mathcal{S}_n) = 0 = \mathrm{Ext}(\mathcal{S}_p,\mathcal{S}_n).$$
In particular, a semibrick pair $\mathcal{S}_p\sqcup \mathcal{S}_n[1]$ is a 2-simple minded collection if and only if $\mathsf{thick}(\mathcal{S}_p\sqcup \mathcal{S}_n[1]) = \mathcal{D}^b(\mathsf{mod}\Lambda)$. We say the semibrick pair $\mathcal{S}_p\sqcup \mathcal{S}_n[1]$ is \emph{completable} if it is a subset of a 2-simple minded collection.
\end{define}

The following shows that not all semibrick pairs are completable.

\begin{cex}
Consider the quiver
\begin{center}
\begin{tikzpicture}
	\node at (3,0) {1};
	\node at (2,0) {2};
	\node at (1,0) {3};
	\node at (0,0) {$Q=$};
	\begin{scope}[decoration={
	markings,
	mark=at position 1.0 with {\arrow[scale=1.3]{>}}}
	]
		\draw[postaction={decorate}] (1.75,0)--(1.25,0);
		\draw[postaction={decorate}] (2.75,0)--(2.25,0);
		\draw[postaction={decorate}] (1,-0.25)--(1,-0.5)--(3,-0.5)--(3,-0.25);
	\end{scope}
\end{tikzpicture}
\end{center}
and let $\Lambda = KQ/\mathrm{rad}^2 KQ$. Then ${\scriptsize\begin{matrix}1\\2\end{matrix}}\sqcup{\scriptsize\begin{matrix}2\\3\end{matrix}}[1]$ is a semibrick pair, but cannot be completed to a 2-simple minded collection. Indeed, the minimal left-$\mathsf{Filt}\left({\scriptsize\begin{matrix}1\\2\end{matrix}}\right)$ approximation ${\scriptsize\begin{matrix}2\\3\end{matrix}}\rightarrow {\scriptsize\begin{matrix}1\\2\end{matrix}}$ is neither mono nor epi.
\end{cex}

It is a natural question whether the only obstruction to completing a semibrick pair is the existence of pairs $S,S'$ of bricks such that $S\sqcup S'[1]$ is a semibrick pair with a minimal left-$\mathsf{Filt} S$ approximation $S'\rightarrow E$ which is neither mono nor epi. This leads to the following definition.

\begin{define}
A semibrick pair $\mathcal{X} = \mathcal{S}_p \sqcup \mathcal{S}_n[1]$ is called \emph{singly left mutation compatible}\footnote{The original version of this manuscript used the name ``mutation compatible'' for this property. The reason for this change is subtle and is highlighted in the followup \cite{hanson_pairwise}.} if for all bricks $S \in \mathcal{S}_p$ and $S'\in \mathcal{S}_n$ every minimal left $\mathsf{Filt} S$-approximation $g_{S'}: S'\rightarrow E$ is either a monomorphism or an epimorphism.
\end{define}

We observe that it can be determined whether a collection $\mathcal{X} = \mathcal{S}_p \sqcup \mathcal{S}_n[1]$ of bricks in $\mathsf{mod}\Lambda \sqcup \mathsf{mod}\Lambda[1]$ is a singly left mutation compatible semibrick pair by checking each pair of indecomposable direct summands. That is, $\mathcal{X}$ is a singly left mutation compatible semibrick pair if and only if all of the following hold:
\begin{enumerate}
	\item For all $S \neq T \in \mathcal{S}_p$, we have $\mathrm{Hom}(S,T) = 0 = \mathrm{Hom}(T,S)$.
	\item For all $S \neq T \in \mathcal{S}_n$, we have $\mathrm{Hom}(S,T) = 0 = \mathrm{Hom}(T,S)$.
	\item For all $S\in \mathcal{S}_p$ and $T \in \mathcal{S}_n$, we have $\mathrm{Hom}(S,T) = 0 = \mathrm{Ext}(S,T)$ and there exists a minimal left $(\mathsf{Filt} S)$-approximation $T \rightarrow S_T$ which is either mono or epi.
\end{enumerate}
We refer to conditions (1)-(3) as \emph{pairwise compatibility conditions} and say that singly left mutation compatible semibrick pairs are ``characterized by pairwise compatibility conditions''. As we shall see in Section \ref{sec:cat0}, pairwise compatibility conditions are one of the key ingredients used to show a cubical category is locally $\mathrm{CAT}(0)$. In particular, we will be interested in whether completable semibrick pairs can be characterized by pairwise compatibility conditions.

We prove in Section \ref{sec:nakayama} that every singly left mutation compatible semibrick pair is completable in the case of Nakayama algebras. In particular, this means that the completable semibrick pairs for Nakayama algebras are indeed characterized by pairwise compatibility conditions. As a followup to this paper, we show in \cite{hanson_pairwise} that this is not true in general. The simplest counterexample we present is a gentle algebra whose quiver contains 4 vertices. In \cite{BaH}, the first author and Barnard further show that if $\mathsf{rk}(\Lambda) \leq 3$, then completable semibrick pairs can always be characterized by pairwise compatibility conditions and that if $\Lambda$ is a preprojective algebra of Dynkin type with $\mathsf{rk}(\Lambda) \geq 4$, then completable semibrick pairs cannot be characterized by pairwise compatibility conditions.

\subsection{The Category of Buan and Marsh}
\label{sec:category}

We begin with the following result, allowing us to identify wide subcategories of $\Lambda$ with module categories. Our notation comes from \cite{buan_category}.

\begin{define}\cite[Def. 3.3]{jasso_reduction},\cite[Sec. 4.2]{demonet_lattice}
	Let $M \sqcup P[1]$ be a support $\tau$-rigid pair for $\Lambda$. Denote
	$$J(M\sqcup P[1]) := (M\sqcup P)^\perp\cap\lperp{\tau M} \subseteq \mathsf{mod}\Lambda,$$
	which we call the \emph{Jasso category} of $M\sqcup P[1]$.
\end{define}

\begin{thm}
\label{thm:basicWide}
Let $M \sqcup P[1]$ be a support $\tau$-rigid pair for $\Lambda$. 
\begin{enumerate}[label=\upshape(\alph*)]
	\item \cite[Thm. 4.12]{demonet_lattice},\cite[Thm. 3.8]{jasso_reduction} The Jasso category of $M\sqcup P[1]$ is wide and is equivalent to the category $\mathsf{mod}\Lambda'$ for some $\Lambda'$ which satisfies $\mathsf{rk}(M) + \mathsf{rk}(P) + \mathsf{rk}(\Lambda') = \mathsf{rk}(\Lambda)$.
	\item \cite[Thm. 4.18]{demonet_lattice} Every wide subcategory $W \in \mathsf{wide}\Lambda$ is the Jasso category of at least one support $\tau$-rigid pair.
	\end{enumerate}
\end{thm}

By identifying each wide subcategory $W \in \mathsf{wide}\Lambda$ with the algebra $\mathsf{mod}\Lambda'$ given in part (a), we can define support $\tau$-rigid pairs and Jasso categories in $W$. This idea is due to Jasso \cite{jasso_reduction}.

\begin{define}
Let $W \in \mathsf{wide} \Lambda$ be a wide subcategory, identified with $\mathsf{mod}\Lambda'$ as above.
\begin{enumerate}
	\item Let $M \in W$ and $P \in \mathsf{proj} W$. Then $M\sqcup P[1]$ is a \emph{support $\tau$-tilting pair} for $W$ if $\mathrm{Hom}(P,M) = 0 = \mathrm{Hom}(M,\tau_W M)$, where $\tau_W$ is the Auslander-Reiten translate in $W(=\mathsf{mod}\Lambda')$. We denote by $\str W$ the set of (isoclasses of) support $\tau$-rigid pairs in $W$.
	\item Let $M \sqcup P[1]$ be a support $\tau$-rigid pair for $W$. Then the \emph{Jasso category} of $M\sqcup P[1]$ in $W$ is the (wide) subcategory
		$$J_W(M\sqcup P[1]):= (M\sqcup P)^\perp\cap\lperp{\tau_W M}\cap W.$$
\end{enumerate}
\end{define}

\begin{rem}
	In general, modules which are projective in $W$ may not be projective in $\mathsf{mod}\Lambda$. Likewise, the functor $\tau_W$ and the Jasso categories in $W$ are often different than in $\mathsf{mod}\Lambda$. Moreover, since $W$ is identified with $\mathsf{mod}\Lambda'$, all of the results about $\tau$-tilting theory and Jasso categories in $\Lambda$ hold in $W$ as well. In particular, this means $\mathrm{Hom}_\Lambda(M,\tau_W M) = 0$ if and only if $\mathrm{Ext}^1_\Lambda(M,(\mathsf{Fac} M)\cap W) = 0$ by \cite[Prop. 5.8]{auslander_almost}.
\end{rem}

Now recall that given an object $U$ in $\mathcal{D}^b(\mathsf{mod}\Lambda)$, a $t$-tuple $(U_1,\ldots,U_t)$ is an \emph{ordered decomposition} of $U$ if each $U_i$ is indecomposable and $\bigsqcup_{i=1}^tU_i = U$. Moreover, for $M,N \in \mathsf{mod}\Lambda$ indecomposable, let $\mathrm{rHom}(M,N)$ be the set of morphisms $M\rightarrow N$ which are not isomorphisms. Then we denote
$$\mathrm{rad}(M,N):= \sum_{f \in \mathrm{rHom}(M,N)} \mathrm{Im}(f).$$
This definition then extends additively to arbitrary modules $M,N \in \mathsf{mod}\Lambda$.

We now recall the following bijections of Buan and Marsh.

\begin{thm}
\label{thm:littlebijections}
Let $W \in \mathsf{wide}\Lambda$ be a wide subcategory.
\begin{enumerate}[label=\upshape(\alph*)]
	\item \cite[Prop. 5.6]{buan_exceptional} Let $M \in \mathsf{ind}(\str W)$ be a module. Then there is a bijection
	$$\mathcal{E}_M^W: \{N\sqcup Q[1]|M\sqcup N\sqcup Q[1] \in \str W\} \leftrightarrow \str J_W(M)$$
	summarized as follows. Decompose $$N\sqcup Q[1] = N_1\sqcup\cdots\sqcup N_n\sqcup Q_1[1]\sqcup\cdots\sqcup Q_q[1]$$ into a direct sum of indecomposable objects. Then
	\begin{eqnarray*}
		\mathcal{E}_M^W(N_i) &=& N_i/\mathrm{rad}(M,N_i) \textnormal{, if }N_i \notin \mathsf{Fac} M\\
		\mathcal{E}_M^W(N_i) &=& B/\mathrm{rad}(M,B)[1] \textnormal{ for some direct summand, }B,\textnormal{ of the}\\
		&&\textnormal{Bongartz complement of }M\textnormal{ in }W\textnormal{, if } N_i\in\mathsf{Fac} M\\
		\mathcal{E}_M^W(Q_i[1]) &=& B/\mathrm{rad}(M,B)[1] \textnormal{ for some direct summand, } B,\textnormal{ of the }\\
		&& \textnormal{Bongartz complement of }M\textnormal{ in }W
	\end{eqnarray*}
	and the formula for the bijection extends additively.
	\item \cite[Prop. 5.10a]{buan_exceptional} Let $P[1] \in \mathsf{ind}(\str W)$ be a shifted projective. Then there is a bijection
	$$\mathcal{E}_{P[1]}^W: \{N\sqcup Q[1]|N\sqcup P[1]\sqcup Q[1] \in \str W\} \leftrightarrow \str J_W(P[1])$$ summarized as follows. Decompose $$N\sqcup Q[1] = N_1\sqcup\cdots\sqcup N_n\sqcup Q_1[1]\sqcup\cdots\sqcup Q_q[1]$$ into a direct sum of indecomposable objects. Then
	\begin{eqnarray*}
		\mathcal{E}_{P[1]}^W(N_i) &=& N_i\\
		\mathcal{E}_{P[1]}^W(Q_i[1]) &=& \mathcal{E}_P^W(Q_i)[1] = Q_i/\mathrm{rad}(P,Q_i)[1]
	\end{eqnarray*}
	and the formula for the bijection extends additively.
	\item \cite[Thm. 3.6, 5.9]{buan_category} Let $M \sqcup P[1] \in \str W$. Then there is a bijection
	$$\mathcal{E}_{M\sqcup P[1]}^W: \{N\sqcup Q[1]|M\sqcup N\sqcup P[1]\sqcup Q[1]\in\str W\}\leftrightarrow \str J_W(M\sqcup P[1])$$ given as follows. Let $(U_1\ldots U_t)$ be any ordered decomposition of $M\sqcup P[1]$. Then define recursively
	$$\mathcal{E}_{M\sqcup P[1]}^W = \mathcal{E}_{\mathcal{E}_{U_1}^W(U_2\sqcup\cdots\sqcup U_t)}^{J_W(U_1)}\circ\mathcal{E}_{U_1}^W.$$
	In particular, the definition of $\mathcal{E}^W_{M\sqcup P[1]}$ is independent of the choice of $(U_1,\ldots,U_t)$.
\end{enumerate}
\end{thm}

These bijections, given explicitly in \cite{buan_category} and \cite{buan_exceptional}, allow for the definition of the $\tau$-cluster morphism category of $\Lambda$.

\begin{define}\cite[Section 1]{buan_category}
\label{def:category}
The \emph{$\tau$-cluster morphism category} of $\Lambda$, denoted $\mathfrak{W}(\Lambda)$, is defined as follows:
\begin{enumerate}
	\item The objects of $\mathfrak{W}(\Lambda)$ are the wide subcategories of $\mathsf{mod}\Lambda$.
	\item For each pair of wide subcategories $W_1,W_2\in\mathsf{wide}\Lambda$, we define
		$$\mathrm{Hom}_{\mathfrak{W}(\Lambda)}(W_1,W_2) = \{[U]|U \in \str{W_1} \textnormal{ and }J_{W_1}(U) = W_2\}.$$
		In particular, $\mathrm{Hom}_{\mathfrak{W}(\Lambda)}(W_1,W_2) = 0$ unless $W_1 \supseteq W_2$, and in the case that $W_1 = W_2$, there is a single morphism $[0]:W_1\rightarrow W_1$.
	\item For $[U] \in \mathrm{Hom}_{\mathfrak{W}(\Lambda)}(W_1, W_2)$ and $[V] \in \mathrm{Hom}_{\mathfrak{W}(\Lambda)}(W_2,W_3)$, we define
		$$[V]\circ[U] = \left[U\sqcup\left(\mathcal{E}_U^{W_1}\right)\inv(V)\right].$$
\end{enumerate}
\end{define}

We remark that this definition agrees with the definition of the classical cluster morphism category, defined for hereditary algebras in \cite[Section 1]{igusa_signed}.

\section{Cubical Categories}
\label{sec:cubical}

Cubical categories were introduced by the second author in \cite{igusa_category} as a tool to study the geometry of the classical cluster morphism categories and picture groups of representation-finite hereditary algebras. The goal of this section will be to show that the $\tau$-cluster morphism category of any $\tau$-tilting finite algebra is cubical (Theorem \ref{thm:cubical}), extending the results of \cite{igusa_category} and \cite{igusa_signed}.

As a motivating example of a cubical category, consider the category $\mathcal{I}^k$, whose objects are the subsets of $\{1,\ldots,k\}$ and whose morphisms are inclusions. This is referred to as the \emph{standard $k$-cube category}.

Before stating the general definition, we recall the notion of the \emph{factorization category} of a morphism $A\xrightarrow{f} B$ in a category $\mathcal{C}$. This is the category $\mathsf{Faq}(f)$ whose objects are factorizations $A\xrightarrow{g}C\xrightarrow{h} B$ with $h\circ g = f$ and whose morphisms
$$(A\xrightarrow{g}C\xrightarrow{h} B)\rightarrow(A\xrightarrow{g'}C'\xrightarrow{h'}B)$$
are morphisms $\phi:C\rightarrow C'$ in $\mathcal{C}$ such that $\phi\circ g = g'$ and $h = h'\circ \phi$. Given an object $A\xrightarrow{g}C\xrightarrow{h} B$ in $\mathsf{Faq}(f)$, we call $g$ a \emph{first factor} of $f$ if $g$ is irreducible in $\mathcal{C}$. Likewise, we call $h$ a \emph{last factor} of $f$ if $h$ is irreducible in $\mathcal{C}$.

\begin{rem}\label{ref: basic property of I0}
We observe that the following are equivalent.
\begin{enumerate}
\item $\mathsf{Faq}(id_A)\cong \mathcal I^0$ for any object $A$ in $\mathcal{C}$.
\item The category $\mathcal{C}$ has only one object in every isomorphism class and each object of $\mathcal{C}$ is indecomposable.
\end{enumerate}
Furthermore, when this property holds, every irreducible morphism $f:A\to B$ will have exactly two factorizations: $f\circ id_A$ and $id_B\circ f$. Equivalently, $\mathsf{Faq}(f)\cong \mathcal I^1$.
\end{rem}

The following definition is an extension of these properties. Figure \ref{fig:cube} gives an example.

\begin{define}\cite[Def. 3.2]{igusa_category}
\label{def:cubical}
A \emph{cubical category} is a small category $\mathcal{C}$ with the following properties:
\begin{enumerate}[label=\upshape(\alph*)]
	\item Every morphism $f: A \rightarrow B$ in $\mathcal{C}$ has a \emph{rank}, $\mathsf{rk}(f)$, which is a non-negative integer so that $\mathsf{rk}(f\circ g) = \mathsf{rk}(f) + \mathsf{rk}(g)$ for all composable morphisms $f,g$ in $\mathcal{C}$. In particular, $\mathsf{rk}(id_A)=0$ for any object $A$.
	\item If $\mathsf{rk}(f) = k$ then the factorization category of $f$ is isomorphic to the standard $k$-cube category: $\mathsf{Faq}(f) \cong \mathcal{I}^k$. In particular, $\mathcal{C}$ satisfies the properties in Remark \ref{ref: basic property of I0}.
	\item The forgetful functor $\mathsf{Faq}(f) \rightarrow \mathcal{C}$ taking $A\xrightarrow{g}C\xrightarrow{h}B$ to $C$ is an embedding. In particular, every morphism of rank $k$ has $k$ distinct first factors and $k$ distinct last factors.
	\item Every morphism of rank $k$ is determined by its $k$ first factors.
	\item Every morphism of rank $k$ is determined by its $k$ last factors.
\end{enumerate}
\end{define}

The remainder of this section is devoted to showing the category $\mathfrak{W}(\Lambda)$ is cubical for any $\tau$-tilting finite algebra $\Lambda$. We first recall how to determine the first and last factors of a morphism in $\mathfrak{W}(\Lambda)$. In \cite{igusa_signed}, exceptional sequences correspond to decompositions of morphisms in $\mathfrak{W}(\Lambda)$ into irreducible factors. This result applies in the more general context under the broader definition of Buan and Marsh.

\begin{define}\cite[Def. 1.3]{buan_exceptional}
\label{def:exceptional}
Let $W$ be a wide subcategory of $\mathsf{mod}\Lambda$ and let $U_1,\ldots,U_t\in\mathcal{D}^b(W)$ be indecomposable. Then $(U_1,\ldots,U_t)$ is called a \emph{signed $\tau$-exceptional sequence} (of length $t$) if $U_t \in \str W$ and $(U_1,\ldots,U_{t-1})$ is a signed $\tau$-exceptional sequence in $J_W(U_t)$.
\end{define}

\begin{thm}\cite[Rmk. 5.12, Thm. 11.8]{buan_exceptional}
\label{thm:bigbijection}
Let $W$ be a wide subcategory of $\mathsf{mod}\Lambda$. Denote by $\mathcal{O}_t(W)$ the set of ordered decompositions of support $\tau$-rigid objects for $W$ with $t$ indecomposable direct summands. Then there is a bijection
$$\mathcal{O}_t(W) \xrightarrow{\Psi_t^W}\{\tau\textnormal{-exceptional sequences for }W\textnormal{ of length }t\}$$
given inductively by
$$\Psi_t^W(U_1,...,U_t) = (\Psi_{t-1}^{J_W(U_t)}(\mathcal{E}^W_{U_t}(U_1),...,\mathcal{E}^W_{U_t}(U_{t-1})), U_t).$$
Moreover, for all $U \in \str(\Lambda)$ with $t$ indecomposable direct summands, the map $\Psi_t^W$ induces a bijection
$$\{\textnormal{Ordered decompositions of }U\}$$
$$\downarrow$$
$$\{\textnormal{Factorizations of }[U]\textnormal{ into irreducible morphisms in }\mathfrak{W}(\Lambda)\}.$$
\end{thm}

In particular, this shows that the first(resp. last) factors of a morphism $[U]$ in $\mathfrak{W}(\Lambda)$ are the morphisms $[U_i]$ for which there is an ordered decomposition of $U$ corresponding to a $\tau$-exceptional sequence with $U_i$ as the rightmost (resp. leftmost) entry. Before examining this further, we prove the following lemma and corollary. The proofs are similar to \cite[Lem. 3.13]{igusa_category}.

\begin{lem}
\label{lem:cubeembedding}
Let $[U]: W \rightarrow W'$ be a morphism in $\mathfrak{W}(\Lambda)$.
\begin{enumerate}[label=\upshape(\alph*)]
	\item If $V \ncong L \in \mathsf{add}(U)$, then $J_W(V) \neq J_W(L)$.
	\item There is exactly one morphism $$(W\xrightarrow{[V_1]}J_W(V_1)\xrightarrow{[V_2]}W') \rightarrow (W\xrightarrow{[L_1]}J_W(L_1)\xrightarrow{[L_2]}W')$$ in $\mathsf{Faq}[U]$ if and only if $V_1 \in \mathsf{add}(L_1)$. Otherwise, there are no such morphisms.
\end{enumerate}
\end{lem}

\begin{proof}
(a) Let $V \ncong L \in \mathsf{add}(U)$. We can assume without loss of generality that $V$ and $L$ have no common direct summands. Now $\mathcal{E}^W_{L}(V) \in \str J_W(L)$ by definition. We will show that $\mathcal{E}^W_{L}(V) \notin \str J_W(V)$.

Write $U = M\sqcup P[1]$ with $M,P \in W$ and $P$ projective. Recall from Theorem \ref{thm:littlebijections}(c) that $\mathcal{E}^W_{L} = \mathcal{E}^{J_W(P)}_{M}\circ\mathcal{E}^W_{P[1]}$ and let $N$ be an indecomposable direct summand of $V$.

If $N = Q[1]$ is not a module, then by Theorem \ref{thm:littlebijections}(b), $\mathcal{E}^W_{P[1]}(Q[1]) = Q'[1]$, where $Q'$ is a nonzero quotient of $Q$. We then have $\mathcal{E}^W_L(Q[1]) = \mathcal{E}^{J_W(P)}_{M}(Q'[1]) = B[1]$, where $B$ is a quotient of the Bongartz complement of $M$ in $J_W(P)$ and there is a nonzero map $Q'\rightarrow B$. We conclude that $\mathrm{Hom}(Q[1],\mathcal{E}^W_{L}(Q[1])) \neq 0$. Therefore $\mathcal{E}^W_L(V) \notin \str J_W(V)$ since $J_W(V) \subseteq Q^\perp$.

If $N$ is a module, then by Theorem \ref{thm:littlebijections}(b), $\mathcal{E}^W_{L}(N) = \mathcal{E}^{J_W(P)}_{M}(N)$. If this is a module, then by Theorem \ref{thm:littlebijections}(a) it is a nonzero quotient of $N$ and \linebreak $\mathrm{Hom}(N,\mathcal{E}^W_{L}(N)) \neq 0$, so we are done as before. Otherwise, Theorem \ref{thm:littlebijections}(a) says that $N \in \mathsf{Fac} M$ and there is some direct summand $B$ of the Bongartz complement of $M$ in $J_W(P)$ so that $$\mathcal{E}^{J_W(P)}_{M}(N) = (B/\mathrm{rad}(M,B))[1] = \left(\mathcal{E}^{J_W(P)}_M(B)\right)[1].$$ In particular, we observe that $B\sqcup N \notin \str J_W(P)$.
Applying Theorem \ref{thm:littlebijections}(b) once again, the fact that $B$ is a module means that $\left(\mathcal{E}^{W}_{P[1]}\right)^{-1}(B) = B$. Putting this together, we then have that $B\sqcup M \in \str W$ and $B\sqcup N \notin \str W$. Now, since $N \in \mathsf{Fac} M$, the fact that $B\sqcup M$ is support $\tau$-rigid (in $W$) implies that $\mathrm{Hom}(N,\tau_W B) = 0$. Therefore, since $B\sqcup N$ is not support $\tau$-rigid (in $W$), there exists a nonzero map $B\rightarrow\tau_{W} N$. Moreover, this map factors through $B/\mathrm{rad}(M,B)$ since $\mathrm{Hom}(M, \tau_W N) = 0$. We conclude that $\mathrm{Hom}(\mathcal{E}^W_{L}(N), \tau_W N[1]) \neq 0$. Therefore $\mathcal{E}_L^W(V) \notin \str J_W(V)$ since $J_W(V) \subseteq \lperp{(\tau_WN)}$.

(b) Let $(W\xrightarrow{[V_1]}J_W(V_1)\xrightarrow{[V_2]}W'), (W\xrightarrow{[L_1]}J_W(L_1)\xrightarrow{[L_2]}W') \in \mathsf{Faq}[U]$. Suppose there exists a morphism $[T]:J_W(V_1) \rightarrow J_W(L_1)$ in $\mathfrak{W}(\Lambda)$ such that $[L_1] = [T]\circ[V_1] = \left[V_1\sqcup \left(\mathcal{E}_{V_1}^{W}\right)\inv(T)\right]$. We see that $V_1 \in \mathsf{add}(L_1)$. Moreover, we see that $T$ is unique since $\mathcal{E}_{V_1}^{W}$ is a bijection.

Finally, it follows from (a) that $V_2 = T \sqcup \left(\mathcal{E}_T^{J_W(V_1)}\right)\inv(L_2)$. Therefore $[V_2] = [L_2]\circ[T]$ and $[T]$ defines a morphism $$(W\xrightarrow{[V_1]}J_W(V_1)\xrightarrow{[V_2]}W') \xrightarrow{[T]} (W\xrightarrow{[L_1]}J_W(L_1)\xrightarrow{[L_2]}W')$$ in $\mathsf{Faq}[U]$ as desired.
\end{proof}

\begin{cor}
\label{cor:cubeembedding}
Let $[U]:W\rightarrow W'$ be a morphism in $\mathfrak{W}(\Lambda)$. Then
\begin{enumerate}[label=\upshape(\alph*)]
	\item There is an isomorphism of categories $\mathsf{Faq}[U] \cong \mathcal{I}^{\mathsf{rk}(U)}$.
	\item The forgetful functor $\mathsf{Faq}[U] \rightarrow \mathfrak{W}(\Lambda)$ is an embedding.
\end{enumerate}
\end{cor}

\begin{proof}
(a) Fix an ordered decomposition $(U_1,\ldots,U_n)$ of $U$. For all $X = \linebreak(W \xrightarrow{V_1} J_W(V_1)\xrightarrow{V_2}W') \in \mathsf{Faq}[U],$ define $FX$ to be the set of indices $i$ for which $U_i \in \mathsf{add}(V_1)$. By Lemma \ref{lem:cubeembedding}(a), the map $X \mapsto FX$ is a bijection between objects of $\mathsf{Faq}[U]$ and subsets of $\{1,\ldots,n\}$.

Now let $f: X\rightarrow Y$ be a morphism in $\mathsf{Faq}[U]$. It follows from Lemma \ref{lem:cubeembedding}(b) that $FX \subseteq FY$ and $f$ is the only morphism $X\rightarrow Y$. Thus define $Ff$ to be the inclusion $FX \subseteq FY$.

Since both categories have a poset structure, the composition law is preserved and $F$ defines a functor $Fac[U] \rightarrow \mathcal{I}^n$. The inverse functor is given by sending every subset $S \subseteq \{1,\ldots,n\}$ to $F\inv S := (W\xrightarrow{U_S}J_W(U_S)\xrightarrow{U'_S}W')$, where $U_S = \sqcup_{i \in S}U_i$ and $U'_S = \mathcal{E}_{U_S}^W(\sqcup_{i \notin S}U_i)$, and sending every inclusion $S \subseteq S'$ to the unique morphism $F\inv S\rightarrow F\inv S'$ in $\mathsf{Faq}[U]$.

(b) It follows from Lemma \ref{lem:cubeembedding}(a) that the forgetful functor $\mathsf{Faq}[U] \rightarrow \mathfrak{W}(\Lambda)$ is injective on objects. Likewise, it follows from Lemma \ref{lem:cubeembedding}(b) that the forgetful functor is faithful.
\end{proof}

\begin{ex}
	Consider the quiver
\begin{center}
\begin{tikzpicture}
	\node at (3,0) {1};
	\node at (2,0) {2};
	\node at (1,0) {3};
	\node at (0,0) {$Q=$};
	\begin{scope}[decoration={
	markings,
	mark=at position 1.0 with {\arrow[scale=1.3]{>}}}
	]
		\draw[postaction={decorate}] (1.75,0)--(1.25,0);
		\draw[postaction={decorate}] (2.75,0)--(2.25,0);
		\draw[postaction={decorate}] (1,-0.25)--(1,-0.5)--(3,-0.5)--(3,-0.25);
	\end{scope}
\end{tikzpicture}
\end{center}
and let $\Lambda = KQ/\mathrm{rad}^2 KQ$. Then the factorization category of the morphism $[P_1\sqcup S_1\sqcup P_3[1]] \in \mathrm{Hom}_{\mathfrak{W}(\Lambda)}(\mathsf{mod}\Lambda, 0)$ is shown in Figure \ref{fig:cube}.
\end{ex}

\begin{figure}
\begin{tikzcd}[row sep={40,between origins}, column sep={40,between origins}]
      & J(P_1) \ar[leftarrow]{rr}{[P_1]}\ar{dd}[near start]{[S_3[1]]}\ar{dl}[anchor=south east]{[P_2[1]]} & & \mathsf{mod}\Lambda \ar{dd}{[P_3[1]]}\ar{dl}{[S_1]} \\
    \mathsf{Filt}(S_1) \ar[leftarrow,crossing over]{rr}[near start,below]{[P_1]} \ar{dd}[left]{[S_3[1]]} & & J(S_1) \\
      & \mathsf{Filt}(S_2)  \ar[leftarrow]{rr}[near end]{[P_1]} \ar{dl}[near start,left]{[S_2[1]]} & &  J(P_3) \ar{dl}{[S_1]} \\
    0 \ar[leftarrow]{rr}{[P_1]} && \mathsf{Filt}(P_1) \ar[from=uu,crossing over]{}[near end,left]{[S_3[1]]}
\end{tikzcd}
\caption{There are 3! = 6 factorizations of $[P_1\sqcup S_1\sqcup P_3[1]]$ into irreducible morphisms, shown here with every square commutative. For example, $[P_1\sqcup S_1\sqcup P_3[1]] = [S_3[1]]\circ[P_2[1]]\circ[P_1]$. The three morphisms with source $\mathsf{mod}\Lambda$ are the first factors of $[P_1\sqcup S_1\sqcup P_3[1]]$ and the three morphisms with target 0 are the last factors of $[P_1\sqcup S_1\sqcup P_3[1]]$. We observe that the classifying space of the factorization category of $[P_1\sqcup S_1\sqcup P_3[1]]$ is (isometric to) a solid cube; that is, $\mathcal{B} \mathsf{Faq}[P_1\sqcup S_1\sqcup P_3[1]] = [0,1]^3$.}
\label{fig:cube}
\end{figure}

\subsection{Last Morphisms}
\label{sec:last}

Our next step is to describe the last factors of morphisms in $\mathfrak{W}(\Lambda)$. Throughout this section, we denote by $\mathsf{last}[U]$ the set of last factors corresponding to the morphism $[U]$. Moreover, if $U \cong U_1\sqcup\cdots\sqcup U_n$ is a direct sum of indecomposable objects, we denote $U_i^c := U_1\sqcup\cdots\sqcup\widehat{U_i}\sqcup\cdots\sqcup U_n$, where $\widehat{U_i}$ means we delete $U_i$ as a direct summand.

The following is immediate from the Theorem \ref{thm:bigbijection}

\begin{rem}
\label{rem:lasts}
	Let $[U]:W\rightarrow W'$ be a morphism in $\mathfrak{W}(\Lambda)$. Decompose $U \cong U_1\sqcup\cdots\sqcup U_n$ as a direct sum of indecomposable objects. Then
	$$\mathsf{last}[U] = \left\{\mathcal{E}^W_{U_i^c}(U_i)\right\}_{i=1}^n.$$
\end{rem}

 We now show that we need only consider morphisms with target $0$.

\begin{lem}
\label{lem:fewerlasts}
The following are equivalent for the category $\mathfrak{W}(\Lambda)$.
\begin{enumerate}
	\item Every morphism is determined by its last factors.
	\item Every morphism with target 0 is determined by its last factors.
\end{enumerate}
\end{lem}

\begin{proof}
Assume every morphism with target 0 is determined by its last factors. Now let $\{[U_i]:W_i\rightarrow W\}$ be a set of rank one morphisms which form the last factors of at least one morphism $[U]:W'\rightarrow W$. Now if $W = 0$, then we are done. Otherwise, choose any morphism $[V]:W'\rightarrow 0$. For each $i$, denote $V^i := \left(\mathcal{E}_{U_i}^W\right)^{-1}(V)$. Then by Remark \ref{rem:lasts},
$$\left\{\left[\mathcal{E}_{V^i}^{W_i}(U_i)\right]\right\}\sqcup\mathsf{last}[V] = \mathsf{last}\left([V]\circ[U]\right).$$
In particular, $[U]$ is defined uniquely as a consequence of Corollary \ref{cor:cubeembedding}.
\end{proof}

For the remainder of this section, we examine the relationship between sets of last factors in $\mathfrak{W}(\Lambda)$ and 2-simple minded collections. We begin with the following result relating bricks to minimal wide subcategories. Recall that in an arbitrary poset, the relation $x < y$ is minimal if $x \lneq z \leq y$ implies $z = y$.

\begin{thm}\label{thm:minIncl}\cite[Thm. 3.3]{demonet_lattice}
There is a bijection
$$\mathsf{brick}\Lambda\rightarrow\{W\in\mathsf{wide}\Lambda|0\subset W\textnormal{ is a minimal inclusion}\}$$
given by $S \mapsto \mathsf{Filt} S$. Moreover, we have $\mathsf{brick}(\mathsf{Filt} S) = \{S\}$.
\end{thm}

\begin{lem}
\label{lem:lastbricks}
Let $S$ be a brick in $\mathsf{mod}\Lambda$ and let $W = \mathsf{Filt} S$. Then $\stt W = \{P_S, P_S[1]\}$ for some $P_S \in W$. Moreover, there is a bijection
$$\mathsf{br}:\mathsf{brick}\Lambda \rightarrow \{P \in \mathsf{mod}\Lambda|\exists W \in \mathsf{wide}\Lambda: \stt W = \{P,P[1]\}\}$$ given by $\mathsf{br}(S) = P_S$ and $\mathsf{br}^{-1}(P) = P/\mathrm{rad}(P,P)$. In particular, the wide subcategory $W$ is uniquely determined by either $S$ or $P_S$.
\end{lem}

\begin{proof}
Let $S \in \mathsf{brick}\Lambda$ and recall from Theorem \ref{thm:basicWide} that $\mathsf{Filt} S$ is equivalent to the (small) module category of a $\tau$-tilting finite algebra. Moreover, Theorem \ref{thm:minIncl} says that $\mathsf{Filt} S$ contains a single brick. In particular, $\mathsf{Filt} S$ contains a unique simple object (namely $S$) and a unique indecomposable projective object, which we denote~$P_S$. In particular, we have $\{P_S,P_S[1]\} \subseteq \stt(\mathsf{Filt} S)$. Moreover, $0 \in \str(\mathsf{Filt} S)$ is almost complete, and so $\stt(\mathsf{Filt} S) = \{P_S,P_S[1]\}$ by Theorem \ref{thm:bongartz}(b). Applying the brick-$\tau$-rigid correspondence of Demonet-Iyama-Jasso (see \cite[Thm. 4.2]{demonet_tilting}), we then have that $P_S/\mathrm{rad}(P_S,P_S) \cong S$; i.e., that $\mathsf{br}^{-1}\circ \mathsf{br}(S) = S$.

Conversely, let $P \in \mathsf{mod}\Lambda$ and suppose there exists a wide subcategory $W \in \mathsf{wide}\Lambda$ so that $\stt W = \{P,P[1]\}$. Once again, Theorem \ref{thm:basicWide} lets us apply the brick-$\tau$-rigid correspondence of Demonet-Iyama-Jasso, and we conclude that $\mathsf{brick} W =  \{P/\mathrm{rad}(P,P)\}$. In particular, $P$ is indecomposable, $P/\mathrm{rad}(P,P)$ is simple in $W$, and $W = \mathsf{Filt}(P/\mathrm{rad}(P,P))$. We conclude that $\mathsf{br}\circ \mathsf{br}^{-1}(P) = P$.
\end{proof}

Now recall that if $L$ is a lattice and $v,v' \in L$ with $v<v'$, then the intervals $[v,v']$ and $(v,v')$ are the sets $\{v''\in L|v\leq v''\leq v\}$ and $\{v''\in L|v<v''<v\}$. The \emph{Hasse quiver} of $L$ is then the quiver with vertex set $L$ and an arrow $v\rightarrow v'$ if $v'<v$ and $(v',v) = \emptyset$ (i.e., there is an arrow $v \rightarrow v'$ if $v' < v$ is a minimal relation). We remark that in our context, the Hasse quivers $\mathsf{Hasse}(\stt\Lambda),\mathsf{Hasse}(\smc\Lambda),$ and $\mathsf{Hasse}(\mathsf{tors}\Lambda)$ are the same up to relabeling the vertices. See \cite[Cor.~4.3, Thm.~4.9]{brustle_ordered}.

\emph{Brick labeling} is a way of assigning a brick to each arrow of $\mathsf{Hasse}(\stt\Lambda)$ (and the other equivalent Hasse quivers) which generalizes the labeling by \emph{$c$-vectors} in the hereditary case (in fact, the representation-theoretic $c$-vectors of \cite{fu} are explicitly related to brick labels in \cite{treffinger_sign}). This labeling has been independently defined by Demonet, Iyama, Reading, Reiten, and Thomas \cite{demonet_lattice}; Barnard, Carroll and Zhu \cite{barnard_minimal}; Asai \cite{asai_semibricks}; and Br\"ustle, Smith, and Treffinger \cite{brustle_wall}. We will discuss brick labeling in more detail in section \ref{sec:pi1}, but in this section we use only the following facts:

\begin{defthm}\
\smallskip
\begin{enumerate}[label=\upshape(\alph*)]
\item \cite[Def. 2.14]{asai_semibricks} Let $U \rightarrow V$ be an arrow in $\mathsf{Hasse}(\stt\Lambda)$. Suppose $U = M\sqcup M'\sqcup P[1]$ with $M$ indecomposable and $V$ a (left) mutation of $U$ at $M$. Then the \emph{brick label} assigned to this arrow is the brick $M/\mathrm{rad}(M\sqcup M',M)$.
\item \cite[Thm. 3.12]{asai_semibricks} Let $U \in \stt\Lambda$. Let $\mathsf{Out}(U)$ be the set of bricks labeling arrows $U\rightarrow V$ for some $V$ in $\mathsf{Hasse}(\stt\Lambda)$ and let $\mathsf{In}(U)$ be the set of bricks labeling arrows $V'\rightarrow U$ for some $V'$ in $\mathsf{Hasse}(\stt\Lambda)$. Then
$$\mathcal{X(U)}:= \{B|B\in \mathsf{Out}(U)\}\sqcup \{B[1]|B\in \mathsf{In}(U)\}$$
is a 2-simple minded collection. Moreover, every 2-simple minded collection occurs in this way.
\end{enumerate}
\end{defthm}

This leads to the main theorem of this section.

\begin{thm}
\label{thm:lastmorphisms}
Let $U$ be a support $\tau$-tilting pair for $\Lambda$. Then $\mathsf{br}\inv(\mathsf{last}[U]) = \mathcal{X}(U)$, where we define $\mathsf{br}\inv(P[1]) := \mathsf{br}\inv(P)[1]$. In particular, $\mathsf{br}\inv\circ\mathsf{last}$ defines a bijection $\stt\Lambda \rightarrow \smc\Lambda$.
\end{thm}

\begin{proof}
Decompose $U = M_1\sqcup\cdots\sqcup M_m\sqcup P_{m+1}[1]\sqcup\cdots\sqcup P_n[1]$ into a direct sum of indecomposable objects.
Now recall from Remark \ref{rem:lasts} that

$$\mathsf{last}[U] = \left\{\left[\mathcal{E}_{M_i^c\sqcup P[1]}(M_i)\right]\right\}\sqcup\left\{\left[\mathcal{E}_{M\sqcup P_j^c[1]}(P_j[1])\right]\right\}.$$

We will show that these morphisms correspond precisely to the brick labels of $U$. First consider $M_i$ a direct summand of $M$. There are then two subcases to consider.

Case 1: Assume $M_i \notin \mathsf{Fac} M_i^c$. Thus $M\sqcup P[1]$ has a left mutation at $M_i$ with label
$$S = M_i/\mathrm{rad}(M,M_i).$$
On the other hand, we have
\begin{eqnarray*}
\mathcal{E}_{M_i^c\sqcup P[1]}(M_i) &=& \mathcal{E}^{J(P)}_{\mathcal{E}_{P[1]}(M_i^c)}\mathcal{E}_{P[1]}(M_i)\\
&=&\mathcal{E}^{J(P)}_{M_i^c}(M_i)\\
&=& M_i/\mathrm{rad}(M_i^c\,M_i)
\end{eqnarray*}
by Theorem \ref{thm:littlebijections}. We observe that $S, \mathcal{E}_{M_i^c\sqcup P[1]}(M_i) \in J(M_i^c\sqcup P[1])$, which has rank one. Thus as $\mathcal{E}_{M_i^c\sqcup P[1]}(M_i)$ is a module, it follows immediately that $\mathcal{E}_{M_i^c\sqcup P[1]}(M_i) = \mathsf{br}(S)$.

Case 2: Assume $M_i \in \mathsf{Fac} M_i^c$. It follows that $\mathsf{Fac} M = \mathsf{Fac} M_i^c$, meaning there exists an indecomposable $M_i'$ such that $M_i^c\sqcup M_i'\sqcup P[1]$ is the right mutation of $M\sqcup P[1]$ at $M_i$, which has label
$$S[1] = \left(M_i'/\mathrm{rad}(M_i^c\sqcup M_i',M_i')\right)[1].$$
On the other hand, by Theorem \ref{thm:littlebijections},
\begin{eqnarray*}
\mathcal{E}_{M_i^c\sqcup P[1]}(M_i) &=& \mathcal{E}^{J(P)}_{\mathcal{E}_{P[1]}(M_i^c)}\mathcal{E}_{P[1]}(M_i)\\
&=&\mathcal{E}^{J(P)}_{M_i^c}(M_i)\\
&=&\left(B_i/\mathrm{rad}(M_i^c,B_i)\right)[1]
\end{eqnarray*}
where $B_i$ is a direct summand of the Bongartz completion of $M_i^c$ in $J(P)$. Now, $M_i^c$ is almost complete support $\tau$-rigid in $J(P)$, which means $B_i$ is precisely the Bongartz complement of $M_i^c$ in $J(P)$. In particular, this means $B_i = M_i'$. We conclude that
$$\mathcal{E}_{M_i^c\sqcup P[1]}(M_i)=\left(M_i'/\mathrm{rad}(M_i^c,M_i')\right)[1].$$
As in Case 1, we observe that $S$ is a brick in $J(M_i^c\sqcup P[1])$ and hence $\mathcal{E}_{
M_i^c\sqcup P[1]}(M_i) = \mathsf{br}(S[1])$.

Now consider $P_j$ a direct summand of $P$. Then there exists an indecomposable $M_j'$ such that $M\sqcup M_j'\sqcup P_j^c[1]$ is the right mutation of $M\sqcup P[1]$ at $P_j[1]$, which has label
$$S[1] = \left(M_j'/\mathrm{rad} (M\sqcup M_j',M_j')\right)[1].$$
On the other hand, by Theorem \ref{thm:littlebijections},
\begin{eqnarray*}
\mathcal{E}_{M\sqcup P_j^c[1]}(P_j[1]) &=& \mathcal{E}^{J(P_j^c)}_{\mathcal{E}_{P_j^c[1]}(M)}\mathcal{E}_{P_j^c[1]}(P_j[1])\\
&=& \mathcal{E}^{J(P_j^c)}_M \left(P_j/\mathrm{rad}(P_j^c,P_j)[1]\right)\\
&=& B_j/\mathrm{rad}(M,B_j)[1]
\end{eqnarray*}
where $B_j$ is a direct summand of the Bongartz completion of $M$ in $J(P_j^c)$. As above, we conclude that $B_j = M_j'$ and hence
$$\mathcal{E}_{M\sqcup P_j^c[1]}(P_j[1]) = \left(M_j'/\mathrm{rad}(M,M_j')\right)[1].$$
As in the previous cases, we observe that $S$ is a brick in $J(M\sqcup P_j^c)$ and hence $\mathcal{E}_{M\sqcup P_j^c[1]}(P_j[1]) = \mathsf{br}(S[1])$.
\end{proof}

As a corollary, we conclude our first main theorem (Theorem A in the introduction).

\begin{thm}
\label{thm:cubical}
The category $\mathfrak{W}(\Lambda)$ is cubical.
\end{thm}

\begin{proof}
	We use the rank function $\mathsf{rk}[U] := \mathsf{rk}(U)$. By definition, $\mathsf{rk}[U] + \mathsf{rk}[V] = \mathsf{rk}([U]\circ[V])$ when $[U]$ and $[V]$ are composable. Conditions (b) and (c) in the definition of a cubical category are shown in Corollary \ref{cor:cubeembedding}.
	
	Now let $[U]: W \rightarrow W'$ be a morphism in $\mathfrak{W}(\Lambda)$. By Theorem \ref{thm:bigbijection}, the first factors of $[U]$ are the morphisms $[U_i]:W \rightarrow J_W(U_i)$ corresponding to the indecomposable direct summands of $U$. These determine $U$ uniquely. Thus we need only show the last factors of $[U]$ determine $[U]$. By Lemma \ref{lem:fewerlasts}, we can assume $W' = 0$. Theorem \ref{thm:lastmorphisms} then implies that the last factors of $[U]$ are the morphisms $[U_i]:W_i \rightarrow 0$ for which $\bigsqcup_i \mathsf{br}\inv(U_i)$ is the 2-simple minded collection bijectively determined by $U$. Again, these determine $U$ uniquely. We conclude that the category $\mathfrak{W}(\Lambda)$ is cubical.
\end{proof}

\subsection{Locally CAT(0) Categories}
\label{sec:cat0}

We now wish to examine the topological properties of the category $\mathfrak{W}(\Lambda)$. In particular, we study its \emph{classifying space} $\mathcal{B}\mathfrak{W}(\Lambda)$. The classifying space $\mathcal{B}\mathcal{C}$ of an arbitrary category $\mathcal{C}$ is the geometric realization of a simplicial set known as the \emph{simplicial nerve} of the category. The 0-simplices correspond to the objects of $\mathcal{C}$ and the nondegenerate $k$-simplices correspond to chains of composable nonidentity morphisms $(X_0\xrightarrow{f_1} X_1 \xrightarrow{f_2}\cdots\xrightarrow{f_k} X_k)$ in $\mathcal{C}$. A more rigorous definition can be found in \cite[Sec. 3.4]{igusa_signed}. Following \cite{igusa_signed} in the hereditary case, we often refer to $\mathcal{B}\mathfrak{W}(\Lambda)$ as the \emph{picture space} of $\Lambda$.

In general, a connected topological space $X$ is called a $K(\pi,1)$ if it satisfies any of the following equivalent conditions:
\begin{enumerate}
	\item The homotopy groups of $X$ above degree one are all trivial.
	\item The universal cover of $X$ is contractible.
	\item The cohomology of $X$ with arbitrary coefficients is isomorphic to the cohomology of its fundamental group.
\end{enumerate}
Such spaces form a subclass of \emph{Eilenberg MacLane spaces} (see e.g. \cite[Sec. 1.B, x4.2]{hatcher_algebraic}). There is a general recipe for building a $K(\pi,1)$ for an arbitrary finitely presented group $G$ (meaning a $K(\pi,1)$ with fundamental group $G$), but in general the result is an infinite-dimensional CW-complex. Therefore, it is interesting to consider both whether a group admits a finite-dimensional $K(\pi,1)$ and whether a finite CW-complex is a $K(\pi,1)$ for its fundamental group.

An important example of $K(\pi,1)$'s are spaces which are \emph{locally $\mathrm{CAT}(0)$}, also referred to as \emph{non-positively curved}. Informally, a $\mathrm{CAT}(0)$ space is a metric space whose geodesic triangles are ``no fatter than Euclidean triangles", and a locally $\mathrm{CAT}(0)$ space is a space with a $\mathrm{CAT}(0)$ universal cover. Such spaces often admit well-behaved group actions by isometries, and are a frequent object of study in geometric group theory. In particular, in the celebrated work \cite{gromov_hyperbolic}, Gromov shows that for spaces known as \emph{cube complexes}, being locally $\mathrm{CAT}(0)$ is equivalent to a local combinatorial condition. Informally, a cube complex is a topological space built from cubes (of varying dimensions) in an analogous way to how a simplicial complex is built from simplices. We refer readers to the first two chapters of \cite{wise_riches} for more details about this story.

Returning to our context, the definition of a cubical category is made so that the classifying space of a cubical category is a cube complex. In previous work, the second author then deduced the following by interpreting the results of Gromov categorically.

\begin{prop}\cite[Prop. 3.4, Prop. 3.7]{igusa_category}
\label{prop:npc}
Suppose $\mathcal{C}$ is a cubical category. Then the following additional properties are sufficient for the classifying space $\mathcal{BC}$ to be locally $\mathrm{CAT}(0)$ and thus a $K(\pi,1)$.
\begin{enumerate}[label=\upshape(\alph*)]
\item A set of $k$ rank 1 morphisms $\{f_i:X \rightarrow Y_i\}$ forms the first factors of a rank $k$ morphism if and only if each pair $\{f_i, f_j\}$ forms the set of first factors of a rank 2 morphism. In other words, first factors are given by pairwise compatibility conditions.
\item  A set of $k$ rank 1 morphisms $\{f_i:X_i \rightarrow Y\}$ forms the last factors of a rank $k$ morphism if and only if each pair $\{f_i, f_j\}$ forms the set of last factors of a rank 2 morphism. In other words, last factors are given by pairwise compatibility conditions.
\item There is a \emph{faithful group functor} $g:\mathcal{C}\rightarrow G$ for some group $G$; that is, a faithful functor $g:\mathcal{C}\rightarrow G$, where $G$ is considered as a groupoid with one object.
\end{enumerate}
\end{prop}

It is shown in \cite{igusa_category} that $\mathfrak{W}(KQ)$ satisfies the hypotheses of Proposition \ref{prop:npc} when $Q$ is the quiver $A_n$ with straight orientation. For a general $\mathfrak{W}(\Lambda)$, condition (a) follows directly from the definition of a $\tau$-rigid pair. Conditions (a) and (b) are shown in \cite{igusa_signed} in the case that $\Lambda$ is hereditary. In Section \ref{sec:nakayama}, we show that condition (b) also holds in the case that $\Lambda$ is Nakayama. In the Nakayama case, we also show that condition (c) holds in Section \ref{sec:pi1}. We end this section with the following lemma, similar to Lemma \ref{lem:fewerlasts}.

\begin{lem}
\label{lem:fewerlasts2}
The following are equivalent for the category $\mathfrak{W}(\Lambda)$.
\begin{enumerate}
	\item The last factors of all morphisms are given by pairwise compatibility conditions.
	\item The last factors of all morphisms with target 0 are given by pairwise compatibility conditions.
	\item The completable semibrick pairs for $\Lambda$ are given by pairwise compatibility conditions.
\end{enumerate}
\end{lem}

\begin{proof}
We first show that (2) implies (3). Assume (2) and let $\mathcal{Y}$ be a semibrick pair for $\Lambda$ for which every semibrick pair $\mathcal{Y}'\subseteq \mathcal{Y}$ with $|\mathcal{Y}'| = 2$ is completable. Let $\mathcal{L} = \{[\mathsf{br}(X)]|X \in \mathcal{Y}\}$. For all $\mathcal{L}' \subseteq \mathcal{L}$ with $|\mathcal{L}'| = 2$, Theorem \ref{thm:lastmorphisms} then implies that there exists a wide subcategory $W'$ and a morphism $[U']: W'\rightarrow 0$ in $\mathfrak{W}(\Lambda)$ with $\mathsf{last}[U'] = \mathcal{L}'$. By (2), there thus exists a wide subcategory $W$ and a morphism $[U]:W\rightarrow 0$ with $\mathsf{last}[U] = \mathcal{L}$. Now choose any morphism $[V]:\mathsf{mod}\Lambda \rightarrow W$ and write $[V'] = [U]\circ [V]$. Applying Theorem \ref{thm:lastmorphisms} again, it follows that $\mathcal{Y}$ is contained in the 2-simple minded collection $\mathcal{X}(V')$.

The argument that (3) implies (2) is similar and it is clear that (1) implies (2). Thus it only remains to show that (2) implies (1).

Assume (2) and let $\{[U_i]:W_i\rightarrow W\}$ be a set of rank one morphisms in $\mathfrak{W}(\Lambda)$ which are pairwise compatible as last factors. That is, for $i \neq j$ there exists a wide subcategory $W_{i,j} = W_{j,i}$ and a rank two object $U_i'\sqcup U_{j}' \in \str(W_{i,j})$ so that $J_{W_{i,j}}(U_i') = W_j$, $\mathcal{E}^{W_{i,j}}_{U_i'}(U_j') = U_j$, and the analogous equalities hold when the roles of $i$ and $j$ are interchanged. In particular, by Theorem \ref{thm:littlebijections}(c) we have
$$\mathcal{E}_{U_j}^{W_j} \circ \mathcal{E}_{U_i'}^{W_{i,j}} = \mathcal{E}_{U_j'\sqcup U_i'}^{W_{i,j}} = \mathcal{E}_{U_i}^{W_i} \circ \mathcal{E}_{U_j'}^{W_{i,j}}.$$

Now choose any morphism $[V]:W \rightarrow 0$ and decompose $V \cong V_1\sqcup\cdots \sqcup V_{m}$ as a direct sum of indecomposable objects. As in the proof of Lemma \ref{lem:fewerlasts}, for each $i$ we denote $V^i = \left(\mathcal{E}_{U_i}^{W_i}\right)^{-1}(V)$. We then claim that $$\mathcal{L} = \left\{\left[\mathcal{E}^{W_i}_{V^i}(U_i)\right]\right\} \sqcup \mathsf{last}[V] = \left\{\left[\mathcal{E}^{W_i}_{V^i}(U_i)\right]\right\} \sqcup \left\{\left[\mathcal{E}^{W}_{V_k^c}(V_k)\right]\right\}.$$
is a set of pairwise compatible (as last factors) morphisms with target 0. For example, consider some $U_i \neq U_j$ and denote
$V^{i,j} := \left(\mathcal{E}_{U_i'\sqcup U_j'}^{W_{i,j}}\right)^{-1}(V) = \left(\mathcal{E}_{U_j'}^{W_{i,j}}\right)^{-1}(V^i).$ Then there is a rank two morphism $$[L]:=\left[\mathcal{E}_{V^{i,j}}^{W_{i,j}}(U_i'\sqcup U_j')\right]: J_{W_{i,j}}(V^{i,j}) \rightarrow 0.$$ 
Moreover, denote $U_j'' = \mathcal{E}_{V^{i,j}}^{W_{i,j}} (U_j')$. Theorem \ref{thm:littlebijections}(c) then implies that
\begin{eqnarray*}
	\mathcal{E}_{U_j''}^{J_{W_{i,j}}(V^{i,j})} \circ \mathcal{E}_{V^{i,j}}^{W_{i,j}} (U_i') &=& \mathcal{E}_{V^{i,j}\sqcup U_j'}^{W_{i,j}}(U_i')\\
	&=& \mathcal{E}^{W_i}_{V^i}\circ \mathcal{E}_{U_j'}^{W_{i,j}}(U_i')\\
	&=& \mathcal{E}^{W_i}_{V^i}(U_i).
\end{eqnarray*}
In particular, we have
$$\mathsf{last}[L] = \left\{\left[\mathcal{E}_{V^i}^{W_i}(U_i)\right],\left[\mathcal{E}_{V^i}^{W_j}(U_j)\right]\right\}.$$
The compatibility of the other pairs of elements in $\mathcal{L}$ follows by analogous reasoning. Therefore by (2), there is a wide subcategory $W'$ and a morphism $[V']:W'\rightarrow 0$ so that $\mathcal{L} = \mathsf{last}[V']$. Moreover, there exists $U \in \mathsf{add}(V')$ so that $[V'] = [V]\circ [U]$. It follows that $\{[U_i]\}$ is the set of last morphisms of $[U]$.
\end{proof}

\section{2-Simple Minded Collections for Nakayama Algebras}
\label{sec:nakayama}

We recall that $\Lambda$ is called a \emph{Nakayama algebra} if every indecomposable $\Lambda$-module has a unique composition series. It is well-known \cite[Thm.~V.3.2]{assem_elements} that a basic connected algebra is Nakayama if and only if its quiver is one of the following.

\begin{center}
\begin{tikzpicture}
\node at (0,1) {};
\begin{scope}[shift={(-5,0)}]
	\node at (3.1,0) {1};
	\node at (2.1,0) {2};
	\node at (0,0) {$n$};
	\node at (1.1,0) {$\cdots$};
	\node at (1.55,-1) {$A_n$};
	\begin{scope}[decoration={
	markings,
	mark=at position 1.0 with {\arrow[scale=1.3]{>}}}
	]
		\draw[postaction={decorate}] (.75,0)--(.25,0);
		\draw[postaction={decorate}] (1.85,0)--(1.35,0);
		\draw[postaction={decorate}] (2.85,0)--(2.35,0);
	\end{scope}
\end{scope}

	\node at (3.1,0) {1};
	\node at (2.1,0) {2};
	\node at (0,0) {$n$};
	\node at (1.1,0) {$\cdots$};
	\node at (1.55,-1) {$\Delta_n$};
	\begin{scope}[decoration={
	markings,
	mark=at position 1.0 with {\arrow[scale=1.3]{>}}}
	]
		\draw[postaction={decorate}] (.75,0)--(.25,0);
		\draw[postaction={decorate}] (1.85,0)--(1.35,0);
		\draw[postaction={decorate}] (2.85,0)--(2.35,0);
		\draw[postaction={decorate}] (0,-0.25)--(0,-0.5)--(3.1,-0.5)--(3.1,-0.25);
	\end{scope}
\end{tikzpicture}
\end{center}

Adachi \cite{adachi_classification} gave a combinatorial interpretation of the $\tau$-tilting modules and support $\tau$-tilting modules with no projective direct summands for Nakayama algebras. Asai \cite{asai_semibricks} gave an enumeration of the semibricks of Nakayama algebras. In this section, we give a combinatorial interpretation of the 2-simple minded collections of Nakayama algebras and use this to prove the following theorem.

\begin{thm}
\label{thm:nakayamapairwise}
Every singly left mutation compatible semibrick pair for a Nakayama algebra is completable. In particular, the 2-simple minded collections for a Nakayama algebra are given by pairwise compatibility conditions.
\end{thm}

In order to prove Theorem \ref{thm:nakayamapairwise}, we can assume without loss of generality that our Nakayama algebras are connected. Thus for the remainder of this section, we assume that $\Lambda$ is connected and Nakayama. We fix $n:=rk(\Lambda)$,

\subsection{Bricks for Nakayama Algebras}

We begin by recalling the description of indecomposable $\Lambda$-modules given in the remarks following \cite[Prop 2.2]{adachi_classification}. For $M \in \mathsf{mod}\Lambda$, we denote the Loewy length of $M$ by $l(M)$. Following Adachi's notation, we denote by $i_n$ the unique integer $1\leq k \leq n$ such that $i-k\in n\mathbb{Z}$. We then define
$$[i,j]_n := \begin{cases}
	[i_n,j_n] \textnormal{ if }i_n \leq j_n\\
	[1,j_n]\cup[i_n,n] \textnormal{ if } i_n > j_n
\end{cases}$$
and likewise for $(i,j)_n$, etc. For example, $(1,4)_5 = \{2,3\}$ and $(4,1)_5 = \{5\}$.

\begin{prop}\cite[Section V]{assem_elements},\cite[Prop. 2.2]{adachi_classification}
There is a bijection
$$M:\{(i,j)\in\mathbb{Z}^2|0 < i \leq n,0 < j \leq l(P_i)\}\leftrightarrow\mathsf{ind}(\mathsf{mod}\Lambda)$$
which sends the pair $(i,j)$ to the module $P_i/\mathrm{rad}^jP_i$. The inverse is given by sending an indecomposable module $M$ with projective cover $P_i$ to the pair $(i,l(M))$.
\end{prop}

Although statements (b) and (c) in the following lemma do not appear in \cite{adachi_classification}, they follow immediately from (a) by writing out the composition series.

\begin{lem}\cite[Lem. 2.4]{adachi_classification}
\label{lem:morphs}
The module $M(i,j)$ has unique composition series corresponding to the composition factors
$$S_i,S_{(i+1)_n},\ldots,S_{(i+j-1)_n}.$$
In particular, we have
\begin{enumerate}[label=\upshape(\alph*)]
\item $\mathrm{Hom}(M(i,j),M(k,l)) \neq 0$ if and only if $i \in [k, k+l-1]_n$ and $(k+l-1)_n \in [i,i+j-1]_n.$
\item There is a monomorphism $M(i,j)\hookrightarrow M(k,l)$ if and only if in addition to (a) we have $(i+j-1)_n = (k+l-1)_n$ and $j \leq l$.
\item There is an epimorphism $M(i,j) \twoheadrightarrow M(k,l)$ if and only if in addition to (a) we have $i = k$ and $j \geq l$.
\end{enumerate}
\end{lem}

This leads to the following characterization of $\mathsf{brick}\Lambda$.

\begin{prop}
\label{prop:nakayamabricks}
Let $M \in \mathsf{ind}(\mathsf{mod}\Lambda)$. Then $M$ is a brick if and only if $l(M) \leq n$.
\end{prop}

\begin{proof}
	Let $M$ be an indecomposable module. Lemma 3.3 implies that $M$ is a brick if and only if each simple module occurs at most once in the unique composition series of $M$. This is true if and only if $l(M) \leq n$.
\end{proof}

As a consequence, we observe the following.

\begin{cor}
\label{cor:monoepibricks}
Let $M,N\in\mathsf{brick}\Lambda$. Then $\mathrm{dim}\,\mathrm{Hom}(M,N) \leq 1$. Moreover, if $\mathrm{dim}\,\mathrm{Hom}(M,N)=1$, then every nonzero morphism $M\rightarrow N$ is a minimal left $\mathsf{Filt} N$-approximation.
\end{cor}

\begin{proof}
We prove the result only for the case that $l(N) = n$. The other cases are trivial, as if $l(N) < n$, then $N$ has no self-extensions. Now if the quiver of $\Lambda$ is $A_n$, then $\Lambda\cong KA_n$ and $N\cong P_1$ has no self extensions. Thus we assume without loss of generality that the quiver of $\Lambda$ is $\Delta_n$ and
\begin{center}
\begin{tikzpicture}
	\node at (-1,0) {$N=$};
	\node at (3.1,0) {$K$};
	\node at (2.1,0) {$K$};
	\node at (0,0) {$K$};
	\node at (1.1,0) {$\cdots$};
	\begin{scope}[decoration={
	markings,
	mark=at position 1.0 with {\arrow[scale=1.3]{>}}}
	]
		\draw[postaction={decorate}] (.75,0)--(.25,0) node [midway,anchor=south]{1};
		\draw[postaction={decorate}] (1.85,0)--(1.35,0) node [midway,anchor=south]{1};
		\draw[postaction={decorate}] (2.85,0)--(2.35,0) node [midway,anchor=south]{1};
		\draw[postaction={decorate}] (0,-0.25)--(0,-0.5)--(3.1,-0.5)--(3.1,-0.25);
	\end{scope}
	\node at (1.55,-0.75) {0};
\end{tikzpicture}
\end{center}
Thus up to isomorphism, any indecomposable $E \in \mathsf{Filt} N$ has the form
\begin{center}
\begin{tikzpicture}
	\node at (-1,0) {$E=$};
	\node at (3.2,0) {$K^m$};
	\node at (2.05,0) {$K^m$};
	\node at (-.1,0) {$K^m$};
	\node at (1.05,0) {$\cdots$};
	\begin{scope}[decoration={
	markings,
	mark=at position 1.0 with {\arrow[scale=1.3]{>}}}
	]
		\draw[postaction={decorate}] (.75,0)--(.25,0) node [midway,anchor=south]{1};
		\draw[postaction={decorate}] (1.75,0)--(1.25,0) node [midway,anchor=south]{1};
		\draw[postaction={decorate}] (2.85,0)--(2.4,0) node [midway,anchor=south]{1};
		\draw[postaction={decorate}] (-.1,-0.25)--(-.1,-0.5)--(3.1,-0.5)--(3.1,-0.25);
	\end{scope}
	\node at (1.55,-0.75) {$L$};
\end{tikzpicture}
\end{center}
where $L_{i,j} = \delta_{i-1,j}$ for $i,j \in \{1,\ldots,m\}$. Now in order for $\mathrm{Hom}(M,N)\neq 0$, we see that there exists an integer $p \in [1, n]$ such that $M$ is supported on $[p,n]$ (in the sense that the simple module $S_i$ is a composition factor of $M$ for all $i \in [p,n]$). For example, we could have
\begin{center}
\begin{tikzpicture}
	\node at (-1,0) {$M=$};
	\node at (3,0) {$K$};
	\node at (5,0) {$K$};
	\node at (0,0) {$K$};
	\node at (2,0) {$K$};
	\node at (1.05,0) {$\cdots$};
	\node at (4.05,0) {$\cdots$};
	\begin{scope}[decoration={
	markings,
	mark=at position 1.0 with {\arrow[scale=1.3]{>}}}
	]
		\draw[postaction={decorate}] (.75,0)--(.25,0) node [midway,anchor=south]{1};
		\draw[postaction={decorate}] (1.75,0)--(1.25,0) node [midway,anchor=south]{1};
		\draw[postaction={decorate}] (2.75,0)--(2.3,0) node [midway,anchor=south]{0};
		\draw[postaction={decorate}] (0,-0.25)--(0,-0.5)--(5,-0.5)--(5,-0.25);
		\draw[postaction={decorate}] (3.75,0)--(3.25,0) node [midway,anchor=south]{1};
		\draw[postaction={decorate}] (4.75,0)--(4.25,0) node [midway,anchor=south]{1};
	\end{scope}
	\node at (2.55,-0.75) {1};
\end{tikzpicture}
\end{center}
In any case, we see that $\mathrm{dim}\,\mathrm{Hom}(M,N) = 1$. Moreover, for $E \in \mathsf{Filt}(N)$ indecomposable, we observe that $\mathrm{dim}\,\mathrm{Hom}(M,E) = 1$ as well. The result follows immediately.
\end{proof}

We are now ready to construct our combinatorial model. For positive integers $n,l(1),\ldots,l(n)$, we denote $D(n,l(1),l(2),\ldots,l(n))$ the punctured disk $D^2\setminus\{0\}$ with $n$ marked points on its boundary, labeled counterclockwise $1,2,\ldots,n$. The value of $l(i)$ is called the \emph{length} of the marked point labeled $i$. We further assume 
\begin{enumerate}
	\item For all $i \leq n$, we have $l(i)\geq l((i-1)_n)-1$.
	\item For all $i <n$, we have $l(i) > 1$.
\end{enumerate}
Under these assumptions, we see that $l(1),\ldots,l(n)$ are the Loewy lengths of the projective modules of some connected Nakayama algebra $\Lambda(l(1),\ldots,l(n))$ whose quiver has $n$ vertices. Conversely, given a Nakayama algebra whose quiver has $n$ vertices, we observe that the Loewy lengths of the projective modules satisfy both (1) and (2). For example, $D(4,3,3,3,3)$ corresponds to the cyclic cluster-tilted algebra of type $D_4$ and $D(3,3,2,1)$ corresponds to $A_3$ with straight orientation.

We now wish to relate $\mathsf{brick}\Lambda(l(1),\ldots,l(n))$ to certain directed paths between the marked points of $D(n,l(1),\ldots,l(n))$. We begin with the following definition.

\begin{define}
For two (not necessarily distinct) marked points $i,j$ of \linebreak$D(n,l(1),\ldots,l(n))$, a directed path $a: i\rightarrow j$ in $D^2\setminus \{0\}$ is called an \emph{arc} if
\begin{enumerate}[label=\upshape(\alph*)]
	\item $a$ is homotopic to the counterclockwise boundary arc $i\rightarrow j$.
	\item $a$ does not intersect itself unless $i=j$, in which case the only intersection occurs at the endpoint.
	\item $l(i)\geq (j-i)_n$.
\end{enumerate}
We call $i$ the \emph{source} of $a$, denoted $s(a)$, and $j$ the \emph{target} of $a$, denoted $t(a)$. We call $(j-i)_n$ the \emph{length} of $a$, denoted $l(a)$, and $[s(a),t(a)]_n$ the \emph{support} of $a$, denoted $\mathrm{supp}(a)$. Condition (c) can then be rephrased as $l(s(a)) \geq l(a)$. We denote the set of homotopy classes of arcs of $D(n,l(1),\ldots l(n))$ by $\mathsf{Arc}(n,l(1),\ldots,l(n))$.
\end{define}

The following is immediate.

\begin{prop}
There is a bijection $$M: \mathsf{Arc}(n,l(1),\ldots,l(n)) \rightarrow \mathsf{brick}\Lambda(l(1),\ldots,l(n))$$ given by sending an arc $a$ to the module $M(a) := M(s(a),l(a))$.
\end{prop}

We now consider the various ways that two distinct arcs can (minimally) intersect and the corresponding morphisms and extensions between modules. As we are only interested in characterizing sets of singly left mutation compatible semibrick pairs, the details regarding extensions are intentionally incomplete. In the diagrams accompanying Lemmas \ref{lem:archomslong} and \ref{lem:archoms}, the arc $a_1$ is always shown as solid and the arc $a_2$ is always shown as dashed. We remark that these figures provide only examples, e.g. figures with two arcs each having length $n$, etc. are not shown.

\begin{lem}
\label{lem:archomslong}
Let $a_1,a_2 \in \mathsf{Arc}(n,l(1)\ldots,l(n))$ with $l(a_1) = n$.
\begin{enumerate}[label=\upshape(\alph*)]
	\item If $a_1$ and $a_2$ intersect at two different points, then both $\mathrm{Hom}(M(a_1),M(a_2))$ and $\mathrm{Hom}(M(a_2),M(a_1))$ are nonzero.
	\item If $a_1$ and $a_2$ intersect at only one point, then $\mathrm{Hom}(M(a_1),M(a_2))$ contains an epimorphism if and only if $s(a_1) = s(a_2)$. In this case, we have $\mathrm{Hom}(M(a_2),M(a_1)) = 0 = \mathrm{Ext}(M(a_2),M(a_1))$. We remark that $a_1$ appears counterclockwise from $a_2$.
	\item If $a_1$ and $a_2$ intersect at only one point, then $\mathrm{Hom}(M(a_2),M(a_1))$ contains a monomorphism if and only if $t(a_1) = t(a_2)$. In this case, we have $\mathrm{Hom}(M(a_1),M(a_2)) = 0 = \mathrm{Ext}(M(a_1),M(a_2))$. We remark that $a_2$ appears counterclockwise from $a_1$.
	\item Otherwise $a_1$ and $a_2$ do not intersect and are Hom-Ext orthogonal.
	\end{enumerate}
\end{lem}

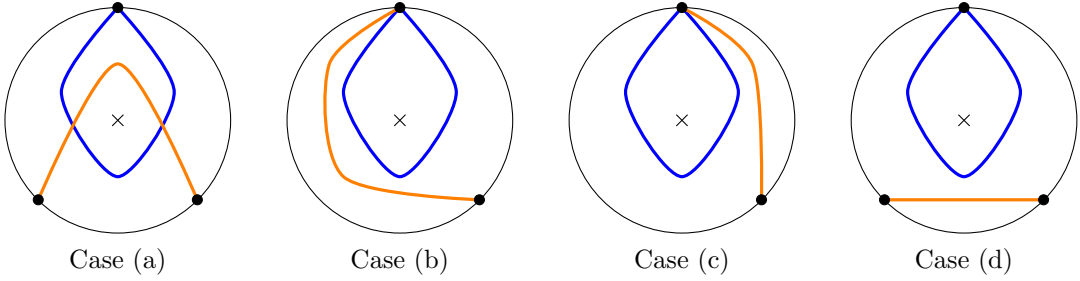
\begin{center}
\begin{tikzpicture}[scale=0.5]
	\draw (0,0) circle[radius=2cm];
	\draw (-.1,-.1) to (.1,.1);
	\draw (-.1,.1) to (.1,-.1);
		\draw[very thick] plot [smooth] coordinates{(0,2) (-1,0.5) (0,-1) (1,0.5) (0,2)};
		\draw[very thick,dashed] plot [smooth] coordinates{(-1.41,-1.41)(0,1)(1.41,-1.41)};
		\node at (0,2) [draw,fill,circle,scale=0.4]{};
		\node at (-1.41,-1.41) [draw,fill,circle,scale=0.4]{};
		\node at (1.41,-1.41) [draw,fill,circle,scale=0.4]{};
	\node at (0,-2.5) {Case (a)};
\begin{scope}[shift={(5,0)}]
	\draw (0,0) circle[radius=2cm];
	\draw (-.1,-.1) to (.1,.1);
	\draw (-.1,.1) to (.1,-.1);
		\draw[very thick] plot [smooth] coordinates{(0,2) (-1,0.5) (0,-1) (1,0.5) (0,2)};
		\draw[very thick,dashed] plot [smooth] coordinates{(0,2)(-1.25,1)(-1,-1)(1.41,-1.41)};
		\node at (0,2) [draw,fill,circle,scale=0.4]{};
		\node at (1.41,-1.41) [draw,fill,circle,scale=0.4]{};
	\node at (0,-2.5) {Case (b)};
\end{scope}
\begin{scope}[shift={(10,0)}]
	\draw (0,0) circle[radius=2cm];
	\draw (-.1,-.1) to (.1,.1);
	\draw (-.1,.1) to (.1,-.1);
		\draw[very thick] plot [smooth] coordinates{(0,2) (-1,0.5) (0,-1) (1,0.5) (0,2)};
		\draw[very thick,dashed] plot [smooth] coordinates{(1.41,-1.41) (1.25,1) (0,2)};
		\node at (0,2) [draw,fill,circle,scale=0.4]{};
		\node at (1.41,-1.41) [draw,fill,circle,scale=0.4]{};
		\node at (0,-2.5) {Case (c)};
\end{scope}
\begin{scope}[shift={(15,0)}]
	\draw (0,0) circle[radius=2cm];
	\draw (-.1,-.1) to (.1,.1);
	\draw (-.1,.1) to (.1,-.1);
		\draw[very thick] plot [smooth] coordinates{(0,2) (-1,0.5) (0,-1) (1,0.5) (0,2)};
		\draw[very thick,dashed] plot [smooth] coordinates{(1.41,-1.41) (-1.41,-1.41)};
		\node at (0,2) [draw,fill,circle,scale=0.4]{};
		\node at (1.41,-1.41) [draw,fill,circle,scale=0.4]{};
		\node at (-1.41,-1.41) [draw,fill,circle,scale=0.4]{};
	\node at (0,-2.5) {Case (d)};
\end{scope}
\end{tikzpicture}
\end{center}

\begin{proof}
All results follow from Lemma \ref{lem:morphs}. We prove (b) in detail as an example.

Suppose $a_1$ and $a_2$ intersect at only one point. In the notation of Lemma \ref{lem:morphs}, we have $M(a_1) = M(s(a_1),n)$ and $M(a_2) = M(s(a_2),l(a_2))$. Thus by Lemma \ref{lem:morphs}(c), there is an epimorphism $M(a_1) \twoheadrightarrow M(a_2)$ if and only if $s(a_1) = s(a_2)$ and $n \geq l(a_2)$.

Thus suppose $s(a_1) = s(a_2)$. Then since $a_1 \neq a_2$, it follows that
$$(t(a_1)-1)_n = (s(a_1)+n-1)_n \notin [s(a_2),s(a_2)+l(a_2)-1]_n = [s(a_2),t(a_2)-1]_n,$$
so by Lemma \ref{lem:morphs}(a), we have $\mathrm{Hom}(M(a_2),M(a_1)) = 0$.

Now let $M(a_1) \hookrightarrow E \twoheadrightarrow M(a_2)$ be an exact sequence. We first observe that $(s(a_1) -1)_n \notin [s(a_1),s(a_1)+l(a_2)-1]_n$. Thus, let $M$ be the (necessarily unique) indecomposable direct summand of $E$ which contains $S_{(s(a_1) -1)_n}$ in its composition series. It follows that there exists a unique indecomposable $N = N(i,j)$ in the image of the induced map $M(a_1) \rightarrow M$ which also contains $S_{(s(a_1) -1)_n}$ in its composition series. By definition, we have $M(a_1) \twoheadrightarrow N$. Thus by Lemma \ref{lem:morphs}(c), we have $i = s(a_1)$ and $j \leq n$ (and hence $j = n$ and $N \cong M(a_1)$). This means that $M$ contains $M(a_1)$ as a submodule. Lemma \ref{lem:morphs}(b) then implies that either $S_{(s(a_1) -1)_n}$ occurs twice in the composition series for $M$ or $M \cong N$. However, we know that $S_{(s(a_1) -1)_n}$ occurs only once as a composition factor of $E$. We conclude that $E \cong M(a_1) \sqcup M(a_2)$ and $\mathrm{Ext}(M(a_2),M(a_1)) = 0$, as desired.
\end{proof}

\begin{lem}
\label{lem:archoms}
Let $a_1,a_2 \in \mathsf{Arc}(n,l(1),\ldots,l(n))$ with $l(a_1),l(a_2) < n$.
\begin{enumerate}[label=\upshape(\alph*)]
	\item If $a_1$ and $a_2$ do not intersect, then $M(a_1)$ and $M(a_2)$ are Hom-Ext orthogonal.
	\item If $t(a_1) = s(a_2)$ and $a_1$ does not otherwise intersect $a_2$ then $M(a_1)$ and $M(a_2)$ are Hom orthogonal and $\mathrm{Ext}(M(a_2),M(a_1))=0$. Moreover,
		\begin{enumerate}[label=\upshape(\roman*)]
			\item if $l(a_1) + l(a_2) \leq l(s(a_1))$ then $\mathrm{dim}\mathrm{Ext}(M(a_1),M(a_2)) = 1$ and there is an exact sequence $M(a_2)\hookrightarrow M(a_1a_2)\twoheadrightarrow M(a_1).$
			\item Otherwise $M(a_1)$ and $M(a_2)$ are Hom-Ext orthogonal.							\end{enumerate}
	\item If $t(a_2) = s(a_1)$ and $t(a_1) = s(a_2)$ then $M(a_1)$ and $M(a_2)$ are Hom orthogonal. Moreover,
		\begin{enumerate}[label=\upshape(\roman*)]
			\item  if $l(a_1)+l(a_2) \leq l(s(a_1))$ then $\mathrm{dim}\mathrm{Ext}(M(a_1),M(a_2)) = 1$ and there is an exact sequence $M(a_2)\hookrightarrow M(a_1a_2)\twoheadrightarrow M(a_1).$
			\item if $l(a_1)+l(a_2) \leq l(s(a_2))$ then $\mathrm{dim}\mathrm{Ext}(M(a_2),M(a_1)) = 1$ and there is an exact sequence $M(a_1)\hookrightarrow M(a_2a_1)\twoheadrightarrow M(a_2).$
			\item otherwise $M(a_1)$ and $M(a_2)$ are Hom-Ext orthogonal.
		\end{enumerate}
	\item If $s(a_1) = s(a_2) < t(a_2) < t(a_1)$ and $a_1, a_2$ do not otherwise intersect (in this case, $a_1$ appears counterclockwise from $a_2$) then $\mathrm{Hom}(M(a_1),M(a_2))$ contains an epimorphism. Moreover, $\mathrm{Hom}(M(a_2),M(a_1)) = 0$ and\linebreak $\mathrm{Ext}(M(a_2),M(a_1))=0$.
	\item If $s(a_2)<s(a_1)<t(a_1)=t(a_2)$ and $a_1, a_2$ do not otherwise intersect (in this case, $a_1$ appears counterclockwise from $a_2$) then $\mathrm{Hom}(M(a_1),M(a_2))$ contains a monomorphism. Moreover, $\mathrm{Hom}(M(a_2),M(a_1)) = 0$, and\linebreak $\mathrm{Ext}(M(a_2),M(a_1)) = 0$.
	\item If $a_1$ and $a_2$ intersect once in the interior of the disk then one of \linebreak$\mathrm{Hom}(M(a_1),M(a_2))$ and $\mathrm{Hom}(M(a_1),M(a_2))$ is nonempty and contains no monomorphism or epimorphism.
	\item If $a_1$ and $a_2$ intersect twice in the interior of the disk, then both \linebreak$\mathrm{Hom}(M(a_1),M(a_2))$ and $\mathrm{Hom}(M(a_2),M(a_1))$ are both nonzero.
\end{enumerate}
\end{lem}

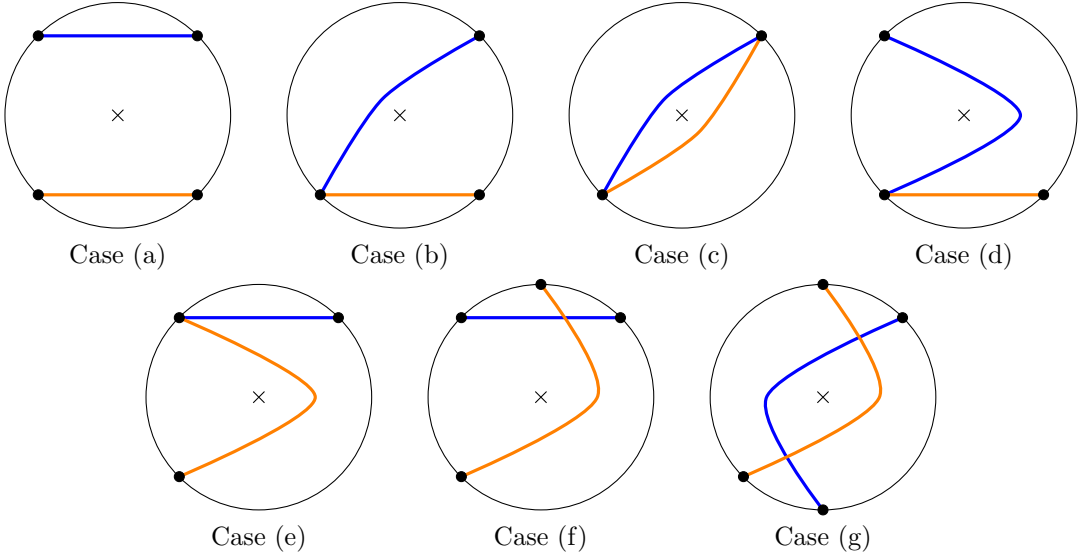
\begin{center}
\begin{tikzpicture}[scale=0.5]
	\draw (0,0) circle[radius=2cm];
	\draw (-.1,-.1) to (.1,.1);
	\draw (-.1,.1) to (.1,-.1);
		\draw[very thick] plot [smooth] coordinates{(1.41,1.41)(-1.41,1.41)};
		\draw[very thick,dashed] plot [smooth] coordinates{(-1.41,-1.41)(1.41,-1.41)};
		\node at (1.41,1.41) [draw,fill,circle,scale=0.4]{};
		\node at (-1.41,1.41) [draw,fill,circle,scale=0.4]{};
		\node at (-1.41,-1.41) [draw,fill,circle,scale=0.4]{};
		\node at (1.41,-1.41) [draw,fill,circle,scale=0.4]{};
	\node at (0,-2.5) {Case (a)};
\begin{scope}[shift={(5,0)}]
	\draw (0,0) circle[radius=2cm];
	\draw (-.1,-.1) to (.1,.1);
	\draw (-.1,.1) to (.1,-.1);
		\draw[very thick] plot [smooth] coordinates{(1.41,1.41)(-0.3,0.3)(-1.41,-1.41)};
		\draw[very thick,dashed] plot [smooth] coordinates{(-1.41,-1.41)(1.41,-1.41)};
		\node at (1.41,1.41) [draw,fill,circle,scale=0.4]{};
		\node at (-1.41,-1.41) [draw,fill,circle,scale=0.4]{};
		\node at (1.41,-1.41) [draw,fill,circle,scale=0.4]{};
	\node at (0,-2.5) {Case (b)};
\end{scope}
\begin{scope}[shift={(10,0)}]
	\draw (0,0) circle[radius=2cm];
	\draw (-.1,-.1) to (.1,.1);
	\draw (-.1,.1) to (.1,-.1);
		\draw[very thick] plot [smooth] coordinates{(1.41,1.41)(-0.3,0.3)(-1.41,-1.41)};
		\draw[very thick,dashed] plot [smooth] coordinates{(-1.41,-1.41)(0.3,-0.3)(1.41,1.41)};
		\node at (1.41,1.41) [draw,fill,circle,scale=0.4]{};
		\node at (-1.41,-1.41) [draw,fill,circle,scale=0.4]{};
	\node at (0,-2.5) {Case (c)};
\end{scope}
\begin{scope}[shift={(15,0)}]
	\draw (0,0) circle[radius=2cm];
	\draw (-.1,-.1) to (.1,.1);
	\draw (-.1,.1) to (.1,-.1);
		\draw[very thick] plot [smooth] coordinates{(-1.41,-1.41)(1,0)(-1.41,1.41)};
		\draw[very thick,dashed] plot [smooth] coordinates{(-1.41,-1.41)(1.41,-1.41)};
		\node at (-1.41,1.41) [draw,fill,circle,scale=0.4]{};
		\node at (-1.41,-1.41) [draw,fill,circle,scale=0.4]{};
		\node at (1.41,-1.41) [draw,fill,circle,scale=0.4]{};
	\node at (0,-2.5) {Case (d)};
\end{scope}
\begin{scope}[shift={(2.5,-5)}]
	\draw (0,0) circle[radius=2cm];
	\draw (-.1,-.1) to (.1,.1);
	\draw (-.1,.1) to (.1,-.1);
		\draw[very thick] plot [smooth] coordinates{(1.41,1.41)(-1.41,1.41)};
		\draw[very thick,dashed] plot [smooth] coordinates{(-1.41,-1.41)(1,0)(-1.41,1.41)};
		\node at (1.41,1.41) [draw,fill,circle,scale=0.4]{};
		\node at (-1.41,1.41) [draw,fill,circle,scale=0.4]{};
		\node at (-1.41,-1.41) [draw,fill,circle,scale=0.4]{};
	\node at (0,-2.5) {Case (e)};
\end{scope}
\begin{scope}[shift={(7.5,-5)}]
	\draw (0,0) circle[radius=2cm];
	\draw (-.1,-.1) to (.1,.1);
	\draw (-.1,.1) to (.1,-.1);
		\draw[very thick] plot [smooth] coordinates{(1.41,1.41)(-1.41,1.41)};
		\draw[very thick,dashed] plot [smooth] coordinates{(-1.41,-1.41)(1,0)(0,2)};
		\node at (1.41,1.41) [draw,fill,circle,scale=0.4]{};
		\node at (-1.41,1.41) [draw,fill,circle,scale=0.4]{};
		\node at (-1.41,-1.41) [draw,fill,circle,scale=0.4]{};
		\node at (0,2) [draw,fill,circle,scale=0.4]{};
	\node at (0,-2.5) {Case (f)};
\end{scope}
\begin{scope}[shift={(12.5,-5)}]
	\draw (0,0) circle[radius=2cm];
	\draw (-.1,-.1) to (.1,.1);
	\draw (-.1,.1) to (.1,-.1);
		\draw[very thick] plot [smooth] coordinates{(1.41,1.41)(-1,0)(0,-2)};
		\draw[very thick,dashed] plot [smooth] coordinates{(-1.41,-1.41)(1,0)(0,2)};
		\node at (1.41,1.41) [draw,fill,circle,scale=0.4]{};
		\node at (0,-2) [draw,fill,circle,scale=0.4]{};
		\node at (-1.41,-1.41) [draw,fill,circle,scale=0.4]{};
		\node at (0,2) [draw,fill,circle,scale=0.4]{};
	\node at (0,-2.5) {Case (g)};
\end{scope}
\end{tikzpicture}
\end{center}

\begin{proof}
As for Lemma \ref{lem:archomslong}, all results follow from Lemma \ref{lem:morphs}.
\end{proof}

We say that an arc $a_1$ is \emph{counterclockwise} from an arc $a_2$ (and hence that $a_2$ is \emph{clockwise} from $a_1$) if $a_1$ intersects $a_2$ as in Lemma \ref{lem:archomslong}(b),(c) or Lemma \ref{lem:archoms}(d),(e).

\subsection{Admissible Arc Patterns}

Recall that an object $\mathcal{X} = \mathcal{S}_p\sqcup \mathcal{S}_n[1] \in \mathcal{D}^b(\mathsf{mod}\Lambda)$ is called a singly left mutation compatible semibrick pair if $\mathcal{S}_p,\mathcal{S}_n\in\mathsf{sbrick}\Lambda$, $\mathrm{Hom}(\mathcal{S}_p,\mathcal{S}_n) = 0 = \mathrm{Hom}(\mathcal{S}_p,\mathcal{S}_n[1])$, and for every $S\in \mathcal{S}_p$ and $S'\in\mathcal{S}_n$ with $\mathrm{Hom}(S',S)\neq0$, every minimal left $\mathsf{Filt} S$ approximation $S'\rightarrow E$ is either mono or epi. By translating this definition into the language of arcs, we arrive at the following.

\begin{define}
A collection $A = A_g \sqcup A_r$ of distinct arcs in $\mathsf{Arc}(n,l(1),\ldots,l(n))$ is an \emph{admissible arc pattern} if it can be drawn in $D(l(1),\ldots,l(n))$ so that, for all $a_1,a_2 \in A$,
	\begin{enumerate}[label=\upshape(\alph*)]
		\item The only intersections between $a_1$ and $a_2$ are on the boundary of the disk.
		\item If $s(a_1) = s(a_2)$ or $t(a_1) = t(a_2)$ and $a_1$ is counterclockwise from $a_2$, then $a_1 \in A_r$ and $a_2 \in A_g$.
		\item If $t(a_1) = s(a_2)$ then at least one of the following holds:
			\begin{enumerate}[label=\upshape(\roman*)]
				\item $a_1 \in A_r$.
				\item $a_2 \in A_g$.
				\item $l(a_1) + l(a_2) > l(s(a_1))$.
			\end{enumerate}
	\end{enumerate}
	We refer to $A_g$ as the set of \emph{green} arcs and $A_r$ as the set of \emph{red} arcs. We say a green (resp. red) arc $a$ is \emph{incompatible} with $A$ if $(A_g\cup\{a\})\sqcup A_r$ (resp. $A_g\sqcup(A_r\cup\{a\})$) is not an admissible arc pattern.
\end{define}

Our convention from now on will be to draw green arcs as solid lines, red arcs as dashed lines, and arbitrary arcs as dotted lines. The above definition is engineered to imply the following result.

\begin{prop}
\label{rem:mutationpairarcs}
There is a bijection between the set of admissible arc patterns for $D(n,l(1),\ldots,l(n))$ and the set of singly left mutation compatible semibrick pairs for $\Lambda(n,l(1),\ldots,l(n))$ given by
$$A \mapsto \mathcal{X}(A) := \{M(a)|a \in A_g\}\sqcup\{M(a)[1]|a \in A_r\}.$$
\end{prop}

\begin{proof}
	Suppose $A = A_g \sqcup A_r$ is an admissible arc pattern. Let $a_1, a_2 \in A$. 
	
	If $a_1,a_2 \in A_g$, then conditions (a) and (b) imply that, without loss of generality, $a_1$ and $a_2$ intersect as in one of Lemmas \ref{lem:archomslong}(d) or \ref{lem:archoms}(a),(b),(c). In all four cases, $M(a_1)$ and $M(a_2)$ are Hom orthogonal. The same argument shows that if $a_1,a_2 \in A_r$ then $M(a_1)$ and $M(a_2)$ are Hom orthogonal.
	
	Now suppose $a_1 \in A_g$ and $a_2 \in A_r$. Conditions (a) and (b) then imply that, up to relabeling, $a_1$ and $a_2$ intersect as in one of Lemmas \ref{lem:archomslong}(b),(c),(d) or \ref{lem:archoms}(a),(b),(c),(d),(e). The fact that if $a_1$ and $a_2$ share a source or target then $a_1$ is counterclockwise from $a_2$ then implies that $\mathrm{Hom}(M(a_1),M(a_2)) = 0$. Moreover, we see from Corollary \ref{cor:monoepibricks} and Lemmas \ref{lem:archomslong} and \ref{lem:archoms} that under this assumption, if $\mathrm{Hom}(M(a_2),M(a_1)) \neq 0$, then every minimal left $\mathsf{Filt} (M(a_2))$-approximation $M(a_1) \rightarrow M(a_2)$ is either mono or epi. Finally, condition (c) implies that in any case, $\mathrm{Ext}(M(a_1),M(a_2)) = 0$. This shows that $\mathcal{X}(A)$ is a singly left mutation compatible semibrick pair.
	
	Now let $\mathcal{X} = \mathcal{S}_p \sqcup \mathcal{S}_n[1]$ be a singly left mutation compatible semibrick pair. Let $A_g = \{a|M(a) \in \mathcal{S}_p\}$ and $A_r = \{a|M(a) \in \mathcal{S}_n\}$. Let $a_1, a_2 \in A$.
	
	Since $M(a_1)$ and $M(a_2)$ are either Hom orthogonal or up to relabeling every nonzero morphism $M(a_1) \rightarrow M(a_2)$ is either mono or epi, Corollary \ref{cor:monoepibricks} and Lemmas \ref{lem:archomslong} and \ref{lem:archoms} imply that either $a_1$ and $a_2$ do not intersect or the only intersections between $a_1$ and $a_2$ are on the boundary of the disk and, up to relabeling, $a_1$ and $a_2$ intersect as in Lemmas \ref{lem:archomslong}(b),(c) or \ref{lem:archoms}(b),(c),(d),(e).
	
	Now suppose $s(a_1) = s(a_2)$ and $a_1$ is counterclockwise from $a_2$ (that is, $a_1$ and $a_2$ intersect as in Lemmas \ref{lem:archomslong}(b) or \ref{lem:archoms}(d)). This means there is a nonzero morphism $M(a_1)\rightarrow M(a_2)$. Therefore $M(a_1)\in\mathcal{S}_n$ and $M(a_2) \in \mathcal{S}_p$, that is, $a_1 \in A_r$ and $a_2 \in A_g$ as desired. A similar argument holds if $t(a_1) = t(a_2)$ and $a_1$ is counterclockwise from $a_2$.
	Finally, suppose that $t(a_1) = s(a_2)$ (that is, $a_1$ and $a_2$ intersect as in Lemmas \ref{lem:archomslong}(b),(c) or \ref{lem:archoms}(b),(c)). Suppose for a contradiction that $a_1 \in A_g$, $a_2 \in A_r$, and $l(a_1) + l(a_2) \leq l(s(a_1))$. This means $M(a_1) \in \mathcal{S}_p, M(a_2) \in \mathcal{S}_n$, and $\mathrm{Ext}(M(a_1),M(a_2)) \neq 0$, a contradiction. We conclude that $A$ is an admissible arc pattern. Moreover, it is clear that $\mathcal{X}(A) = \mathcal{X}$ and these maps are inverse to each other.
\end{proof}

Our next task is to determine which admissible arc patterns correspond to 2-simple minded collections.

\begin{define}
An admissible arc pattern is \emph{maximal} if it is not a proper subset of another admissible arc pattern.
\end{define}

\begin{ex}\label{ex:arcpattern} A maximal admissible arc pattern for $D(3,3,2,1)$, which corresponds to the algebra $KA_3$ (where $A_3$ is given straight orientation), is shown in Figure \ref{fig:arcpattern}. We note that in this example, $\mathcal{X}(A) = 1\sqcup \scriptsize{\begin{matrix}2\\3\end{matrix}\sqcup 3[1]}$ is a 2-simple minded collection.
A complete set of examples for this algebra and for the cyclic cluster-tilted algebra of type $D_4$ can be found in Section \ref{sec:examples}.
\end{ex}

\begin{figure}
\begin{tikzpicture}[scale=0.5,baseline=0]
	\draw (0,0) circle[radius=2cm];
	\draw (-.1,-.1) to (.1,.1);
	\draw (-.1,.1) to (.1,-.1);
		\draw[very thick] plot [smooth] coordinates{(0,2)(-1.73,-1)};
		\draw[very thick,dashed] plot [smooth] coordinates{(0,2)(1,1)(1.73,-1)};
		\draw[very thick] plot [smooth] coordinates{(0,2)(0.5,-0.5)(-1.73,-1)};
		\node at (0,2) [draw,fill,circle,scale=0.4,label=above:1]{};
		\node at (-1.73,-1) [draw,fill,circle,scale=0.4,label=south west:2]{};
		\node at (1.73,-1) [draw,fill,circle,scale=0.4,label=south east:3]{};
	\node at (0,2.5){};
	\node at (0,-2.5){};
\end{tikzpicture}
\caption{A maximal admissible arc pattern for $D(3,3,2,1)$. The corresponding semibrick pair is a 2-simple minded collection. Recall that green arcs are drawn as solid and red arcs are drawn as dashed.}\label{fig:arcpattern}
\end{figure}

We now wish to prove in general that the semibrick pair corresponding to a maximal admissible arc pattern is a 2-simple minded collection.

\begin{thm}
\label{thm:pairwisesmc}
Let $A$ be an admissible arc pattern for $D(n,l(1),\ldots,l(n))$. Then $\mathcal{X}(A)$ is a 2-simple minded collection if and only if $A$ is maximal. In particular, by Proposition \ref{rem:mutationpairarcs}, every 2-simple minded collection occurs in this way.
\end{thm}

As admissible arc patterns are given by pairwise compatibility conditions, this result is equivalent to Theorem \ref{thm:nakayamapairwise}. We begin with several results pertaining to maximal arc patterns.

\begin{define}
Let $A$ be an admissible arc pattern for $D(n,l(1),\ldots,l(n))$ and let $A^*$ denote the set of inverse paths to those in $A$. For $v, v'$ marked points on the punctured disk, we say a sequence $(a_1,\ldots,a_k) \in (A\cup A^*)^k$ is an \emph{arc path} of $A$ from $v$ to $v'$ if
\begin{enumerate}[label=\upshape(\alph*)]
	\item $s(a_1) = v$.
	\item $t(a_k) = v'$.
	\item $t(a_j) = s(a_{j+1})$ for $1\leq j < k$.
	\item $a_1\cdots a_k$ is homotopic to the boundary arc $v\rightarrow v'$.
	\item $M(a_1,\ldots,a_k):=M(v\rightarrow v')$ is defined (and is a brick).
\end{enumerate}
We set $\mathrm{supp}(a_1,\ldots,a_k):=\mathrm{supp}(v\rightarrow v')$. An arc path $(a_1,\ldots a_k)$ is called \emph{proper} if every subpath $(a_i,\ldots a_j), 1\leq i\leq j\leq k$, is either an arc path or the reverse of an arc path. If only conditions (a)-(c) hold, we refer to the sequence of arcs as a \emph{pseudo-arc path}.
\end{define}

\begin{ex}
In the setup of Example \ref{ex:arcpattern} shown above, we identify each brick with its arc. Then $\left(\scriptsize{\begin{matrix}2\\3\end{matrix}}, 3[1]^*\right)$ is a proper arc path from 2 to 3 and $\left(\scriptsize{\begin{matrix}2\\3\end{matrix}}, 1, \scriptsize{\begin{matrix}2\\3\end{matrix}},3[1]^*\right)$ is a pseudo arc path from 2 to 3. The notation $3[1]^*$ represents traveling clockwise along the arc corresponding to $3[1]$.
\end{ex}

We now build toward the main theorem of this section by proving several propositions about proper arc paths.

\begin{prop}
\label{prop:degreezero}
Let $A$ be a maximal arc pattern for $D(n,l(1),\ldots,l(n))$. Then every vertex is the source or target of at least one arc.
\end{prop}

\begin{proof}
Assume there exists a marked point $j$ which is not the source or target of any arc in $A$. Then since $A$ is maximal, a red arc $j\rightarrow (j+1)_n$ must be incompatible with $A$. Thus there must be a red arc $R \in A$ with $t(R) = (j+1)_n$. Likewise, as a green arc $(j-1)_n\rightarrow j$ must be incompatible with $A$, there must be a green arc $G\in A$ with $s(G) = (j-1)_n$. Now for $R$ and $G$ to not cross in the interior of the disk, either $s(G) = (j+1)_n$ or $t(R) = (j-1)_n$. However, in either case, $G$ will be counterclockwise from $R$, a contradiction.

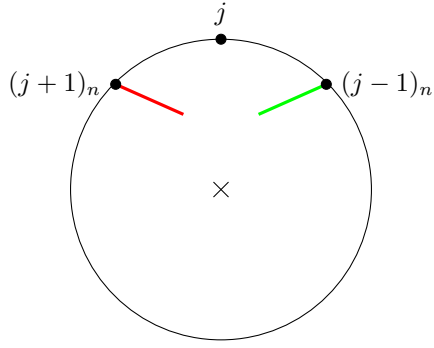
\begin{center}
\begin{tikzpicture}[scale=0.55]
	\draw (0,0) circle[radius=2cm];
	\draw (-.1,-.1) to (.1,.1);
	\draw (-.1,.1) to (.1,-.1);
		\draw[very thick] (1.4,1.4) to (.5,1);
		\draw[very thick,dashed] (-.5,1) to  (-1.4,1.4);
		\node at (0,2) [draw,fill,circle,scale=0.4,label=$j$]{};
		\node at (-1.4,1.4) [draw,fill,circle,scale=0.4,label=left:$(j+1)_n$]{};
		\node at (1.4,1.4) [draw,fill,circle,scale=0.4,label=right:$(j-1)_n$]{};
\end{tikzpicture}
\end{center}
\end{proof}

\begin{prop}
Let $A$ be a maximal arc pattern for $D(n,l(1),\ldots,l(n))$ and let $i$ and $j$ be marked points. Then there exists a pseudo arc path $i\rightarrow j$.
\end{prop}

\begin{proof}
We say a pair of marked points are \emph{connected} by $A$ if there is a pseudo arc path between them. By Proposition \ref{prop:degreezero} we know that every marked point is connected to itself under this definition. Moreover, it is clear from the definition that the concatenation of two pseudo arc paths is again a pseudo arc path. Therefore we can partition the set of marked points into connected components. We will show that this partition consists of only a single connected component, which will imply the result.

Assume for a contradiction that there are two distinct connected components $C_1$ and $C_2$. Now choose marked points $i_1,j_1 \in C_1$ and $i'_2,j'_2 \in C_2$ so that $C_1 \subseteq [i_1,j_1]_n$ and $C_2 \subseteq [i'_2,j'_2]_n$. (The reason for using $i'_2$ and $j'_2$ as opposed to $i_2$ and $j_2$ will be made apparent in the following paragraphs.) In particular, there exists a pseudo arc path $s: i'_2\rightarrow j'_2$. Moreover, for any pseudo arc path $t:i\rightarrow j$ between marked points in $C_1$, there can be no intersections between $s$ and $t$. Indeed, $s$ and $t$ cannot intersect in the interior of the disk because $A$ is an admissible arc pattern. Likewise, $s$ and $t$ cannot intersect at a marked point because then there would be a pseudo arc path $i'_2\rightarrow j$, contradicting that $i'_2$ and $j$ are in distinct connected components.

One consequence of the previous paragraph is that $C_1$ and $C_2$ cannot both consist of a single marked point. Indeed, if $i_1 = j_1$, then Proposition \ref{prop:degreezero} implies that there is an arc in $A$ with $i_1$ as both its source and target. This arc would then necessarily intersect the analogous arc with $i'_2 = j'_2$ as both its source and target, a contradiction. Therefore, we assume without loss of generality that $i'_2 \neq j'_2$. We then see that at least one of $C_1 \cap (i'_2,j'_2)_n$ and $C_1\cap (j'_2,i'_2)_n$ must be empty. Indeed, if there is a pseudo arc path between a marked point in $(i'_2,j'_2)_n$ and one in $(j'_2,i'_2)_n$, then this pseudo arc path will necessarily intersect the one between $i'_2$ and $j'_2$.

If $C_1 \cap (i'_2,j'_2)_n = \emptyset$, then we drop the primes and write $i_2 := i'_2$ and $j_2 := j'_2$. Otherwise, we have $C_2 \cap (j_1,i_1)_n \neq \emptyset$, and so by symmetry $C_2 \cap (i_1,j_1)_n = \emptyset$. In this case, there then exist unique marked points $i_2 \in C_2\cap (j_1,j'_2]$ and $j_2 \in C_2 \cap [i'_2,i_1)$ so that $C_2 \subseteq [i_2,j_2]_n$. In either case, we then have that $C_1 \cap (i_2,j_2)_n = \emptyset = C_2\cap (i_1,j_1)_n$.

We now have two cases to consider.

\noindent Case 1: Assume there exists a green arc $G\in A$ with source $j_1$. This includes the case where $i_1 = j_1$ and there is a green arc from $j_1$ to $j_1$. It follows that $l(k) \geq (i_1-k)_n$ for all $k \in [j_1,i_1]_n$. Moreover, there cannot be an arc with source $j_2$ or target $i_2$, as it would intersect $G$. Likewise, the only arc that can possibly have target $i_1$ is $G$, as any red arc with target $i_1$ would be clockwise from $G$ or intersect $G$.

Now, by the maximality of $A$, there must be a green arc $G'\in A$ with target $j_2$. Otherwise a red arc $j_2 \rightarrow i_1$ would be compatible with $A$. Now any red arc with source $s(G')$ would need to have target in $(s(G'),j_2)_n$, but this would make it clockwise from $G'$. Thus we can assume the existence a green arc $G''\in A$ with target $s(G')$ as before. By iterating this process, we can assume that $s(G'') = i_2$. At this point, there is no way to prevent the compatibility of a red arc $i_2\rightarrow i_1$ with $A$, a contradiction.

\begin{center}
\begin{tikzpicture}[scale=0.55]
	\draw (0,0) circle[radius=2cm];
	\draw (-.1,-.1) to (.1,.1);
	\draw (-.1,.1) to (.1,-.1);
		\draw[very thick,dotted] plot [smooth] coordinates{(1.41,1.41) (-0.5,0) (1.73,-1)};
		\draw[very thick] plot [smooth] coordinates{(-1.4,-1.4) (-2,0)};
		\draw[very thick] plot [smooth] coordinates{(-1.4,1.4) (-2,0)};
		\draw[very thick,dashed] plot [smooth] coordinates{(-1.4,1.4) (-0.5,-1) (1.41,-1.41)};
		\node at (1.41,1.41) [draw,fill,circle,scale=0.4,label=right:$j_1$]{};
		\node at (-1.41,1.41) [draw,fill,circle,scale=0.4,label=left:$i_2$]{};
		\node at (-1.41,-1.41) [draw,fill,circle,scale=0.4,label=left:$j_2$]{};
		\node at (1.41,-1.41) [draw,fill,circle,scale=0.4,label=right:$i_1$]{};
		\node at (-2,0) [draw,fill,circle,scale=0.4,label=left:$s(G')$]{};
		\node at (1.73,-1) [draw,fill,circle,scale=0.4]{};
		
		\node at (3,0){};
		\node at (-3,0){};
\end{tikzpicture}
\end{center}

\noindent Case 2: Now suppose there does not exist a green arc with source $j_1$. Moreover, assume without loss of generality that there does not exist an arc $a \in A$ with source and target in $[i_2,j_2]_n$ such that $[i_1,j_1]_n \subset \mathrm{supp}(a)$. (Indeed, if such an arc exists, then it would intersect any arc $b \in A$ with source and target in $[i_1,j_1]_n$ such that $[i_2,j_2]_n \subseteq\mathrm{supp}(b)$.) In particular, there is no arc in $A$ with target $i_2$.

Then since a green arc $(i_2-1)_n \rightarrow i_2$ cannot be compatible with $A$, there must be a red arc $R\in A$ with source $i_2$ and $l((i_2-1)_n) \geq (t(R)-i_2-1)_n$. Now any green arc with target $t(R)$ would necessarily be counterclockwise from $R$. Thus since a green arc $(i_2-1)_n \rightarrow t(R)$ must be incompatible with $A$, there must exist a red arc $R'\in A$ with source $t(R)$ and $l((i_2-1)_n) \geq (t(R')-i_2-1)_n$. By iterating this process, we can assume that $t(R') = j_2$. At this point, there is no way to prevent the compatibility of a green arc $(i_2-1)_n \rightarrow j_2$ with $A$, a contradiction.

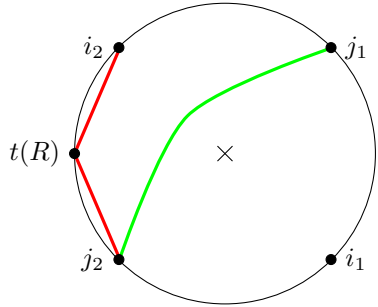
\begin{center}
\begin{tikzpicture}[scale=0.55]
	\draw (0,0) circle[radius=2cm];
	\draw (-.1,-.1) to (.1,.1);
	\draw (-.1,.1) to (.1,-.1);
		\draw[very thick,dashed] plot [smooth] coordinates{(-1.4,-1.4) (-2,0)};
		\draw[very thick,dashed] plot [smooth] coordinates{(-1.4,1.4) (-2,0)};
		\draw[very thick] (0,2)--(-1.4,-1.4);
		\node at (1.41,1.41) [draw,fill,circle,scale=0.4,label=right:$j_1$]{};
		\node at (-1.41,1.41) [draw,fill,circle,scale=0.4,label=left:$i_2$]{};
		\node at (-1.41,-1.41) [draw,fill,circle,scale=0.4,label=left:$j_2$]{};
		\node at (1.41,-1.41) [draw,fill,circle,scale=0.4,label=right:$i_1$]{};
		\node at (-2,0) [draw,fill,circle,scale=0.4,label=left:$t(R)$]{};
		\node at (0,2) [draw,fill,circle,scale=0.4,label=above:$(i_2-1)_n$]{};
		\node at (3,0){};
		\node at (-3,0){};
\end{tikzpicture}
\end{center}

We have shown that the marked points of $D(n,l(1),\ldots,l(n))$ cannot be partitioned into more than one connected component. This means there exists a pseudo arc path between any two marked points.
\end{proof}

\begin{prop}
\label{prop:simplepaths}
Let $A$ be a maximal arc pattern for $D(n,l(1),\ldots,l(n))$ and let $j$ be a marked point. Then there exists a proper arc path $j\rightarrow (j+1)_n$.
\end{prop}

\begin{proof}
Clearly, if there is an arc $j\rightarrow (j+1)_n$ in $A$ then we are done. Thus assume no such arc exists. We first reduce to the case that there exists an arc $a\in A$ with $[j,j+1]_n \subset \mathrm{supp}(a)$. Indeed, if there exists no such arc, then in particular, no arc in $A$ has source $j$ or target $(j+1)_n$. Thus by the maximality of $A$, there must be a green arc $G\in A$ with target $j$ and a red arc $R \in A$ with source $(j+1)_n$ such that $l(s(G)) \geq (j+1-s(G))_n$ and $l(j) \geq (t(R)-j)_n$. Otherwise an arc $j\rightarrow (j+1)_n$ would be compatible with $A$.

Now any green arc in $A$ with target $t(R)$ must have source in $(t(R),s(G))_n$, but such an arc would contain $[j,j+1]_n$ in its support. Thus we can assume there is no green arc in $A$ with target $t(R)$. Then as a green arc $j\rightarrow t(R)$ must be incompatible with $A$, there must exist a red arc $R'\in A$ with source $t(R)$ such that $l(j)\geq (t(R')-j)_n$. As before, we can assume there is no green arc with target $t(R')$ (or it would contain $[j,j+1]_n$ in its support). Thus there exists a red arc $R''\in A$ with source $t(R')$. Iterating this argument, we can assume that $t(R'') = s(G)$. Again, there can be no green arc with target $s(G)$. But then there is no way to prevent the compatibility of a red arc $s(G) \rightarrow (j+1)_n$ with $A$, a contradiction. We conclude that there exists some arc $a \in A$ such that $[j,j+1]_n \subset \mathsf{supp}(a)$.

\begin{center}
\begin{tikzpicture}[scale=0.55]
	\draw (0,0) circle[radius=2cm];
	\draw (-.1,-.1) to (.1,.1);
	\draw (-.1,.1) to (.1,-.1);
		\draw[very thick] plot [smooth] coordinates{(-1.73,1) (0,2)};
		\draw[very thick,dashed] plot [smooth] coordinates{(-1.73,-1) (0,-2)};
		\draw[very thick,dashed] plot [smooth] coordinates{(2,0) (0,-2)};
		\draw[very thick,dashed] plot [smooth] coordinates{(2,0) (0,2)};
		\draw[very thick,dashed] plot [smooth] coordinates {(0,2) (-1.73,-1)};
		\node at (-1.73,1) [draw,fill,circle,scale=0.4,label=left:$j$]{};
		\node at (-1.73,-1) [draw,fill,circle,scale=0.4,label=left:$j+1$]{};
		\node at (0,2) [draw,fill,circle,scale=0.4,label=$s(G)$]{};
		\node at (0,-2) [draw,fill,circle,scale=0.4,label=below:$t(R)$]{};
		\node at (2,0) [draw,fill,circle,scale=0.4,label=right:$t(R')$]{};
		\node at (3,0){};
		\node at (-3,0){};
\end{tikzpicture}
\end{center}

Now for $k \in \mathrm{supp}(a)$ we have $l(k) \geq (t(a)-k)_n$. Now by the previous proposition, there exists a pseudo arc path $(a_1,\ldots,a_k): t(a) \rightarrow (j+1)_n$. Thus there exists some index $i \leq k$ for which $s(a_i) = s(a)$ or $s(a_i) = t(a)$ and $(a_{i+1},\ldots,a_k)$ remains in the interval $(s(a),t(a))_n$. If $s(a_i) = s(a)$, then $(a_{i+1},\ldots,a_k)$ is a proper arc path $s(a) \rightarrow (j+1)_n$. Otherwise, $s(a_i) = t(a)\neq s(a)$ and $(a,a_i,\ldots,a_k)$ is a proper arc path $s(a) \rightarrow (j+1)_n$. We denote whichever proper arc path exists by $(a_{j+1}): s(a) \rightarrow (j+1)_n$.

Likewise, there exists a pseudo arc path $(b_1,\ldots,b_{k'}): s(a) \rightarrow j$. Again, there is some index $i'\leq k'$ for which $s(b_{i'}) = s(a)$ or $s(b_{i'}) = t(a)$ and $(b_{i'+1},\ldots,b_{k'})$ remains in the interval $(s(a),t(a))_n$. If $s(b_{i'}) = s(a)$, then $(b_{i'},\ldots, b_{k'})$ is a proper arc path $s(a)\rightarrow j$ and $(b_{k'}^*,\ldots b_{i'}^*,a_{j+1})$ is a proper arc path $j\rightarrow (j+1)_n$. Otherwise, $s(b_{i'}) = t(a)\neq s(a)$. In this case, $(a,b_{i'},\ldots,b_{k'})$ is a proper arc path $s(a) \rightarrow j$ and $(b_{k'}^*,\ldots,b_{i'}^*,a^*,a_{j+1})$ is a proper arc path $j\rightarrow (j+1)_n$.

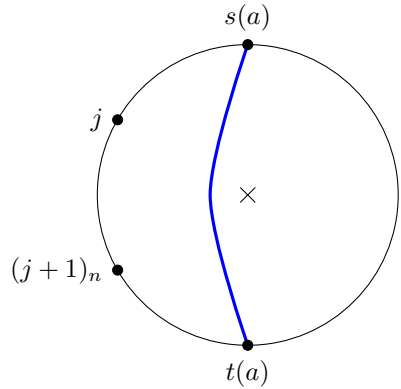
\begin{center}
\begin{tikzpicture}[scale=0.55]
	\draw (0,0) circle[radius=2cm];
	\draw (-.1,-.1) to (.1,.1);
	\draw (-.1,.1) to (.1,-.1);
		\draw[very thick,dotted] plot [smooth] coordinates{(0,2) (-.5,0) (0,-2)};
		\node at (-1.73,1) [draw,fill,circle,scale=0.4,label=left:$j$]{};
		\node at (-1.73,-1) [draw,fill,circle,scale=0.4,label=left:$(j+1)_n$]{};
		\node at (0,2) [draw,fill,circle,scale=0.4,label=$s(a)$]{};
		\node at (0,-2) [draw,fill,circle,scale=0.4,label=below:$t(a)$]{};
		\node at (3,0){};
		\node at (-3,0){};
\end{tikzpicture}
\end{center}
\end{proof}

We are now ready to prove Theorem \ref{thm:pairwisesmc}. Our proof is modeled after that of \cite[Lem.~8.11]{garver_oriented} by Garver and McConville.

\begin{proof}[Proof of Theorem \ref{thm:pairwisesmc}]
It is clear that if $\mathcal{X}(A)$ is a 2-simple minded collection, then $A$ is maximal. Thus let $A$ be a maximal arc pattern and let $\mathfrak{T} = \mathsf{thick}(\mathcal{X}(A))$. We will show that $\mathfrak{T}$ contains all simple modules and therefore $\mathfrak{T} = \mathcal{D}^b(\mathsf{mod}\Lambda)$.

Let $S \in \mathsf{mod}\Lambda$ be simple and let $(p_1,\ldots,p_k)$ be a sequence of proper arc paths and reverses of proper arc paths. Then we say $(p_1,\ldots,p_k)$ is an \emph{admissible path sequence} for $S$ if
\begin{enumerate}
	\item $M(|p_j|) \in \mathfrak{T}$ for all $j$, where by $|p_j|$ we mean the proper arc path corresponding to $p_j$.
	\item Considered as a single pseudo arc path, $(p_1p_2\cdots p_k)$ is a proper arc path for $S$.
\end{enumerate}

We observe that every $a \in A$ is itself an admissible path sequence for $M(a)$. Moreover, by Proposition \ref{prop:simplepaths}, every simple module $S \in \mathsf{mod}\Lambda$ admits an admissible path sequence. We now prove that if $S$ has an admissible path sequence of length $k>1$, then $S$ has an admissible path sequence of length $k-1$. In particular, this means $S$ has an admissible path sequence of length 1 and hence $S \in \mathfrak{T}$.

Let $(p_1,\ldots p_k)$ be an admissible path sequence for $S$ of length $k > 1$. Then there exists some index $1 < j \leq k$ such that $\mathrm{supp}(|p_{j-1}|)\neq \mathrm{supp} (|p_j|)$.

If $\mathrm{supp}(|p_{j-1}|)\cap \mathrm{supp}(|p_j|) \neq \emptyset$ contains only a single marked point, then there is either an exact sequence
$$M(|p_{j-1}|)\hookrightarrow M(|p_{j-1}p_j|)\twoheadrightarrow M(|p_j|)$$
or
$$M(|p_{j}|)\hookrightarrow M(|p_{j-1}p_j|) \twoheadrightarrow M(|p_{j-1}|).$$
In either case, $M(|p_{j-1}p_j|) \in T$ and $(p_1,\ldots,p_{j-2},p_{j-1}p_j,p_{j+1},\ldots, p_k)$ is an admissible path sequence for $S$ of length $k-1$.

If $\mathrm{supp}(|p_{j-1}|)\cap \mathrm{supp}(|p_j|)\neq\emptyset$ contains more than a single marked point, then one of the following holds:
\begin{enumerate}
	\item $\mathrm{supp}(|p_{j-1}|) \subset \mathrm{supp}(|p_{j}|)$ and there is a monomorphism $f: M(|p_{j-1}|)\hookrightarrow M(p_j)$.
	\item $\mathrm{supp}(|p_{j-1}|) \subset \mathrm{supp}(|p_{j}|)$ and there is an epimorphism $f: M(|p_{j}|)\twoheadrightarrow M(p_{j-1})$.
	\item $\mathrm{supp}(|p_{j}|) \subset \mathrm{supp}(|p_{j-1}|)$ and there is a monomorphism $f: M(|p_{j}|)\hookrightarrow M(p_{j-1})$.
	\item $\mathrm{supp}(|p_{j}|) \subset \mathrm{supp}(|p_{j-1}|)$ and there is an epimorphism $f: M(|p_{j-1}|)\twoheadrightarrow M(p_j)$.
\end{enumerate}
In case (2) and (4), we see that $N:=\mathrm{ker} f \in \mathfrak{T}$. Likewise in case (1) and (3), we see that $N:=\mathrm{coker} f \in \mathfrak{T}$. In any case, we then have $M(|p_{j-1}p_j|) = N$ and $(p_1,\ldots,p_{j-2},p_{j-1}p_j,p_{j+1},\ldots,p_k)$ is an admissible path sequence for $S$ of length $k-1$.
\end{proof}

As previously stated, Theorem \ref{thm:nakayamapairwise} then follows immediately from Theorem \ref{thm:pairwisesmc} and the fact that admissible arc patterns are defined by pairwise compatibility conditions.

\section{Picture Groups and the Fundamental Group}
\label{sec:pi1}

In the case that $\Lambda$ is hereditary, the \emph{picture group} of $\Lambda$, denoted $G(\Lambda)$, is defined in \cite{igusa_picture}. It is then shown in \cite{igusa_signed} that the fundamental group of $\mathcal{B}\mathfrak{W}(\Lambda)$ is isomorphic to $G(\Lambda)$. The goal of this section is to extend the definition of the picture group and this result to non-hereditary algebras, and to construct a faithful group functor $\mathfrak{W}(\Lambda)\rightarrow G(\Lambda)$ in the case that $\Lambda$ is a Nakayama algebra.

\subsection{Picture Groups}

We begin by recalling the definition of the picture group in the hereditary case. We denote by $\alpha_1,\ldots,\alpha_n$ the positive simple roots of a hereditary algebra $\Lambda$. For a (not necessarily simple) positive root $\alpha$, we denote by $M_\alpha$ the corresponding module.

\begin{defthm}\cite[Def. 1.1.5, Def. 1.1.8, Thm. 2.2.1]{igusa_picture}
\label{defthm:hereditarypicture}
Let $\Lambda$ be a representation finite hereditary algebra. Then the picture group $G(\Lambda)$ has the following presentation.
\begin{itemize}
	\item $G(\Lambda)$ has one generator $X_S$ for every brick $S \in \mathsf{brick}\Lambda$.
	\item For each pair $S_1,S_2$ of Hom-orthogonal bricks such that $\mathrm{Ext}(S_1,S_2) = 0$, write $S_1 = M_\alpha$ and $S_2 = M_\beta$. We then have the relation
		$$X_{S_1}X_{S_2} = \prod_{T \in \mathsf{brick}(\mathsf{Filt}(S_1\sqcup S_2))}X_T$$
		where the order of the product is given by the increasing order of the ratio $r/s$ of the decomposition $T = M_{r\alpha+s\beta}$ (going from 0/1 to 1/0).
\end{itemize}
\end{defthm}

We extend this definition to the non-hereditary case by first recalling several facts about the lattice $\mathsf{tors}\Lambda$. We recall that an interval $[X,Y]$ in a lattice is called a \emph{polygon} if $(X,Y)$ consists of two disjoint nonempty chains.

\begin{thm}\cite[Thm. 4.16, Prop. 4.21]{demonet_lattice}\label{thm:polygons}
\
\begin{enumerate}[label=\upshape(\alph*)]
	\item The lattice $\mathsf{tors}\Lambda$ is \emph{polygonal}. That is
	\begin{enumerate}[label=\upshape(\roman*)]
		\item Given arrows $\mathcal{T}\rightarrow \mathcal{T}_1$ and $\mathcal{T}\rightarrow \mathcal{T}_2$ in $\mathsf{Hasse}(\mathsf{tors}\Lambda)$, the interval $[\mathcal{T}_1\wedge\mathcal{T}_2,\mathcal{T}]$ is a polygon.
		\item Given arrows $\mathcal{T}_1\rightarrow \mathcal{T}$ and $\mathcal{T}_2\rightarrow \mathcal{T}$ in $\mathsf{Hasse}(\mathsf{tors}\Lambda)$, the interval $[\mathcal{T},\mathcal{T}_1\vee\mathcal{T}_2]$ is a polygon.
	\end{enumerate}
	\item Let $\mathcal{P}$ be a polygon in $\mathsf{tors}\Lambda$. Then there exists a polygon $\mathcal{P}'$ with the same brick labels as $\mathcal{P}$ and a support $\tau$-tilting pair $M\sqcup P[1]$ with $rk(M\sqcup P[1]) = rk(\Lambda)-2$ such that $\mathcal{P}'$ has the form
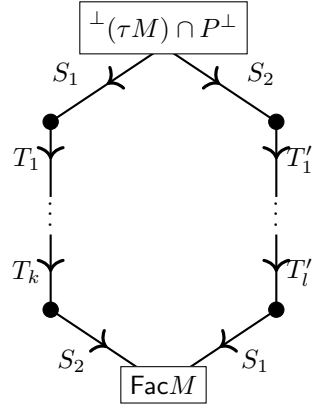
\begin{center}
\begin{tikzpicture}
	\begin{scope}[thick,decoration={
	markings,
	mark=at position 0.5 with {\arrow[scale=1.5]{>}}}
	]
	\draw[postaction={decorate}] (0,-0.5)--(-1.5,-1.5) node [pos=0.65,anchor=south east]{$S_1$};
	\draw[postaction={decorate}] (0,-0.5)--(1.5,-1.5) node [pos=0.65,anchor=south west]{$S_2$};
	\draw[postaction={decorate}] (-1.5,-1.5)--(-1.5,-2.5) node [midway,anchor=east]{$T_1$};
	\draw[postaction={decorate}] (1.5,-1.5)--(1.5,-2.5) node [midway,anchor=west]{$T'_1$};
	\draw[postaction={decorate}] (-1.5,-3)--(-1.5,-4) node [midway,anchor=east]{$T_k$};
	\draw[postaction={decorate}] (1.5,-3)--(1.5,-4) node [midway,anchor=west]{$T'_l$};
	\draw[postaction={decorate}] (-1.5,-4)--(0,-5) node [pos=0.4,anchor=north east]{$S_2$};
	\draw[postaction={decorate}] (1.5,-4)--(0,-5) node [pos=0.4,anchor=north west]{$S_1$};
	\end{scope}
	
	\node at (0,-0.25) [draw,fill=white]{$\lperp{(\tau M)}\cap P^\perp$};
	\node at (0,-5) [draw,fill=white]{$\mathsf{Fac} M$};
	
	\node at (-1.5,-1.5) [draw,fill,circle,scale=0.6]{};
	\node at (1.5,-1.5) [draw,fill,circle,scale=0.6]{};
	\node at (-1.5,-4) [draw,fill,circle,scale=0.6]{};
	\node at (1.5,-4) [draw,fill,circle,scale=0.6]{};
	
	\node at (-1.5,-2.7)[rotate=90]{\small $\cdots$};
	\node at (1.5,-2.7)[rotate=90]{\small $\cdots$};
\end{tikzpicture}
\end{center}
where each arrow is labeled with the corresponding brick label.
	\item In the above polygon, we have $J(M\sqcup P[1]) = \mathsf{Filt}(S_1\sqcup S_2)$. Moreover, $\mathsf{brick}(\mathsf{Filt}(S_1\sqcup S_2)) =\{S_1,S_2,T_1,\ldots,T_k,T'_1,\ldots,T'_l\}$, with each brick appearing once in this collection.
\end{enumerate}
\end{thm}

We observe that there may be multiple torsion classes $\mathcal{T}$ for which $S_1$ and $S_2$ label arrows $\mathcal{T}\rightarrow\mathcal{T}_1$ and $\mathcal{T}\rightarrow\mathcal{T}_2$ (resp. $\mathcal{T}_1\rightarrow\mathcal{T}$ and $\mathcal{T}_2\rightarrow\mathcal{T}$). However, as a consequence of (c), the labeling of the polygon $[\mathcal{T}_1\wedge\mathcal{T}_2,\mathcal{T}]$ (resp. $[\mathcal{T},\mathcal{T}_1\vee\mathcal{T}_2]$) depends only on $S_1$ and $S_2$, not on $\mathcal{T}$. Thus we refer to any polygon which occurs in $\mathsf{Hasse}(\mathsf{tors}\Lambda)$ with this labeling as $\mathcal{P}(S_1,S_2)$. We refer to the sequences $(S_1,T_1,\ldots,T_k,S_2)$ and $(S_2,T'_1,\ldots,T'_l,S_1)$ as the \emph{sides} of the polygon $\mathcal{P}(S_1,S_2)$. We now propose the following extension of the definition of the picture group.

\begin{defthm}
\label{defthm:picture}
The \emph{picture group} $G(\Lambda)$ is defined to have the following presentation.
\begin{itemize}
	\item $G(\Lambda)$ has one generator $X_S$ for every brick $S \in \mathsf{brick}\Lambda$.
	\item For each polygon $\mathcal{P}(S_1,S_2)$ in $\mathsf{tors}\Lambda$, we have the \emph{polygon relation}
$$X_{S_1}X_{T_1}\cdots X_{T_k}X_{S_2} = X_{S_2}X_{T'_1}\cdots X_{T'_l}X_{S_1}$$
	where $(S_1,T_1,\ldots,T_k,S_2)$ and $(S_2,T'_1,\ldots,T'_l,S_1)$ are the sides of $\mathcal{P}(S_1,S_2)$.
	\end{itemize}
In the hereditary case, this agrees with Definition-Theorem \ref{defthm:hereditarypicture}.
\end{defthm}

\begin{proof}
In the case that $\Lambda$ is a representation-finite hereditary algebra, our relations are the inverse of the relations given in Definition-Theorem \ref{defthm:hereditarypicture} and hence the groups are isomorphic. Indeed, as in \cite[Example 2.2.2]{igusa_picture}, we need only consider the six cases where $\Lambda = KQ$ and $Q$ is Dynkin with two vertices. For example, if
$$Q = 2\xleftarrow{(1,2)}1 \cong B_2 \cong C_2$$
then $\mathsf{brick}\Lambda = \{M_{1,0},M_{0,1},M_{1,1},M_{2,1}\}$ and $\mathsf{Hasse}(\mathsf{tors}\Lambda)$ consists of two chains:
$$\mathsf{mod}\Lambda\xrightarrow{M_{1,0}}\mathsf{Filt}\mathsf{Fac} M_{2,1}\xrightarrow{M_{2,1}}\mathsf{Filt}\mathsf{Fac} M_{1,1}\xrightarrow{M_{1,1}}\mathsf{Filt} M_{0,1}\xrightarrow{M_{0,1}}0$$
and
$$\mathsf{mod}\Lambda\xrightarrow{M_{0,1}}\mathsf{Filt} M_{1,0}\xrightarrow{M_{1,2}}0.$$
Thus in this case, the two definitions yield isomorphic groups. The other five cases follow from identical reasoning.
\end{proof}

Recall from \cite[Thm. 2.12, 2.13]{li_maximal} that \emph{maximal green sequences} correspond to sequences $(S_1,\ldots,S_k)$ of bricks labeling a (directed) path $\mathsf{mod}\Lambda\rightarrow 0$ in $\mathsf{Hasse}(\mathsf{tors}\Lambda)$. In the hereditary case, the second author and Todorov explore the relationship between picture groups and maximal green sequences in \cite{igusa_pictureMGS}. We now give three additional presentations of the picture group $G(\Lambda)$ that will be useful in proving our main results.

\begin{prop}\
\label{prop:presentations}
\smallskip
\begin{enumerate}[label=\upshape(\alph*)]
	\item $G(\Lambda)$ is generated by $\{X_S|S\in\mathsf{brick}\Lambda\}$ with a relation
	$$X_{S_1}\cdots X_{S_k} = X_{S'_1}\cdots X_{S'_l}$$
	if there exist torsion classes $\mathcal{T},\mathcal{T}'\in\mathsf{tors}\Lambda$ such that $(S_1,\ldots,S_k)$ and $(S'_1,\ldots S'_l)$ label (directed) paths $\mathcal{T}\rightarrow \mathcal{T}'$ in $\mathsf{Hasse}(\mathsf{tors}\Lambda)$.
	\item $G(\Lambda)$ is generated by $\{X_S|S\in\mathsf{brick}\Lambda\}$ with a relation
	$$X_{S_1}\cdots X_{S_k} = X_{S'_1}\cdots X_{S'_l}$$
	if $(S_1,\ldots,S_k)$ and $(S'_1,\ldots S'_l)$ correspond to maximal green sequences for~$\Lambda$.
	\item $G(\Lambda)$ is generated by $\{X_S|S\in\mathsf{brick}\Lambda\}\cup\{g_\mathcal{T}|\mathcal{T}\in\mathsf{tors}\Lambda\}$
	with a relation
	$$g_{\mathcal{T}} = X_Sg_{\mathcal{T}'}$$
	if there is an arrow $\mathcal{T}\xrightarrow{S}\mathcal{T}'$ in $\mathsf{Hasse}(\mathsf{tors}\Lambda)$ and the relation
	$$g_0 = e.$$
\end{enumerate}
\end{prop}

\begin{proof}
(a) is an immediate consequence of \cite[Lem. 9-6.3]{reading_lattice}. (b) and (c) follow immediately from~(a).
\end{proof}

\begin{rem}
	We could also give a presentation of the picture group using the wall and chamber structure of \cite{brustle_wall}. This construction is the one originally given in the hereditary case (see \cite{igusa_picture,loday_homotopical}). In this presentation, the generators correspond to bricks and the relations make the product of the labels of walls traversed by any oriented cycle in $\mathbb{R}^n$ transverse to these walls trivial. This is equivalent to the definition in terms of polygon relations due to the duality of $\mathsf{Hasse}(\mathsf{tors}\Lambda)$ with the wall and chamber structure (see \cite[Prop. 4.5]{brustle_wall}).
\end{rem}

\subsection{The Fundamental Group}

Our next goal is to show that the fundamental group of the picture space $\mathcal{B}\mathfrak{W}(\Lambda)$ is isomorphic to the picture group $G(\Lambda)$. The argument is similar to those found in \cite[Section 4]{igusa_category} and \cite[Section 4.1, 4.2]{igusa_signed}; however, due to differences in notation, we include many of the details here. Readers unfamiliar with the definitions of the \emph{star} and \emph{link} of vertex in a simplicial complex and the method for computing the fundamental group of a CW complex are referred to \cite[Sec. 1.2]{munkres_elements} and \cite[Sec 1.2]{hatcher_algebraic}.

Given a wide subcategory $W \in \mathsf{wide}\Lambda$, write $W = \mathsf{Filt}(\mathcal{S})$ with $\mathcal{S} \in \mathsf{sbrick}\Lambda$. Then the \emph{rank} of $W$ is defined as $\mathsf{rk}(\mathcal{S})$. We recall the following, which can be deduced from \cite[Thm. 5.2]{demonet_tilting} or \cite[Prop. 3.15]{brustle_wall}.

\begin{lem}\label{lem:sphere}
	Let $W$ be a wide subcategory of rank $k\neq 0$. Let $\Sigma(W)$ be the simplicial complex with vertices the (isoclasses of) indecomposable support $\tau$-rigid pairs for $W$ so that a set of vertices spans a simplex in $\Sigma(W)$ if and only if their direct sum is a support $\tau$-rigid pair. Then $\Sigma(W)$ is homeomorphic to a $(k-1)$-sphere and each morphism $[U]:W\rightarrow0$ corresponds to the $(k-1)$-simplex whose vertices are the direct summands of $U$.
\end{lem}

The following proposition and its proof are virtually identical to \cite[Sec. 4.1]{igusa_signed} and \cite[Thm 4.4]{igusa_category}. We give an example following the proof.

\begin{prop}
\label{prop:CW}
The classifying space $\mathcal{B}\mathfrak{W}(\Lambda)$ is a $\mathsf{rk}(\Lambda)$-dimensional CW-complex having one cell $e(W)$ of dimension $k$ for every wide subcategory $W$ of rank $k$. The $k$-cell $e(W)$ is the union of the factorization cubes of all morphisms $W\rightarrow 0$.
\end{prop}

\begin{proof}
	For simplicity, we identify each morphism in $\mathfrak{W}(\Lambda)$ with its factorization cube in $\mathcal{B}\mathfrak{W}(\Lambda)$. Let $W \in \mathsf{wide}\Lambda$ be a wide subcategory of rank $k$ and let $e(W)$ be the union of the factorization cubes of all morphisms $W \rightarrow 0$. These morphisms are in bijection with elements of $\stt W$.  An example of a cell $e(W)$ for a wide subcategory of rank 2 is shown in Figure \ref{fig:homotopy}. We must show that $e(W)$ is a disk of dimension $k$ attached to lower dimensional cells along its boundary.
	
	Let $X(W) := \bigsqcup_{U \in \stt(W)}[U]$. We remark that a (cubical) face $[U'] \subseteq [U]$ corresponds to a morphism $[U']:W' \rightarrow W''$ for some $W'' \subseteq W' \subseteq W$. We define an equivalence relation $\sim_1$ on $X(W)$ by identifying faces corresponding to the same morphism. Thus by definition, $e(W) \equiv X(W)/\sim_1$.
	
	Now consider the under category $W\backslash \mathfrak{W}(\Lambda)$, whose objects are pairs $(W',[U])$ with $W' \in \mathsf{wide}\Lambda$ and $[U] \in \mathrm{Hom}_{\mathfrak{W}(\Lambda)}(W,W')$, and whose morphisms $(W',[U])\rightarrow (W'',[V])$ are morphisms $[U']:W' \rightarrow W''$ in $\mathfrak{W}(\Lambda)$ so that $[V] = [U']\circ[U]$. By Lemma \ref{lem:cubeembedding}(b), there is a unique morphism $(W',[U]) \rightarrow (W'',[V])$ in $W\backslash \mathfrak{W}(\Lambda)$ if and only if $U \in \mathsf{add} V$, otherwise there are no such morphisms. This gives $W\backslash \mathfrak{W}(\Lambda)$ the structure of a poset category. The classifying space of this category is the cone (with cone point $(W,[0])$) on the simplicial complex $\Sigma W$ from Lemma \ref{lem:sphere}. In particular, this implies that $\mathcal{B} W\backslash \mathfrak{W}(\Lambda)$ is homeomorphic to a $k$-disk. Geometrically, we also see that $\mathcal{B} W\backslash \mathfrak{W}(\Lambda) \equiv X(W)/\sim_{2}$, where $\sim_2$ only identifies faces corresponding to the same morphism with source $W$ (and their subfaces). Let $\sim_3$ be the equivalence relation on $\mathcal{B} W\backslash \mathfrak{W}(\Lambda)$ which identifies all remaining faces corresponding to the same morphism, so $e(W) \equiv \mathcal{B} W\backslash \mathfrak{W}(\Lambda)/\sim_3$.
	
	Consider again the vertex $W = (W,[0]) \in \mathcal{B} W\backslash \mathfrak{W}(\Lambda)$. By Lemma  \ref{lem:sphere} and the discussion above, the link of $W$ is homeomorphic to a $(k-1)$ sphere. Thus $W$ is an interior vertex of $\mathcal{B} W\backslash \mathfrak{W}(\Lambda)$. Moreover, if $v = (W',[U])$ is another vertex, then the star of $v$ is homeomorphic to the join of $[U]$ and $\mathcal{B} W'\backslash \mathfrak{W}(\Lambda)$. Thus the link of $v$ is homeomorphic to the join of a ($rk[U]-1$)-disk and a ($k- rk[U]$-1)-sphere, which is a $(k-1)$-disk. This means $v$ is a boundary vertex of $\mathcal{B} W\backslash \mathfrak{W}(\Lambda)$. Now any simplex in $\mathcal{B} W\backslash \mathfrak{W}(\Lambda)$ which contains the vertex $W$ is necessarily in its own equivalence class (under $\sim_3$). Thus the only identifications made by $\sim_3$ are of simplices containing only boundary vertices; that is, simplices on the boundary of $\mathcal{B} W\backslash \mathfrak{W}(\Lambda)$. Moreover, these identified simplices are equivalence classes of (factorizations of) morphisms of rank strictly less than $k$; i.e., $e(W)$ is attached to lower dimensional cells.
\end{proof}

\begin{ex}\label{ex:cell}
Let $\Lambda = K(2\rightarrow 1)$. Then $\mathcal{B}\mathfrak{W}(\Lambda)$ has a single 0-cell, three 1-cells, and a single 2-cell. The 0-cell corresponds to the wide subcategory 0. Moreover, for each of the three bricks $T \in \mathsf{brick}\Lambda$, the 1-cell $e(\mathsf{Filt} T)$ corresponds to the pair of morphisms $0\xleftarrow{[T]} \mathsf{Filt} T \xrightarrow{[T[1]]}0$. Finally, the 2-cell $e(\mathsf{mod} \Lambda)$ corresponds to the under category $(\mathsf{mod}\Lambda\backslash\mathfrak{W}(\Lambda))$ and is shown in Figure~\ref{fig:cell}.
	
\begin{figure}
{\scriptsize
\begin{tikzpicture}[scale=0.7]
	\begin{scope}[thick,decoration={
	markings,
	mark=at position 0.5 with {\arrow[scale=1.5]{>}}}
	]
		\draw[postaction={decorate}] (0,0)--(6,0);
		\draw[postaction={decorate}] (0,0)--(3,3);
		\draw[postaction={decorate}] (0,0)--(3,-3);
		\draw[postaction={decorate}] (0,0)--(-3,-2);
		\draw[postaction={decorate}] (0,0)--(-3,2);
		\draw[postaction={decorate}] (0,0)--(6,2);
		\draw[postaction={decorate}] (0,0)--(6,-2);
		\draw[postaction={decorate}] (0,0)--(-6,0);
		\draw[postaction={decorate}] (0,0)--(0,4);
		\draw[postaction={decorate}] (0,0)--(0,-4);

		\draw[postaction={decorate}] (6,0)--(6,2);
		\draw[postaction={decorate}] (6,0)--(6,-2);
		\draw[postaction={decorate}] (-3,-2)--(0,-4);
		\draw[postaction={decorate}] (-3,-2)--(-6,0);
		\draw[postaction={decorate}] (-3,2)--(-6,0);
		\draw[postaction={decorate}] (-3,2)--(0,4);

	\end{scope}
	
	\begin{scope}[thick,decoration={
	markings,
	mark=at position 0.55 with {\arrow[scale=1.5]{>}}}
	]
	
		\draw[postaction={decorate}] (3,3)--(0,4);
		\draw[postaction={decorate}] (3,3)--(6,2);
		
		\draw[postaction={decorate}] (3,-3)--(6,-2);
		\draw[postaction={decorate}] (3,-3)--(0,-4);
	\end{scope}
	
	\node at (6,0) [draw,fill=white] {$(\mathsf{Filt} P_2,[S_2])$};
	\node at (3,3) [draw,fill=white] {$(\mathsf{Filt} P_1,[P_2])$};
	\node at (3,-3) [draw,fill=white] {$(\mathsf{Filt} S_2,[P_1[1]])$};
	\node at (-3,-2) [draw,fill=white] {$(\mathsf{Filt} P_1,[P_2[1]])$};
	\node at (-3,2) [draw,fill=white] {$(\mathsf{Filt} S_2,[P_1])$};
	\node at (6,2) [draw,fill=white] {$(0,[P_2\sqcup S_2])$};
	\node at (6,-2) [draw,fill=white] {$(0,[S_2\sqcup P_1[1]])$};
	\node at (-6,0) [draw,fill=white] {$(0,[P_1\sqcup P_2[1]])$};
	\node at (0,4) [draw,fill=white] {$(0,[P_1\sqcup P_2])$};
	\node at (0,-4) [draw,fill=white] {$(0,[P_1[1]\sqcup P_2[1]])$};
	
	\node at (0,0) [draw,fill=white] {$(\mathsf{mod}\Lambda,[0])$};
	
	\draw[very thick,color=blue,dashed] (0,1.5)--(-1,0.75)--(-1.5,0)--(-1,-0.75)--(0,-1.5)--(1,-1)--(1.5,-0.5)--(1.5,0)--(1.5,0.5)--(1,1)--(0,1.5);
	
	\draw[very thick,color=orange,dotted,bend right] (-4,1.33)--(-3,0.5)--(-2,1.33)--(-1.25,1.95)--(-2,2.667);
	
	\node at (-4,1.33) [draw,fill,circle,scale=0.6,color=orange]{};
	\node at (-2,1.33) [draw,fill,circle,scale=0.6,color=orange]{};
	\node at (-2,2.667) [draw,fill,circle,scale=0.6,color=orange]{};

\end{tikzpicture}
}
\caption{The 2-cell $e(\mathsf{mod}\Lambda)$ for $\Lambda = K(2\leftarrow 1)$. We see that the link of the vertex $(\mathsf{mod}\Lambda,[0])$, drawn in dashed blue, is homeomorphic to a 1-sphere. This means $(\mathsf{mod}\Lambda,[0])$ lies in the interior of the cell. Likewise, the link of the vertex $(\mathsf{Filt} S_2,[S_1])$, drawn in dotted orange, is homeomorphic to a 1-disk. This means $(\mathsf{Filt}(S_2),[S_1])$ lies on the boundary of the disk. Details of the attaching map for this 2-cell are given in Example \ref{ex:cell}}\label{fig:cell}
\end{figure}
We see that the link of the central vertex $(\mathsf{mod}\Lambda,[0])$, drawn in dashed blue, is indeed homeomorphic to a 1-sphere. Moreover, the link of the vertex $(\mathsf{Filt} S_2,[P_1])$, drawn in dotted orange, is the join of a 0-sphere (corresponding to the vertices on the arrows $(\mathsf{Filt} S_2,[P_1]) \rightarrow (0,[P_1\sqcup P_2])$ and $(\mathsf{Filt} S_2,[P_1]) \rightarrow (0,[P_1\sqcup P_2[1])$) and a 0-disk (corresponding to the vertex on the arrow $[S_1]:(W,[0])\rightarrow (\mathsf{Filt} S_2,[P_1])$). This is homeomorphic to a 1-disk, and shows that $(\mathsf{Filt} S_2,[P_1])$ lies on the boundary of $e(\mathsf{mod}\Lambda)$. In addition, we see that the 0-sphere in question can be identified with the link of $(\mathsf{Filt} S_2,[0])$ in $\mathcal{B} (\mathsf{Filt} S_2\backslash \mathfrak{W}(\mathsf{Filt} S_2))$.

In this example, the equivalence relation $\sim_3$ from the proof of Proposition \ref{prop:CW} makes the following identifications:
\begin{itemize}
	\item The five vertices of the form $(0,[U_1\sqcup U_2])$ are identified. These vertices correspond to the factorization cube of the identity morphism $[0]:0\rightarrow 0$ in $\mathfrak{W}(\Lambda)$.
	\item The 1-cells corresponding to $(0,[P_1\sqcup P_2]) \leftarrow (\mathsf{Filt} S_2,[P_1]) \rightarrow (0,[P_1\sqcup P_2[1]])$ and $(0,[S_2\sqcup P_1[1]]) \leftarrow (\mathsf{Filt} S_2,[P_1[1]]) \rightarrow (0,[P_1[1]\sqcup P_2[1]])$ are identified. These both correspond to the union of the factorization cubes of the morphisms $[S_2]:\mathsf{Filt} S_2 \rightarrow 0$ and $[S_2[1]]:\mathsf{Filt} S_2\rightarrow 0$ in $\mathfrak{W}(\Lambda)$.
	\item The 1-cells corresponding to $(0,[P_1\sqcup P_2]) \leftarrow (\mathsf{Filt} P_1,[P_2]) \rightarrow (0,[P_2\sqcup S_2])$ and $(0,[P_1\sqcup P_2[1]]) \leftarrow (\mathsf{Filt} P_1,[P_2[1]]) \rightarrow (0,[P_1[1]\sqcup P_2[1]])$ are identified. These both correspond to the union of the factorization cubes of the morphisms $[P_1]:\mathsf{Filt} P_1 \rightarrow 0$ and $[P_1[1]]:\mathsf{Filt} P_1\rightarrow 0$ in $\mathfrak{W}(\Lambda)$.
\end{itemize}
In particular, none of simplices which contain the vertex $(\mathsf{mod}\Lambda,[0])$ are identified, and all identifications are on the boundary of the disk $e(\mathsf{mod}\Lambda)$.
\end{ex}

We now recall that the generators of the fundamental group of a CW complex correspond to the 1-cells. From Section \ref{sec:last}, the set of rank one wide subcategories is precisely $\{\mathsf{Filt} S|S\in \mathsf{brick}\Lambda\}.$ Moreover, for any brick $S$ in $\mathsf{mod}\Lambda$, there are exactly two morphisms $\mathsf{Filt} S\rightarrow 0$ in $\mathfrak{W}(\Lambda)$, namely $[\mathsf{br}(S)]$ and $[\mathsf{br}(S)[1]]$. Considering these morphisms as directed paths in $\mathcal{B}\mathfrak{W}(\Lambda)$, it follows that the 1-cell $e(\mathsf{Filt} S)$ is a copy of $S^1$, with fundamental group generated by the loop $[\mathsf{br}(S)]\inv[\mathsf{br}(S)[1]]$, where we compose paths left-to-right. By Proposition \ref{prop:CW}, this proves the following.

\begin{prop}
\label{prop:generators}
The fundamental group $\pi_1(\mathcal{B}\mathfrak{W}(\Lambda))$ is generated by $$\{X_S:=[\mathsf{br}(S)]\inv[\mathsf{br}(S)[1]]|S \in \mathsf{brick}\Lambda\}.$$
\end{prop}

We now need only describe the relations, which are given by the 2 cells. Let $W = \mathsf{Filt} (S_1\sqcup S_2)$ be a wide subcategory of rank 2. We now observe that for $T,T'\in\mathsf{brick} W$, we have that $[\mathsf{br}(T)]$ and $[\mathsf{br}(T')[1]]$ are compatible as last morphisms if and only if $T\sqcup T'[1]\in\smc W$. This means $e(W)$ and the polygon $\mathcal{P}(S_1,S_2)$ have the form shown below in Figure \ref{fig:homotopy}, where $e(W)$ gives a homotopy 
$$X_{S_1}X_{T_1}\cdots X_{T_k}X_{S_2}\sim X_{S_2}X_{T'_1}\cdots X_{T'_l}X_{S_1}.$$
\begin{figure}
\begin{tikzpicture}[scale=0.9]
	\begin{scope}[thick,decoration={
	markings,
	mark=at position 0.5 with {\arrow[scale=1.5]{>}}}
	]
		\draw[postaction={decorate},dashed] (0,0)--(-2,4);
		\draw[postaction={decorate},dashed] (0,0)--(2,4);
		\draw[postaction={decorate},dashed] (0,0)--(-4,1.5);
		\draw[postaction={decorate},dashed] (0,0)--(4,1.5);
		\draw[postaction={decorate},dashed] (0,0)--(4,-1.5);
		\draw[postaction={decorate},dashed] (0,0)--(-4,-1.5);
		\draw[postaction={decorate},dashed] (0,0)--(2,-4);
		\draw[postaction={decorate},dashed] (0,0)--(-2,-4);
		
		\draw[postaction={decorate}] (-2,4)--(0,5) node [midway,anchor=south east]{$S_1$};
		\draw[postaction={decorate}] (2,4)--(0,5) node [midway,anchor=south west]{$S_2$};
		\draw[postaction={decorate}] (-2,4)--(-4,3) node [midway,anchor=south east]{$S_1[1]$};
		\draw[postaction={decorate}] (-4,1.5)--(-4,3) node [midway,anchor=east]{$T_1$};
		\draw[postaction={decorate}] (2,4)--(4,3) node [midway,anchor=south west]{$S_2[1]$};
		\draw[postaction={decorate}] (4,1.5)--(4,3) node [midway,anchor=west]{$T'_1$};
		
		\draw[postaction={decorate}] (2,-4)--(0,-5) node [midway,anchor=north west]{$S_1[1]$};
		\draw[postaction={decorate}] (-2,-4)--(0,-5) node [midway,anchor=north east]{$S_2[1]$};
		\draw[postaction={decorate}] (2,-4)--(4,-3) node [midway,anchor=north west]{$S_1$};
		\draw[postaction={decorate}] (4,-1.5)--(4,-3) node [midway,anchor=west]{$T'_l[1]$};
		\draw[postaction={decorate}] (-2,-4)--(-4,-3) node [midway,anchor=north east]{$S_2$};
		\draw[postaction={decorate}] (-4,-1.5)--(-4,-3) node [midway,anchor=east]{$T_k[1]$};
	\end{scope}
	
	\begin{scope}[thick,decoration={
	markings,
	mark=at position 0.7 with {\arrow[scale=1.5]{>}}}
	]
		\draw[postaction={decorate}] (-4,1.5)--(-4,0.5) node [near end,anchor=east]{$T_1[1]$};
		\draw[postaction={decorate}] (4,-1.5)--(4,-0.5) node [near end,anchor=west]{$T'_l$};
		\draw[postaction={decorate}] (-4,-1.5)--(-4,-0.5) node [near end,anchor=east]{$T_k$};
		\draw[postaction={decorate}] (4,1.5)--(4,0.5) node [near end,anchor=west]{$T'_1[1]$};
	\end{scope}
	
	\node at (-4,0) [rotate=90] {$\cdots$};
	\node at (4,0) [rotate=90] {$\cdots$};
	
	\node at (0,0) [draw,fill=white,very thick,dotted] {$W$};
	\node at (0,5) [draw,fill=white] {$0$};
	\node at (-4,3) [draw,fill=white] {$0$};
	\node at (4,3) [draw,fill=white] {$0$};
	\node at (-4,-3) [draw,fill=white] {$0$};
	\node at (4,-3) [draw,fill=white] {$0$};
	\node at (0,-5) [draw,fill=white] {$0$};
	\node at (-2,4) [draw,fill=white,very thick,dotted] {$\mathsf{Filt} S_1$};
	\node at (2,4) [draw,fill=white,very thick,dotted] {$\mathsf{Filt} S_2$};
	\node at (-2,-4) [draw,fill=white,very thick,dotted] {$\mathsf{Filt} S_2$};
	\node at (2,-4) [draw,fill=white,very thick,dotted] {$\mathsf{Filt} S_1$};
	\node at (4,1.5) [draw,fill=white,very thick,dotted] {$\mathsf{Filt} T'_1$};
	\node at (-4,1.5) [draw,fill=white,very thick,dotted] {$\mathsf{Filt} T_1$};
	\node at (4,-1.5) [draw,fill=white,very thick,dotted] {$\mathsf{Filt} T'_l$};
	\node at (-4,-1.5) [draw,fill=white,very thick,dotted] {$\mathsf{Filt} T_k$};
	
	\node at (0,-6) {$e(\mathsf{Filt}(S_1\sqcup S_2))$};
	
	\begin{scope}[shift={(6,-5)}]
		\begin{scope}[thick,decoration={
	markings,
	mark=at position 0.5 with {\arrow[scale=1.5]{>}}}
	]
	\draw[postaction={decorate}] (0,-0.5)--(-1.5,-1.5) node [pos=0.65,anchor=south east]{$S_1$};
	\draw[postaction={decorate}] (0,-0.5)--(1.5,-1.5) node [pos=0.65,anchor=south west]{$S_2$};
	\draw[postaction={decorate}] (-1.5,-1.5)--(-1.5,-2.5) node [midway,anchor=east]{$T_1$};
	\draw[postaction={decorate}] (1.5,-1.5)--(1.5,-2.5) node [midway,anchor=west]{$T'_1$};
	\draw[postaction={decorate}] (-1.5,-3)--(-1.5,-4) node [midway,anchor=east]{$T_k$};
	\draw[postaction={decorate}] (1.5,-3)--(1.5,-4) node [midway,anchor=west]{$T'_l$};
	\draw[postaction={decorate}] (-1.5,-4)--(0,-5) node [pos=0.4,anchor=north east]{$S_2$};
	\draw[postaction={decorate}] (1.5,-4)--(0,-5) node [pos=0.4,anchor=north west]{$S_1$};
	\end{scope}
	
	\node at (0,-0.5) [draw,fill=white,anchor=south]{$\mathsf{Filt}\mathsf{Fac}(S_1\sqcup S_2$)};
	\node at (0,-5) [draw,fill=white]{$0$};
	
	\node at (-1.5,-1.5) [draw,fill,circle,scale=0.6]{};
	\node at (1.5,-1.5) [draw,fill,circle,scale=0.6]{};
	\node at (-1.5,-4) [draw,fill,circle,scale=0.6]{};
	\node at (1.5,-4) [draw,fill,circle,scale=0.6]{};
	
	\node at (-1.5,-2.7)[rotate=90]{\small $\cdots$};
	\node at (1.5,-2.7)[rotate=90]{\small $\cdots$};
	\node at (0,-6){$\mathcal{P}(S_1,S_2)$};
	\end{scope}
\end{tikzpicture}
\caption{The 2-cell $e(\mathsf{Filt}(S_1\sqcup S_2))$ (top left) and the corresponding polygon in $\mathsf{tors}\Lambda$ (bottom right). In the 2-cell $e(\mathsf{Filt}(S_1\sqcup S_2))$, wide subcategories different from 0 are dotted and arrows corresponding to morphisms in $\mathfrak{W}(\mathsf{Filt}(S_1\sqcup S_2))$ with target different from 0 are dashed to emphasize that they are not themselves cells in the CW-structure. The attaching map for the 2-cell $e(\mathsf{Filt}(S_1\sqcup S_2))$ is given by identifying arrows on the boundary with the same label. Moreover, the 1-cells on the boundary of $e(\mathsf{Filt}(S_1\sqcup S_2))$ correspond precisely to the brick labels of the polygon, in the same order.}
\label{fig:homotopy}
\end{figure}

We have shown that the relations of $\pi_1(\mathcal{B}\mathfrak{W}(\Lambda))$ are precisely the polygon relations in the presentation of $G(\Lambda)$ (see Definition-Theorem \ref{defthm:picture}). Thus we have proven our second main theorem (Theorem B in the introduction).
\begin{thm}
\label{thm:pi1}
The fundamental group $\pi_1(\mathcal{B}\mathfrak{W}(\Lambda))$ is isomorphic to $G(\Lambda)$, the picture group of~$\Lambda$.
\end{thm}


\subsection{Faithful Group Functors for Nakayama Algebras}
\label{sec:nakaymapictures}

The remainder of this section is aimed at constructing a faithful group functor $\mathfrak{W}(\Lambda)\rightarrow G(\Lambda)$ in the case that $\Lambda$ is Nakayama. One of the key properties used to prove this functor is faithful is that the generator of $G(\Lambda)$ corresponding to each brick in $\mathsf{mod}\Lambda$ is nontrivial. We prove this by defining a morphism of groups $G(\Lambda) \rightarrow \mathcal{B}^*(\Lambda)$ which does not vanish on any of the generators. The group $\mathcal{B}^*(\Lambda)$ is the group of units of the \emph{brick algebra} of $\Lambda$, defined below in Definition-Theorem \ref{defthm:brickalg}. This is a variation of the Hall algebra of~$\Lambda$. As not all results in this section rely on the assumption that $\Lambda$ is Nakayama, we will emphasize precisely where this assumption is used.

\begin{defthm}
\label{defthm:brickalg}
Let $\Lambda$ be a Nakayama algebra. Then the \emph{brick algebra} of $\Lambda$ is the free $\mathbb{Z}$-module with basis $\mathsf{brick}\Lambda\cup\{1\}$ and with multiplication given by
	\begin{eqnarray*}
		1*1 &=& 1\\
		1*S &=& S*1 = S \textnormal{ for all }S \in \mathsf{brick}\Lambda\\
		S*T &=& \begin{cases} B \textnormal{ if }S\sqcup T\in\mathsf{sbrick}\Lambda, B \in \mathsf{brick}\Lambda\textnormal{ and }T \hookrightarrow B \twoheadrightarrow S \textnormal{ is exact}\\ 0 \textnormal{ otherwise}\end{cases}
	\end{eqnarray*}
	This is an associative algebra.
\end{defthm}

\begin{proof}
We first observe, by Lemmas \ref{lem:archomslong} and \ref{lem:archoms}, that for all pairs $S = M(a_S), T=M(a_T)$ of bricks in $\mathsf{mod}\Lambda$, we have $S*T = 0$ unless the following conditions are met:
\begin{enumerate}
	\item $t(a_S) = s(a_T)$
	\item $l(a_S) + l(a_T) \leq \min(l(s(a_S)),n)$
\end{enumerate}
Moreover, if these conditions are met, then $B = M(a_Sa_T)$ is the only brick with an exact sequence $T\hookrightarrow B \twoheadrightarrow S$. We conclude that this multiplication is well-defined.

We now wish to show that this multiplication is associative. Let $S = M(a_S), T = M(a_T), R = M(a_R)$ be bricks in $\mathsf{mod}\Lambda$. Assume first that $S*T = 0$. Now if $T*R = 0$ then we are done. Otherwise, we have $T*R = M(a_Ta_R)$ and either $t(a_S) \neq s(a_T) = s(a_Ta_R)$ or
\begin{eqnarray*}
	l(a_S)+ l(a_Ta_R) &>& l(a_S) + l(a_T)\\
	&>& \min(l(s(a_S)),n).
\end{eqnarray*}
We conclude that $S*(T*R) = 0$.

Likewise, assume that $T*R = 0$ and $S*T \neq 0$. Then we have $t(a_Sa_T) = t(a_T) \neq s(a_R)$ or 
	\begin{eqnarray*}
	l(a_Sa_T) + l(a_R) &=& l(a_S) + l(a_T) + l(a_R)\\
	&>& l(a_S) + \min(l(s(a_T)),n)\\
	&\geq& \min(l(s(a_S)),n).
	\end{eqnarray*}
We conclude that $(S*T)*R = 0$.

Finally, assume that $S*T \neq 0$ and $T*R \neq 0$. Thus we have $t(a_S) = s(a_T)$ and $t(a_T) = s(a_R)$. Therefore $(S*T)*R = 0$ if and only if $l(a_S) + l(a_T) + l(a_R) > \min(l(s(a_S)),n)$ if and only if $S*(T*R) = 0$. Otherwise, we must have that $(S*T)*R = M(a_Sa_Ta_R) = S*(T*R)$. We conclude that $\mathcal{B}(\Lambda)$ is an associative algebra.
\end{proof}

We now prove the following lemmas, which will be critical in what follows. Recall that $J(N\sqcup Q[1])$ is the Jasso category of the support $\tau$-rigid pair $N \sqcup Q[1]$ (see Theorem \ref{thm:basicWide}) and $\mathcal{E}_{N\sqcup Q[1]}$ is the bijection described in Theorem \ref{thm:littlebijections}.

\begin{lem} Let $\Lambda$ be an arbitrary $\tau$-tilting finite algebra.
\label{lem:welldef}
\smallskip
\begin{enumerate}[label=\upshape(\alph*)]
\item \cite[Thm. 4.12, Prop. 4.13]{demonet_lattice} Let $N\sqcup Q[1]$ be a support $\tau$-rigid pair for $\Lambda$. Then there is a label-preserving isomorphism of lattices
$$\mathcal{G}_{N\sqcup Q[1]}:[\mathsf{Fac} N,\lperp{(\tau N)}\cap Q^\perp]\rightarrow \mathsf{tors} (J(N\sqcup Q[1]))$$
given by
$$\mathcal{G}_{N\sqcup Q[1]}(\mathcal{T}) := J(N\sqcup Q[1])\cap\mathcal{T}.$$

\item Let $N\sqcup M\sqcup Q[1]\sqcup P[1]$ be a support $\tau$-rigid pair for $\Lambda$. Then we have
$$\mathcal{G}_{N\sqcup Q[1]}(\mathsf{Fac} (N\sqcup M)) = \mathsf{Fac}(\mathcal{E}_{N\sqcup Q[1]}(M'))$$
where $M'$ is the direct sum of the indecomposable direct summands of $M$ sent to modules by $\mathcal{E}_{N\sqcup Q[1]}$ (see Theorem \ref{thm:littlebijections}).
\end{enumerate}
\end{lem}

\begin{proof}
(b) First recall that $\mathsf{Fac} (N \sqcup M) = \mathsf{Fac} (N \sqcup M')$ by Theorem \ref{thm:littlebijections}. Thus using the definitions, we have
\begin{eqnarray*}
	\mathcal{G}_{N\sqcup Q[1]}(\mathsf{Fac} (N\sqcup M)) &=&  J(N\sqcup Q[1])\cap \mathsf{Fac}(N\sqcup M')\\
	&=& N^\perp\cap\lperp{(\tau N)}\cap Q^\perp\cap\mathsf{Fac}(N\sqcup M')\\
	&=&N^\perp\cap \lperp{(\tau N)}\cap Q^\perp\cap\mathsf{Fac} M'.
\end{eqnarray*}
We now claim $\mathsf{Fac} M' \subseteq \lperp{(\tau N)}\cap Q^\perp$. Indeed, let $L \in \mathsf{Fac} M$. Since $M'$ is a direct summand of $M$ and $N\sqcup M \sqcup Q[1] \sqcup P[1]$ is a support $\tau$-rigid pair, we know $M' \in \lperp{(\tau N)} \cap Q^\perp$. Since $Q$ is projective, this implies that $L \in \lperp{(\tau N)} \cap Q^\perp$ as well.

This means we need only show that $N^\perp\cap\mathsf{Fac} M' = \mathsf{Fac}(\mathcal{E}_{N\sqcup Q[1]}(M'))$. By definition, $\mathcal{E}_{N\sqcup Q[1]}(M') = M'/\mathrm{rad}(N,M') \in N^\perp \cap \mathsf{Fac} M'$, so let $L \in N^\perp\cap\mathsf{Fac} M'$. Thus there is a positive integer $t$ and an epimorphism $q:(M')^t \rightarrow L$. Necessarily, $\mathrm{rad}(N,(M')^t) = \mathrm{rad}(N,M')^t$ is contained in $\mathrm{ker} q$, which means $q$ factors through $(M'/\mathrm{rad}(N,M'))^t = \left(\mathcal{E}_{N\sqcup Q[1]}(M')\right)^t$. That is, $L \in \mathsf{Fac}\left(\mathcal{E}_{N\sqcup Q[1]}(M')\right)$. This completes the proof.
\end{proof}

{\begin{lem}
\label{lem:polygons}\
\begin{enumerate}[label=\upshape(\alph*)]
	\item \cite[Thm. 2.2.6]{barnard_minimal}\label{prop:bricklabels}
Suppose $\mathcal{T}\rightarrow\mathcal{T}'$ is an arrow in $\mathsf{Hasse}(\mathsf{tors}\Lambda)$ with brick label $B$. Then $\mathcal{T} = \mathsf{Filt}(\mathcal{T}'\sqcup\{B\})$ and every proper factor of $B$ lies in $\mathcal{T}'$.
	\item Let $\Lambda$ be a Nakayama algebra and let $S \sqcup T \in \mathsf{sbrick}\Lambda$ be a semibrick with $\mathsf{rk}(S\sqcup T) = 2$. Let $(S,B_1,\ldots,B_k,T)$ be a side of the polygon $\mathcal{P}(S,T)$. If $\mathrm{Ext}(T,S) = 0$, then $k = 0$. Otherwise, $k = 1$ and there is an exact sequence $S\hookrightarrow B_1 \twoheadrightarrow T$.
	\end{enumerate}
\end{lem}

\begin{proof}
	(b) It follows from Theorem \ref{thm:polygons}(c) that each for $i\neq j$, none of $B_i,B_j,S$, and $T$ are isomorphic. Now by definition, there is a sequence of arrows in $\mathsf{Hasse}(\mathsf{tors}\Lambda)$ with brick labels as follows:
	$$\mathcal{T}\xrightarrow{S}\mathcal{T}_S\xrightarrow{B_1}\mathcal{T}_{B_1}\xrightarrow{B_2}\cdots\xrightarrow{B_k}\mathcal{T}_{B_k}\xrightarrow{T}\mathcal{T}'.$$
	
	Fix some $B_i$ and suppose there exists a nonzero morphism $f:B_i \rightarrow S$. By (a) we then have that $\mathrm{coker}(f) \in \mathcal{T}_{S}$ and $\mathrm{Im}(f) \in \mathcal{T}_{B_{i-1}} \subseteq \mathcal{T}_S$. Since $\mathcal{T}_S$ is closed under extensions, this implies that $S \in \mathcal{T}_S$, a contradiction. We conclude that $\mathrm{Hom}(B_i,S) = 0$. A similar argument shows that $\mathrm{Hom}(T,B_i) = 0$.
	
	Now recall from Theorem \ref{thm:polygons}(c) that $B_i \in \mathsf{Filt}(S\sqcup T)$. As $S$ and $T$ are simple objects in this category, we conclude that there must be morphisms $S\hookrightarrow B_i$ and $B_i \twoheadrightarrow T$. Moreover, by Proposition \ref{prop:nakayamabricks}, $S$ and $T$ can only occur once in the composition series for $B_i$. That is, the length of $B_i$ in $\mathsf{Filt}(S\sqcup T)$ is 2 and thus the sequence $S\hookrightarrow B_i \twoheadrightarrow T$ is exact. The result then follows as in the proof of Definition-Theorem \ref{defthm:brickalg}.
\end{proof}

\begin{lem}
\label{lem:faithful}
Let $\Lambda$ be a Nakayama algebra.
\begin{enumerate}[label=\upshape(\alph*)]
	\item Let $S\ncong S'$ be a pair of bricks in $\mathsf{mod}\Lambda$. Then $e \neq X_S \neq X_{S'} \in G(\Lambda)$.
	\item Let $\mathcal{T}\neq\mathcal{T}'$ be a pair of torsion classes in $\mathsf{tors}\Lambda$. Then $g_{\mathcal{T}} \neq g_{\mathcal{T}'} \in G(\Lambda)$.
\end{enumerate}
\end{lem}

\begin{proof}
(a) We wish to define a morphism of groups $\phi: G(\Lambda) \rightarrow \mathcal{B}^*(\Lambda)$ by $\phi(X_S)=1+S$ and $\phi(X_S\inv) = 1-S$. We will show that $\phi$ is well-defined and is in fact a morphism of groups. This will imply that each $X_S$ is nontrivial in $G(\Lambda)$.

First, we observe that for all $S \in \mathsf{brick}\Lambda$, we have 
$$\phi(X_SX_S\inv) = (1+S)*(1-S) = 1 - S*S = 1$$
because $S$ cannot have a self-extension which is a brick.

We now need only show that $\phi$ preserves the polygon relations. Let $\mathcal{P}(S,T)$ be a polygon in $\mathsf{Hasse}(\mathsf{tors}\Lambda)$. By Lemma \ref{lem:polygons}(b) above, there are only four possibilities:
\begin{enumerate}
	\item The sides of the polygon $\mathcal{P}(S,T)$ are $(S,T)$ and $(T,S)$. In this case, $\mathrm{Ext}(S,T) = 0 = \mathrm{Ext}(T,S)$.
	\item The sides of the polygon $\mathcal{P}(S,T)$ are $(S,B,T)$ and $(T,S)$ where $S\hookrightarrow B \twoheadrightarrow T$ is exact. In this case, $\mathrm{Ext}(S,T) = 0$.
	\item The sides of the polygon $\mathcal{P}(S,T)$ are $(S,T)$ and $(T,B,S)$ where $T\hookrightarrow B \twoheadrightarrow S$ is exact. In this case, $\mathrm{Ext}(T,S) = 0$.
	\item The sides of the polygon $\mathcal{P}(S,T)$ are $(S,B,T)$ and $(T,B',S)$ where $S\hookrightarrow B \twoheadrightarrow T$ and $T\hookrightarrow B' \twoheadrightarrow S$ are exact.
\end{enumerate}

We only prove case (2), as the other cases follow by nearly identical arguments. It follows from the proof of Lemma \ref{lem:polygons}(b) that neither $S \sqcup B$ nor $B \sqcup T$ is a semibrick. Thus we have
\begin{eqnarray*}
	\phi(X_SX_BX_T) &=& (1+S)*(1+B)*(1+T)\\
		&=& 1 + S + B + T\\
		&=& 1 + S + T + T*S\\
		&=& (1+T)*(1+S)\\
		&=& \phi(X_TX_S)
\end{eqnarray*}
We conclude that $\phi$ is a morphism of groups and hence $e \neq X_S \neq X_{S'} \in G(\Lambda)$ for all $S, S' \in \mathsf{brick}\Lambda$.

(b) Let $\mathcal{T},\mathcal{T}' \in \mathsf{tors}\Lambda$ and let $(S_1,\ldots,S_k), (S'_1,\ldots,S'_l)$ label maximal length paths $\mathcal{T}\rightarrow \mathcal{T}\wedge \mathcal{T}', \mathcal{T}'\rightarrow \mathcal{T}\wedge\mathcal{T}'$ in $\mathsf{tors}(\Lambda)$. Assume that $X_{S_1}\cdots X_{S_k} = X_{S'_1}\cdots X_{S'_l} \in G(\Lambda)$ (and hence $g_\mathcal{T} = g_{\mathcal{T}'}$). If $\max(k,l) > 0$, let $S \in \{S_1,\ldots,S_k,S'_1,\ldots,S'_l\}$ be of minimal length. By applying the morphism $\phi$, we then have
$$(1+S_1)*\cdots *(1+S_k) = (1+S'_1)*\cdots * (1+S'_l).$$
Now given $R, R' \in \mathsf{brick}\Lambda$, either $R*R'$ is 0 or has length $l(R) + l(R')$. Thus the above equality implies that $S$ must occur in both $(S_1,\ldots,S_k)$ and $(S'_1,\ldots,S'_l)$. Thus $\mathcal{T} \wedge\mathcal{T}' \subsetneq\mathsf{Filt}\mathsf{Fac}(\mathsf{brick}(\mathcal{T}\wedge\mathcal{T}')\sqcup S) \subset \mathcal{T},\mathcal{T}'$, a contradiction. We conclude that $\max(k,l) = 0$, and hence $\mathcal{T} = \mathcal{T}\wedge\mathcal{T}' = \mathcal{T}'$.
\end{proof}

We are now ready to define the faithful group functor.

\begin{thm}
\label{thm:faithful}
Let $\Lambda$ be any $\tau$-tilting finite algebra which satisfies the conclusion of Lemma \ref{lem:faithful} and let $W,W' \in \mathsf{wide}\Lambda$. For each $[M\sqcup P[1]]\in \mathrm{Hom}_{\mathfrak{W}(\Lambda)}(W,W')$, let $(S_1,\ldots,S_k)$ label a sequence $\mathsf{Fac} M\rightarrow 0$ in $\mathsf{tors} W$. Then there is a faithful functor $F:\mathfrak{W}(\Lambda)\rightarrow G(\Lambda)$ given by
$$F[M\sqcup P[1]]:=X_{S_1}\cdots X_{S_k},$$
where $G(\Lambda)$ is considered as a groupoid with one object.
\end{thm}

\begin{proof}
	Let $[M\sqcup P[1]]\in \mathrm{Hom}_{\mathfrak{W}(\Lambda)}(W,W')$. By Theorem \ref{thm:basicWide}, write $W = J(N\sqcup Q[1])$ and let $M' \sqcup P'[1] = \mathcal{E}\inv_{N\sqcup Q[1]}(M\sqcup P[1])$. Now suppose $(S_1,\ldots,S_k)$ and $(S_1'\ldots,S_{k'}')$ both label sequences $\mathsf{Fac} M \rightarrow 0$ in $\mathsf{tors} W$. By Lemma \ref{lem:welldef}, both $(S_1,\ldots,S_k)$ and $(S_1',\ldots,S'_{k'})$ label sequences $\mathsf{Fac} (N \sqcup M')\rightarrow \mathsf{Fac}(N)$ in $\mathsf{tors}\Lambda$. Proposition \ref{prop:presentations} then implies $X_{S_1}\cdots X_{S_k} = X_{S'_1}\cdots X_{S'_{k'}}$ in $G(\Lambda)$. Thus $F$ is well-defined.
	
	To see that $F$ is faithful, let $[M_1\sqcup P_1[1]], [M_2\sqcup P_2[1]] \in \mathrm{Hom}_{\mathfrak{W}(\Lambda)}(W,W')$.  As before, let $M_1'\sqcup P_1'[1] =  \mathcal{E}\inv_{N\sqcup Q[1]}(M_1\sqcup P_1[1])$ and likewise for $M_2'\sqcup P_2'[1]$. Then by Proposition \ref{prop:presentations}, we can write $F[M_1\sqcup P_1[1]] = g_{\mathsf{Fac}(N\sqcup M_1')}\cdot g_{\mathsf{Fac} N}\inv$ and $F[M_2\sqcup P_2[1]] = g_{\mathsf{Fac}(N\sqcup M_2')}\cdot g_{\mathsf{Fac} N}\inv$. By Lemma \ref{lem:faithful}, these are the same group element if and only if $\mathsf{Fac}(N\sqcup M'_1) = \mathsf{Fac}(N\sqcup M'_2)$. This means the module part of $\mathcal{E}_{N\sqcup Q[1]}(M_1')$, which is $M_1$, is isomorphic to the module part of $\mathcal{E}_{N\sqcup Q[1]}(M_2')$, which is $M_2$. Since $J_W (M_1\sqcup P_1) = J_W(M_2\sqcup P_2)$, it follows that $P_1[1] \cong P_2[1]$, as well. 
	
Finally, to show $F$ preserves the composition law, let
$$W\xrightarrow{[M\sqcup P[1]]}W'\xrightarrow{[M'\sqcup P[1]]}W''$$
be a sequence of composable morphisms. Let $(S_1,\ldots,S_k)$ label a path $\mathsf{Fac} M\rightarrow 0$ in $\mathsf{tors} W$ and let $(T_1,\ldots,T_l)$ label a path $\mathsf{Fac} M'\rightarrow 0$ in $\mathsf{tors} W'$. It then follows from Lemma \ref{lem:welldef} that $(T_1,\ldots,T_l)$ labels a path
$$\mathsf{Fac}(M\sqcup\left(\mathcal{E}^W_{M\sqcup P[1]}\right)\inv(M'))\rightarrow \mathsf{Fac} M$$
in $\mathsf{Hasse}(\mathsf{tors} W)$. Thus the composition $(T_1,\ldots,T_l,S_1,\ldots,S_k)$ labels a path
$$\mathsf{Fac}(M\sqcup\left(\mathcal{E}^W_{M\sqcup P[1]}\right)\inv(M'))\rightarrow 0$$
in $\mathsf{Hasse}(\mathsf{tors} W)$. We conclude that $F([M'\sqcup P'[1]]\circ[M\sqcup P[1]]) = F[M'\sqcup P'[1]]\cdot F[M\sqcup P[1]]$.
\end{proof}

We are now ready to conclude our final main result (Theorem C in the introduction).

\begin{thm}
\label{thm:cat0}
For a Nakayama algebra $\Lambda$, the picture space (i.e., the classifying space of the category $\mathfrak{W}(\Lambda)$) is a locally $\mathrm{CAT}(0)$ cube complex and thus is a $K(\pi,1)$ for the picture group $G(\Lambda)$.
\end{thm}

\begin{proof}
	Let $\Lambda$ be a Nakayama algebra. We recall that we need only show that $\mathfrak{W}(\Lambda)$ satisfies conditions (a)-(c) in Proposition \ref{prop:npc}. As discussed in the paragraph following Proposition \ref{prop:npc}, condition (a) follows immediately from the definition of a $\tau$-rigid pair. Moreover, condition (b) follows from Theorem \ref{thm:nakayamapairwise} and Lemma \ref{lem:fewerlasts2}. Finally, we have shown condition (c) explicitly in Theorem \ref{thm:faithful}.
\end{proof}


\section{Examples of Theorem \ref{thm:pairwisesmc}}
\label{sec:examples}

\begin{ex}
Let $A_3$ be given straight orientation. The maximal arc patterns and corresponding 2-simple minded collections for $KA_3 \cong \Lambda(3,3,2,1)$ are shown in Table~\ref{table1}.
\end{ex}

{\footnotesize
\begin{center}
\begin{table}
\begin{tabular}{|c|c||c|c||c|c|}
\hline
MAP & 2-smc & MAP & 2-smc & MAP & 2-smc\\
\hline
\begin{tikzpicture}[scale=0.35,baseline=0]
	\draw (0,0) circle[radius=2cm];
	\draw (-.1,-.1) to (.1,.1);
	\draw (-.1,.1) to (.1,-.1);
		\draw[very thick] plot [smooth] coordinates{(0,2)(-1.73,-1)};
		\draw[very thick] plot [smooth] coordinates{(0,2)(1.73,-1)};
		\draw[very thick] plot [smooth] coordinates{(1.73,-1)(-1.73,-1)};
		\node at (0,2) [draw,fill,circle,scale=0.4,label=above:1]{};
		\node at (-1.73,-1) [draw,fill,circle,scale=0.4,label=south west:2]{};
		\node at (1.73,-1) [draw,fill,circle,scale=0.4,label=south east:3]{};
	\node at (0,2.5){};
	\node at (0,-2.5){};
\end{tikzpicture}
& {\scriptsize $1\sqcup2\sqcup3$} & 
\begin{tikzpicture}[scale=0.35,baseline=0]
	\draw (0,0) circle[radius=2cm];
	\draw (-.1,-.1) to (.1,.1);
	\draw (-.1,.1) to (.1,-.1);
		\draw[very thick,dashed] plot [smooth] coordinates{(0,2)(-1.73,-1)};
		\draw[very thick,dashed] plot [smooth] coordinates{(0,2)(1.73,-1)};
		\draw[very thick,dashed] plot [smooth] coordinates{(1.73,-1)(-1.73,-1)};
		\node at (0,2) [draw,fill,circle,scale=0.4,label=above:1]{};
		\node at (-1.73,-1) [draw,fill,circle,scale=0.4,label=south west:2]{};
		\node at (1.73,-1) [draw,fill,circle,scale=0.4,label=south east:3]{};
	\node at (0,2.5){};
	\node at (0,-2.5){};
\end{tikzpicture}
& {\scriptsize $1[1]\sqcup2[1]\sqcup3[1]$}
&\begin{tikzpicture}[scale=0.35,baseline=0]
	\draw (0,0) circle[radius=2cm];
	\draw (-.1,-.1) to (.1,.1);
	\draw (-.1,.1) to (.1,-.1);
		\draw[very thick] plot [smooth] coordinates{(0,2)(-1,1)(-1.73,-1)};
		\draw[very thick,dashed] plot [smooth] coordinates{(0,2)(0.5,0)(0,-0.5)(-0.5,0)(0,2)};
		\draw[very thick] plot [smooth] coordinates{(1.73,-1)(-1.73,-1)};
		\node at (0,2) [draw,fill,circle,scale=0.4,label=above:1]{};
		\node at (-1.73,-1) [draw,fill,circle,scale=0.4,label=south west:2]{};
		\node at (1.73,-1) [draw,fill,circle,scale=0.4,label=south east:3]{};
	\node at (0,2.5){};
	\node at (0,-2.5){};
\end{tikzpicture}
& {\scriptsize $1\sqcup2\sqcup\begin{matrix}1\\2\\3\end{matrix}[1]$} \\
\hline 
\begin{tikzpicture}[scale=0.35,baseline=0]
	\draw (0,0) circle[radius=2cm];
	\draw (-.1,-.1) to (.1,.1);
	\draw (-.1,.1) to (.1,-.1);
		\draw[very thick,dashed] plot [smooth] coordinates{(0,2)(-1.73,-1)};
		\draw[very thick] plot [smooth] coordinates{(0,2)(1.73,-1)};
		\draw[very thick,dashed] plot [smooth] coordinates{(1.73,-1)(-1.73,-1)};
		\node at (0,2) [draw,fill,circle,scale=0.4,label=above:1]{};
		\node at (-1.73,-1) [draw,fill,circle,scale=0.4,label=south west:2]{};
		\node at (1.73,-1) [draw,fill,circle,scale=0.4,label=south east:3]{};
	\node at (0,2.5){};
	\node at (0,-2.5){};
\end{tikzpicture}
& {\scriptsize $3\sqcup 1[1]\sqcup 2[1]$}

&\begin{tikzpicture}[scale=0.35,baseline=0]
	\draw (0,0) circle[radius=2cm];
	\draw (-.1,-.1) to (.1,.1);
	\draw (-.1,.1) to (.1,-.1);
		\draw[very thick,dashed] plot [smooth] coordinates{(0,2)(-1.73,-1)};
		\draw[very thick] plot [smooth] coordinates{(0,2)(1.73,-1)};
		\draw[very thick] plot [smooth] coordinates{(1.73,-1)(-1.73,-1)};
		\node at (0,2) [draw,fill,circle,scale=0.4,label=above:1]{};
		\node at (-1.73,-1) [draw,fill,circle,scale=0.4,label=south west:2]{};
		\node at (1.73,-1) [draw,fill,circle,scale=0.4,label=south east:3]{};
	\node at (0,2.5){};
	\node at (0,-2.5){};
\end{tikzpicture}
& {\scriptsize $2\sqcup3\sqcup1[1]$} & 
\begin{tikzpicture}[scale=0.35,baseline=0]
	\draw (0,0) circle[radius=2cm];
	\draw (-.1,-.1) to (.1,.1);
	\draw (-.1,.1) to (.1,-.1);
		\draw[very thick] plot [smooth] coordinates{(0,2)(0.5,0)(0,-0.5)(-0.5,0)(0,2)};
		\draw[very thick,dashed] plot [smooth] coordinates{(0,2)(1,1)(1.73,-1)};
		\draw[very thick,dashed] plot [smooth] coordinates{(1.73,-1)(0,-1.5)(-1.73,-1)};
		\node at (0,2) [draw,fill,circle,scale=0.4,label=above:1]{};
		\node at (-1.73,-1) [draw,fill,circle,scale=0.4,label=south west:2]{};
		\node at (1.73,-1) [draw,fill,circle,scale=0.4,label=south east:3]{};
	\node at (0,2.5){};
	\node at (0,-2.5){};
\end{tikzpicture}
& {\scriptsize $\begin{matrix}1\\2\\3\end{matrix}\sqcup2[1]\sqcup3[1]$}\\
\hline

\begin{tikzpicture}[scale=0.35,baseline=0]
	\draw (0,0) circle[radius=2cm];
	\draw (-.1,-.1) to (.1,.1);
	\draw (-.1,.1) to (.1,-.1);
		\draw[very thick] plot [smooth] coordinates{(0,2)(-1,1)(-1.73,-1)};
		\draw[very thick] plot [smooth] coordinates{(0,2)(1.73,-1)};
		\draw[very thick,dashed] plot [smooth] coordinates{(1.73,-1)(-0.5,-0.5)(0,2)};
		\node at (0,2) [draw,fill,circle,scale=0.4,label=above:1]{};
		\node at (-1.73,-1) [draw,fill,circle,scale=0.4,label=south west:2]{};
		\node at (1.73,-1) [draw,fill,circle,scale=0.4,label=south east:3]{};
	\node at (0,2.5){};
	\node at (0,-2.5){};
\end{tikzpicture}
& {\scriptsize $1\sqcup3\sqcup\begin{matrix}1\\2\end{matrix}[1]$} & 
\begin{tikzpicture}[scale=0.35,baseline=0]
	\draw (0,0) circle[radius=2cm];
	\draw (-.1,-.1) to (.1,.1);
	\draw (-.1,.1) to (.1,-.1);
		\draw[very thick] plot [smooth] coordinates{(0,2)(0.5,-0.5)(-1.73,-1)};
		\draw[very thick,dashed] plot [smooth] coordinates{(0,2)(1,1)(1.73,-1)};
		\draw[very thick,dashed] plot [smooth] coordinates{(0,2)(-1.73,-1)};
		\node at (0,2) [draw,fill,circle,scale=0.4,label=above:1]{};
		\node at (-1.73,-1) [draw,fill,circle,scale=0.4,label=south west:2]{};
		\node at (1.73,-1) [draw,fill,circle,scale=0.4,label=south east:3]{};
	\node at (0,2.5){};
	\node at (0,-2.5){};
\end{tikzpicture}
& {\scriptsize $\begin{matrix}2\\3\end{matrix}\sqcup1[1]\sqcup3[1]$}

&\begin{tikzpicture}[scale=0.35,baseline=0]
	\draw (0,0) circle[radius=2cm];
	\draw (-.1,-.1) to (.1,.1);
	\draw (-.1,.1) to (.1,-.1);
		\draw[very thick] plot [smooth] coordinates{(0,2)(-1.73,-1)};
		\draw[very thick,dashed] plot [smooth] coordinates{(0,2)(1,1)(1.73,-1)};
		\draw[very thick] plot [smooth] coordinates{(0,2)(0.5,-0.5)(-1.73,-1)};
		\node at (0,2) [draw,fill,circle,scale=0.4,label=above:1]{};
		\node at (-1.73,-1) [draw,fill,circle,scale=0.4,label=south west:2]{};
		\node at (1.73,-1) [draw,fill,circle,scale=0.4,label=south east:3]{};
	\node at (0,2.5){};
	\node at (0,-2.5){};
\end{tikzpicture}
& {\scriptsize $1\sqcup\begin{matrix}2\\3\end{matrix}\sqcup3[1]$} \\
\hline
\begin{tikzpicture}[scale=0.35,baseline=0]
	\draw (0,0) circle[radius=2cm];
	\draw (-.1,-.1) to (.1,.1);
	\draw (-.1,.1) to (.1,-.1);
		\draw[very thick,dashed] plot [smooth] coordinates{(0,2)(0.5,-0.5)(-1.73,-1)};
		\draw[very thick] plot [smooth] coordinates{(-1.73,-1)(0,-1.5)(1.73,-1)};
		\draw[very thick,dashed] plot [smooth] coordinates{(0,2)(-1.73,-1)};
		\node at (0,2) [draw,fill,circle,scale=0.4,label=above:1]{};
		\node at (-1.73,-1) [draw,fill,circle,scale=0.4,label=south west:2]{};
		\node at (1.73,-1) [draw,fill,circle,scale=0.4,label=south east:3]{};
	\node at (0,2.5){};
	\node at (0,-2.5){};
\end{tikzpicture}
& {\scriptsize $2\sqcup1[1]\sqcup\begin{matrix}2\\3\end{matrix}[1]$}

&\begin{tikzpicture}[scale=0.35,baseline=0]
	\draw (0,0) circle[radius=2cm];
	\draw (-.1,-.1) to (.1,.1);
	\draw (-.1,.1) to (.1,-.1);
		\draw[very thick] plot [smooth] coordinates{(0,2)(0.5,0)(0,-0.5)(-0.5,0)(0,2)};
		\draw[very thick,dashed] plot [smooth] coordinates{(0,2)(1,1)(1,-0.5)(-1.73,-1)};
		\draw[very thick] plot [smooth] coordinates{(1.73,-1)(0,-1.5)(-1.73,-1)};
		\node at (0,2) [draw,fill,circle,scale=0.4,label=above:1]{};
		\node at (-1.73,-1) [draw,fill,circle,scale=0.4,label=south west:2]{};
		\node at (1.73,-1) [draw,fill,circle,scale=0.4,label=south east:3]{};
	\node at (0,2.5){};
	\node at (0,-2.5){};
\end{tikzpicture}
& {\scriptsize $2\sqcup\begin{matrix}1\\2\\3\end{matrix}\sqcup\begin{matrix}2\\3\end{matrix}[1]$} &
\begin{tikzpicture}[scale=0.35,baseline=0]
	\draw (0,0) circle[radius=2cm];
	\draw (-.1,-.1) to (.1,.1);
	\draw (-.1,.1) to (.1,-.1);
		\draw[very thick,dashed] plot [smooth] coordinates{(0,2)(0.5,0)(0,-0.5)(-0.5,0)(0,2)};
		\draw[very thick] plot [smooth] coordinates{(0,2)(-1.5,-0.5)(1.73,-1)};
		\draw[very thick,dashed] plot [smooth] coordinates{(1.73,-1)(0,-1.5)(-1.73,-1)};
		\node at (0,2) [draw,fill,circle,scale=0.4,label=above:1]{};
		\node at (-1.73,-1) [draw,fill,circle,scale=0.4,label=south west:2]{};
		\node at (1.73,-1) [draw,fill,circle,scale=0.4,label=south east:3]{};
	\node at (0,2.5){};
	\node at (0,-2.5){};
\end{tikzpicture}
& {\scriptsize $\begin{matrix}1\\2\end{matrix}\sqcup2[1]\sqcup\begin{matrix}1\\2\\3\end{matrix}[1]$}\\

\hline

\begin{tikzpicture}[scale=0.35,baseline=0]
	\draw (0,0) circle[radius=2cm];
	\draw (-.1,-.1) to (.1,.1);
	\draw (-.1,.1) to (.1,-.1);
		\draw[very thick] plot [smooth] coordinates{(0,2)(-1,-0.5)(1.73,-1)};
		\draw[very thick] plot [smooth] coordinates{(0,2)(1.73,-1)};
		\draw[very thick,dashed] plot [smooth] coordinates{(1.73,-1)(0,-1.5)(-1.73,-1)};
		\node at (0,2) [draw,fill,circle,scale=0.4,label=above:1]{};
		\node at (-1.73,-1) [draw,fill,circle,scale=0.4,label=south west:2]{};
		\node at (1.73,-1) [draw,fill,circle,scale=0.4,label=south east:3]{};
	\node at (0,2.5){};
	\node at (0,-2.5){};
\end{tikzpicture}
& {\scriptsize $3\sqcup\begin{matrix}1\\2\end{matrix}\sqcup2[1]$} &
\begin{tikzpicture}[scale=0.35,baseline=0]
	\draw (0,0) circle[radius=2cm];
	\draw (-.1,-.1) to (.1,.1);
	\draw (-.1,.1) to (.1,-.1);
		\draw[very thick] plot [smooth] coordinates{(0,2)(-1,1)(-1.73,-1)};
		\draw[very thick,dashed] plot [smooth] coordinates{(0,2)(1.73,-1)};
		\draw[very thick,dashed] plot [smooth] coordinates{(1.73,-1)(-0.5,-0.5)(0,2)};
		\node at (0,2) [draw,fill,circle,scale=0.4,label=above:1]{};
		\node at (-1.73,-1) [draw,fill,circle,scale=0.4,label=south west:2]{};
		\node at (1.73,-1) [draw,fill,circle,scale=0.4,label=south east:3]{};
	\node at (0,2.5){};
	\node at (0,-2.5){};
\end{tikzpicture}
& {\scriptsize$1\sqcup3[1]\sqcup\begin{matrix}1\\2\end{matrix}[1]$}& \multicolumn{2}{c|}{}\\
\hline

\end{tabular}
\caption{The maximal arc patterns (MAP) and 2-simple minded collections (2-smc) for $KA_3 \cong \Lambda(3,3,2,1)$. Recall that green arcs are drawn as solid and red arcs are drawn as dashed.}\label{table1}
\end{table}
\end{center}}


\begin{ex}
The maximal arc patterns and corresponding 2-simple minded collections for $\Lambda(4,3,3,3,3)$, which is cluster tilted of type $D_4$, are shown in Table~\ref{table2}.
\end{ex}


{\footnotesize
\begin{center}
\begin{table}
\begin{tabular}{|c|c|c||c|c|c|}
\hline
MAP & 2-smc & Perms & MAP & 2-smc & Perms\\
\hline
\begin{tikzpicture}[scale=0.45,baseline=0]
	\draw (0,0) circle[radius=2cm];
	\draw (-.1,-.1) to (.1,.1);
	\draw (-.1,.1) to (.1,-.1);
		\draw[very thick] plot [smooth] coordinates{(-1.42,1.42)(-1.42,-1.42)};
		\draw[very thick] plot [smooth] coordinates{(-1.42,-1.42)(1.42,-1.42)};
		\draw[very thick] plot [smooth] coordinates{(1.42,-1.42)(1.42,1.42)};
		\draw[very thick] plot [smooth] coordinates{(1.42,1.42)(-1.42,1.42)};
		\node at (-1.42,1.42) [draw,fill,circle,scale=0.4,label=above:1]{};
		\node at (-1.42,-1.42) [draw,fill,circle,scale=0.4,label=south:2]{};
		\node at (1.42,-1.42) [draw,fill,circle,scale=0.4,label=south:3]{};
		\node at (1.42,1.42) [draw,fill,circle,scale=0.4,label=above:4]{};
	\node at (0,2.5){};
	\node at (0,-2.5){};
\end{tikzpicture}
& {\scriptsize $1\sqcup2\sqcup3\sqcup4$} & 1 &
\begin{tikzpicture}[scale=0.45,baseline=0]
	\draw (0,0) circle[radius=2cm];
	\draw (-.1,-.1) to (.1,.1);
	\draw (-.1,.1) to (.1,-.1);
		\draw[very thick,dashed] plot [smooth] coordinates{(-1.42,1.42)(-1.42,-1.42)};
		\draw[very thick,dashed] plot [smooth] coordinates{(-1.42,-1.42)(1.42,-1.42)};
		\draw[very thick,dashed] plot [smooth] coordinates{(1.42,-1.42)(1.42,1.42)};
		\draw[very thick,dashed] plot [smooth] coordinates{(1.42,1.42)(-1.42,1.42)};
		\node at (-1.42,1.42) [draw,fill,circle,scale=0.4,label=above:1]{};
		\node at (-1.42,-1.42) [draw,fill,circle,scale=0.4,label=south:2]{};
		\node at (1.42,-1.42) [draw,fill,circle,scale=0.4,label=south:3]{};
		\node at (1.42,1.42) [draw,fill,circle,scale=0.4,label=above:4]{};
	\node at (0,2.5){};
	\node at (0,-2.5){};
\end{tikzpicture}
& {\scriptsize $1[1]\sqcup2[1]\sqcup3[1]\sqcup4[1]$} & 1\\

\hline

\begin{tikzpicture}[scale=0.45,baseline=0]
	\draw (0,0) circle[radius=2cm];
	\draw (-.1,-.1) to (.1,.1);
	\draw (-.1,.1) to (.1,-.1);
	\begin{scope}
		\draw[very thick] plot [smooth] coordinates{(-1.42,1.42)(-1.42,-1.42)};
		\draw[very thick] plot [smooth] coordinates{(-1.42,-1.42)(1.42,-1.42)};
		\draw[very thick] plot [smooth] coordinates{(1.42,-1.42)(1.42,1.42)};
		\draw[very thick,dashed] plot [smooth] coordinates{(-1.42,-1.42)(0.5,0)(-1.42,1.42)};
	\end{scope}
		\node at (-1.42,1.42) [draw,fill,circle,scale=0.4,label=above:1]{};
		\node at (-1.42,-1.42) [draw,fill,circle,scale=0.4,label=south:2]{};
		\node at (1.42,-1.42) [draw,fill,circle,scale=0.4,label=south:3]{};
		\node at (1.42,1.42) [draw,fill,circle,scale=0.4,label=above:4]{};
	\node at (0,2.5){};
	\node at (0,-2.5){};
\end{tikzpicture}
& {\scriptsize $1\sqcup2\sqcup3\sqcup\begin{matrix}2\\3\\4\end{matrix}[1]$} & 4 &
\begin{tikzpicture}[scale=0.45,baseline=0]
	\draw (0,0) circle[radius=2cm];
	\draw (-.1,-.1) to (.1,.1);
	\draw (-.1,.1) to (.1,-.1);
	\begin{scope}
		\draw[very thick,dashed] plot [smooth] coordinates{(-1.42,1.42)(-1.42,-1.42)};
		\draw[very thick,dashed] plot [smooth] coordinates{(-1.42,-1.42)(1.42,-1.42)};
		\draw[very thick,dashed] plot [smooth] coordinates{(1.42,-1.42)(1.42,1.42)};
		\draw[very thick] plot [smooth] coordinates{(1.42,1.42)(-0.5,0)(1.42,-1.42)};
	\end{scope}
		\node at (-1.42,1.42) [draw,fill,circle,scale=0.4,label=above:1]{};
		\node at (-1.42,-1.42) [draw,fill,circle,scale=0.4,label=south:2]{};
		\node at (1.42,-1.42) [draw,fill,circle,scale=0.4,label=south:3]{};
		\node at (1.42,1.42) [draw,fill,circle,scale=0.4,label=above:4]{};
	\node at (0,2.5){};
	\node at (0,-2.5){};
\end{tikzpicture}
& {\scriptsize $\begin{matrix}4\\1\\2\end{matrix}\sqcup1[1]\sqcup2[1]\sqcup3[1]\sqcup$} & 4\\

\hline

\begin{tikzpicture}[scale=0.45,baseline=0]
	\draw (0,0) circle[radius=2cm];
	\draw (-.1,-.1) to (.1,.1);
	\draw (-.1,.1) to (.1,-.1);
		\draw[very thick] plot [smooth] coordinates{(-1.42,1.42)(-1.42,-1.42)};
		\draw[very thick] plot [smooth] coordinates{(-1.42,-1.42)(1.42,-1.42)};
		\draw[very thick] plot [smooth] coordinates{(1.42,-1.42)(0.3,0.3)(-1.42,1.42)};
		\draw[very thick,dashed] plot [smooth] coordinates{(1.42,1.42)(-1.42,1.42)};
		\node at (-1.42,1.42) [draw,fill,circle,scale=0.4,label=above:1]{};
		\node at (-1.42,-1.42) [draw,fill,circle,scale=0.4,label=south:2]{};
		\node at (1.42,-1.42) [draw,fill,circle,scale=0.4,label=south:3]{};
		\node at (1.42,1.42) [draw,fill,circle,scale=0.4,label=above:4]{};
	\node at (0,2.5){};
	\node at (0,-2.5){};
\end{tikzpicture}
& {\scriptsize $1\sqcup2\sqcup\begin{matrix}3\\4\end{matrix}\sqcup4[1]$} & 4 &
\begin{tikzpicture}[scale=0.45,baseline=0]
	\draw (0,0) circle[radius=2cm];
	\draw (-.1,-.1) to (.1,.1);
	\draw (-.1,.1) to (.1,-.1);
		\draw[very thick,dashed] plot [smooth] coordinates{(-1.42,1.42)(-1.42,-1.42)};
		\draw[very thick,dashed] plot [smooth] coordinates{(-1.42,-1.42)(1.42,-1.42)};
		\draw[very thick,dashed] plot [smooth] coordinates{(1.42,-1.42)(0.3,0.3)(-1.42,1.42)};
		\draw[very thick] plot [smooth] coordinates{(1.42,-1.42)(1.42,1.42)};
		\node at (-1.42,1.42) [draw,fill,circle,scale=0.4,label=above:1]{};
		\node at (-1.42,-1.42) [draw,fill,circle,scale=0.4,label=south:2]{};
		\node at (1.42,-1.42) [draw,fill,circle,scale=0.4,label=south:3]{};
		\node at (1.42,1.42) [draw,fill,circle,scale=0.4,label=above:4]{};
	\node at (0,2.5){};
	\node at (0,-2.5){};
\end{tikzpicture}
& {\scriptsize $3\sqcup1[1]\sqcup2[1]\sqcup\begin{matrix}3\\4\end{matrix}[1]$} & 4\\
\hline

\begin{tikzpicture}[scale=0.45,baseline=0]
	\draw (0,0) circle[radius=2cm];
	\draw (-.1,-.1) to (.1,.1);
	\draw (-.1,.1) to (.1,-.1);
		\draw[very thick] plot [smooth] coordinates{(-1.42,1.42)(-1.42,-1.42)};
		\draw[very thick] plot [smooth] coordinates{(-1.42,-1.42)(0.5,0)(-1.42,1.42)};
		\draw[very thick] plot [smooth] coordinates{(1.42,-1.42)(1.42,1.42)};
		\draw[very thick,dashed] plot [smooth] coordinates{(1.42,-1.42)(0.3,1)(-1.42,1.42)};
		\node at (-1.42,1.42) [draw,fill,circle,scale=0.4,label=above:1]{};
		\node at (-1.42,-1.42) [draw,fill,circle,scale=0.4,label=south:2]{};
		\node at (1.42,-1.42) [draw,fill,circle,scale=0.4,label=south:3]{};
		\node at (1.42,1.42) [draw,fill,circle,scale=0.4,label=above:4]{};
	\node at (0,2.5){};
	\node at (0,-2.5){};
\end{tikzpicture}
& {\scriptsize $1\sqcup3\sqcup\begin{matrix}2\\3\\4\end{matrix}\sqcup\begin{matrix}3\\4\end{matrix}[1]$} & 4 &
\begin{tikzpicture}[scale=0.45,baseline=0]
	\draw (0,0) circle[radius=2cm];
	\draw (-.1,-.1) to (.1,.1);
	\draw (-.1,.1) to (.1,-.1);
		\draw[very thick,dashed] plot [smooth] coordinates{(-1.42,1.42)(-1.42,-1.42)};
		\draw[very thick,dashed] plot [smooth] coordinates{(-1.42,-1.42)(0.5,0)(-1.42,1.42)};
		\draw[very thick,dashed] plot [smooth] coordinates{(1.42,-1.42)(1.42,1.42)};
		\draw[very thick] plot [smooth] coordinates{(1.42,1.42)(0.3,-1)(-1.42,-1.42)};
		\node at (-1.42,1.42) [draw,fill,circle,scale=0.4,label=above:1]{};
		\node at (-1.42,-1.42) [draw,fill,circle,scale=0.4,label=south:2]{};
		\node at (1.42,-1.42) [draw,fill,circle,scale=0.4,label=south:3]{};
		\node at (1.42,1.42) [draw,fill,circle,scale=0.4,label=above:4]{};
	\node at (0,2.5){};
	\node at (0,-2.5){};
\end{tikzpicture}
& {\scriptsize $\begin{matrix}2\\3\end{matrix}\sqcup1[1]\sqcup3[1]\sqcup\begin{matrix}2\\3\\4\end{matrix}[1]$}& 4\\

\hline

\begin{tikzpicture}[scale=0.45,baseline=0]
	\draw (0,0) circle[radius=2cm];
	\draw (-.1,-.1) to (.1,.1);
	\draw (-.1,.1) to (.1,-.1);
	\begin{scope}
		\draw[very thick] plot [smooth] coordinates{(-1.42,1.42)(-1.42,-1.42)};
		\draw[very thick] plot [smooth] coordinates{(-1.42,-1.42)(.5,0)(-1.42,1.42)};
		\draw[very thick,dashed] plot [smooth] coordinates{(1.42,-1.42)(1.42,1.42)};
		\draw[very thick,dashed] plot [smooth] coordinates{(1.42,1.42)(-1.42,1.42)};
	\end{scope}
		\node at (-1.42,1.42) [draw,fill,circle,scale=0.4,label=above:1]{};
		\node at (-1.42,-1.42) [draw,fill,circle,scale=0.4,label=south:2]{};
		\node at (1.42,-1.42) [draw,fill,circle,scale=0.4,label=south:3]{};
		\node at (1.42,1.42) [draw,fill,circle,scale=0.4,label=above:4]{};
	\node at (0,2.5){};
	\node at (0,-2.5){};
\end{tikzpicture}
& {\scriptsize $1\sqcup\begin{matrix}2\\3\\4\end{matrix}\sqcup 3[1]\sqcup 4[1]$} & 4 &
\begin{tikzpicture}[scale=0.45,baseline=0]
	\draw (0,0) circle[radius=2cm];
	\draw (-.1,-.1) to (.1,.1);
	\draw (-.1,.1) to (.1,-.1);
	\begin{scope}
		\draw[very thick,dashed] plot [smooth] coordinates{(-1.42,1.42)(-1.42,-1.42)};
		\draw[very thick,dashed] plot [smooth] coordinates{(-1.42,-1.42)(.5,0)(-1.42,1.42)};
		\draw[very thick] plot [smooth] coordinates{(1.42,-1.42)(1.42,1.42)};
		\draw[very thick] plot [smooth] coordinates{(1.42,-1.42)(-1.42,-1.42)};
	\end{scope}
		\node at (-1.42,1.42) [draw,fill,circle,scale=0.4,label=above:1]{};
		\node at (-1.42,-1.42) [draw,fill,circle,scale=0.4,label=south:2]{};
		\node at (1.42,-1.42) [draw,fill,circle,scale=0.4,label=south:3]{};
		\node at (1.42,1.42) [draw,fill,circle,scale=0.4,label=above:4]{};
	\node at (0,2.5){};
	\node at (0,-2.5){};
\end{tikzpicture}
& {\scriptsize $2\sqcup 3\sqcup1[1]\sqcup\begin{matrix}2\\3\\4\end{matrix}[1]$} & 4\\

\hline

\begin{tikzpicture}[scale=0.45,baseline=0]
	\draw (0,0) circle[radius=2cm];
	\draw (-.1,-.1) to (.1,.1);
	\draw (-.1,.1) to (.1,-.1);
		\draw[very thick] plot [smooth] coordinates{(-1.42,1.42)(-1.42,-1.42)};
		\draw[very thick] plot [smooth] coordinates{(1.42,1.42)(-.5,0)(1.42,-1.42)};
		\draw[very thick,dashed] plot [smooth] coordinates{(1.42,-1.42)(1.42,1.42)};
		\draw[very thick,dashed] plot [smooth] coordinates{(-1.42,1.42)(-.5,-1)(1.42,-1.42)};
		\node at (-1.42,1.42) [draw,fill,circle,scale=0.4,label=above:1]{};
		\node at (-1.42,-1.42) [draw,fill,circle,scale=0.4,label=south:2]{};
		\node at (1.42,-1.42) [draw,fill,circle,scale=0.4,label=south:3]{};
		\node at (1.42,1.42) [draw,fill,circle,scale=0.4,label=above:4]{};
	\node at (0,2.5){};
	\node at (0,-2.5){};
\end{tikzpicture}
& {\scriptsize $1\sqcup\begin{matrix}4\\1\\2\end{matrix}\sqcup3[1]\sqcup\begin{matrix}1\\2\end{matrix}[1]$} & 4 &
\begin{tikzpicture}[scale=0.45,baseline=0]
	\draw (0,0) circle[radius=2cm];
	\draw (-.1,-.1) to (.1,.1);
	\draw (-.1,.1) to (.1,-.1);
		\draw[very thick,dashed] plot [smooth] coordinates{(-1.42,1.42)(-1.42,-1.42)};
		\draw[very thick,dashed] plot [smooth] coordinates{(1.42,1.42)(-.5,0)(1.42,-1.42)};
		\draw[very thick] plot [smooth] coordinates{(1.42,-1.42)(1.42,1.42)};
		\draw[very thick] plot [smooth] coordinates{(-1.42,-1.42)(-.5,1)(1.42,1.42)};
		\node at (-1.42,1.42) [draw,fill,circle,scale=0.4,label=above:1]{};
		\node at (-1.42,-1.42) [draw,fill,circle,scale=0.4,label=south:2]{};
		\node at (1.42,-1.42) [draw,fill,circle,scale=0.4,label=south:3]{};
		\node at (1.42,1.42) [draw,fill,circle,scale=0.4,label=above:4]{};
	\node at (0,2.5){};
	\node at (0,-2.5){};
\end{tikzpicture}
& {\scriptsize $3\sqcup\begin{matrix}4\\1\end{matrix}\sqcup1[1]\sqcup\begin{matrix}4\\1\\2\end{matrix}[1]$} & 4\\
\hline

\begin{tikzpicture}[scale=0.45,baseline=0]
	\draw (0,0) circle[radius=2cm];
	\draw (-.1,-.1) to (.1,.1);
	\draw (-.1,.1) to (.1,-.1);
		\draw[very thick] plot [smooth] coordinates{(-1.42,1.42)(-1.42,-1.42)};
		\draw[very thick] plot [smooth] coordinates{(1.42,-1.42)(1.42,1.42)};
		\draw[very thick,dashed] plot [smooth] coordinates{(1.42,-1.42)(.3,.3)(-1.42,1.42)};
		\draw[very thick,dashed] plot [smooth] coordinates{(1.42,-1.42)(-.3,-.3)(-1.42,1.42)};
		\node at (-1.42,1.42) [draw,fill,circle,scale=0.4,label=above:1]{};
		\node at (-1.42,-1.42) [draw,fill,circle,scale=0.4,label=south:2]{};
		\node at (1.42,-1.42) [draw,fill,circle,scale=0.4,label=south:3]{};
		\node at (1.42,1.42) [draw,fill,circle,scale=0.4,label=above:4]{};
	\node at (0,2.5){};
	\node at (0,-2.5){};
\end{tikzpicture}
& {\scriptsize $1\sqcup3\sqcup\begin{matrix}1\\2\end{matrix}[1]\sqcup\begin{matrix}3\\4\end{matrix}[1]$} & 2 &
\begin{tikzpicture}[scale=0.45,baseline=0]
	\draw (0,0) circle[radius=2cm];
	\draw (-.1,-.1) to (.1,.1);
	\draw (-.1,.1) to (.1,-.1);
		\draw[very thick,dashed] plot [smooth] coordinates{(-1.42,1.42)(-1.42,-1.42)};
		\draw[very thick,dashed] plot [smooth] coordinates{(1.42,-1.42)(1.42,1.42)};
		\draw[very thick] plot [smooth] coordinates{(1.42,1.42)(.3,-.3)(-1.42,-1.42)};
		\draw[very thick] plot [smooth] coordinates{(1.42,1.42)(-.3,.3)(-1.42,-1.42)};
		\node at (-1.42,1.42) [draw,fill,circle,scale=0.4,label=above:1]{};
		\node at (-1.42,-1.42) [draw,fill,circle,scale=0.4,label=south:2]{};
		\node at (1.42,-1.42) [draw,fill,circle,scale=0.4,label=south:3]{};
		\node at (1.42,1.42) [draw,fill,circle,scale=0.4,label=above:4]{};
	\node at (0,2.5){};
	\node at (0,-2.5){};
\end{tikzpicture}
& {\scriptsize $\begin{matrix}2\\3\end{matrix}\sqcup\begin{matrix}4\\1\end{matrix}\sqcup1[1]\sqcup3[1]$} & 2 \\
\hline

\begin{tikzpicture}[scale=0.45,baseline=0]
	\draw (0,0) circle[radius=2cm];
	\draw (-.1,-.1) to (.1,.1);
	\draw (-.1,.1) to (.1,-.1);
		\draw[very thick] plot [smooth] coordinates{(-1.42,1.42)(-1.42,-1.42)};
		\draw[very thick,dashed] plot [smooth] coordinates{(1.42,1.42)(-1.42,1.42)};
		\draw[very thick] plot [smooth] coordinates{(1.42,-1.42)(.3,.3)(-1.42,1.42)};
		\draw[very thick,dashed] plot [smooth] coordinates{(1.42,-1.42)(-.3,-.3)(-1.42,1.42)};
		\node at (-1.42,1.42) [draw,fill,circle,scale=0.4,label=above:1]{};
		\node at (-1.42,-1.42) [draw,fill,circle,scale=0.4,label=south:2]{};
		\node at (1.42,-1.42) [draw,fill,circle,scale=0.4,label=south:3]{};
		\node at (1.42,1.42) [draw,fill,circle,scale=0.4,label=above:4]{};s
	\node at (0,2.5){};
	\node at (0,-2.5){};
\end{tikzpicture}
& {\scriptsize $1\sqcup\begin{matrix}3\\4\end{matrix}\sqcup4[1]\sqcup\begin{matrix}1\\2\end{matrix}[1]$} & 4 & \multicolumn{3}{c|}{}\\
\hline
\end{tabular}
\caption{The maximal arc patterns (MAP) and 2-simple minded collections (2-smc) for $\Lambda(4,3,3,3,3)$, which is cluster tilted algebra of type $D_4$. Each MAP and 2-smc is shown up to cyclic permutation of the vertices. Recall that green arcs are drawn as solid and red arcs are drawn as dashed.}\label{table2}
\end{table}
\end{center}}

\section*{Acknowledgements}
A portion of this work is included in the first author's PhD thesis. The first author is thankful to Emily Barnard and Job Rock for meaningful conversations and suggestions. The second author is thankful to Aslak Buan and Bethany Marsh for explaining their results to him and sharing their preprints, and in particular for the insights provided in Aslak Buan's lecture at Tsinghua University. Both authors would like to thank Corey Bregman for meaningful conversations and Gordana Todorov for support and suggestions. They would also like to extend their gratitude to an anonymous referee for their thorough review and suggestions for improving this paper.

\bibliographystyle{amsplain}
\bibliography{HansonIgusa_tCMC_final}

\end{document}